\documentclass[10pt]{article}
\usepackage{amsmath,amsfonts,amssymb,amsthm,mathrsfs}
\usepackage[a4paper,vmargin={3.5cm,3.5cm},hmargin={2.5cm,2.5cm}]{geometry}
\usepackage[font=sf, labelfont={sf,bf}, margin=1cm]{caption}
\usepackage{graphicx,graphics}
\usepackage{latexsym}
\usepackage[applemac]{inputenc}
\usepackage{ae,aecompl}
\usepackage[english]{babel}
 \usepackage[colorlinks=true]{hyperref}
 \usepackage{bbm}
\usepackage{enumerate}
\newcommand{\E}{\mathbb{E}}
\def\e{{\mathcal{E}}}
\def\f{\mathcal{F}}
\def\a{\mathcal{A}}
\def\b{\mathcal{B}}
\def\g{{\mathcal{G}}}
\def\u{{\mathcal{U}}}
\def\l{{\mathcal{L}}}
\def\ve{\varepsilon}
\def\la{{\longrightarrow}}
\def\h{{\mathcal{H}}}
\def\wt{\widetilde}
\def\tc{{\mathscr{T}}}
\def\t{{\mathcal{T}}}
\def\ov{\overline}

\def\T{{\mathbb{T}}}
\def\P{{\mathbb{P}}}
\def\Z{{\mathbb{Z}}}
\def\F{{\mathbb{F}}}
\def\C{{\mathbb{C}}}
\def\R{{\mathbb{R}}}
\def\bm{\mathbf{m}}
\def\bp{\mathbf{p}}

\def\build#1_#2^#3{\mathrel{
\mathop{\kern 0pt#1}\limits_{#2}^{#3}}}

\newtheorem{theorem}{Theorem}[]

\newtheorem{remark}{Remark}[]
\newtheorem{definition}{Definition}[]
\newtheorem{proposition}[]{Proposition}
\newtheorem{lemma}[proposition]{Lemma}
\newtheorem{corollary}[proposition]{Corollary}

\title{\textsc{First-passage percolation and local modifications\\ of distances in random triangulations}\footnote{Supported in part
by the grant ANR-14-CE25-0014 and by the Newton Institute for Mathematical Sciences}}
\date{}
\begin{document}

\author{Nicolas Curien\footnote{Université Paris-Sud, nicolas.curien@gmail.com}~ and Jean-Fran\c cois Le Gall\footnote{Université Paris-Sud, jean-francois.legall@math.u-psud.fr}}
\maketitle

\begin{abstract} We study local modifications of the graph distance in large random triangulations. Our main results
show that, in large scales, the modified distance behaves like a deterministic constant $\mathbf{c}~\in~(0,\infty)$ times the usual
graph distance. This applies in particular to the first-passage percolation distance obtained by
assigning independent random weights to the edges of the graph. We also consider the graph distance on the dual map, and the first-passage percolation
on the dual map with exponential edge weights, which is closely related to the so-called Eden model. In the latter two cases, we are able to compute explicitly
the constant $\mathbf{c}$ by using earlier results about asymptotics for the peeling process. In general however, the constant $\mathbf{c}$ is obtained from a subadditivity argument 
in the infinite half-plane model that describes the 
asymptotic shape of the triangulation near the boundary of a large ball.
Our results apply in particular to the infinite random triangulation known as the UIPT, and show that balls of the 
UIPT for the modified distance are asymptotically close to balls for the graph distance.
 \end{abstract}
 
 \section{Introduction}
 
 In the recent years, there has been much effort to understand the large-scale geometry of
 random planar maps viewed as random metric spaces for the usual graph distance on their
 vertex set. A major achievement in the area is the construction and study of the so-called
 Brownian map, which has been proved to be the universal scaling limit of many
 different classes of planar maps equipped with the graph distance (see \cite{LG11,Mie11} and
 more recently \cite{Ab13, AA13, BJM13}). An account of these developments can be
 found in the surveys \cite{LeGallICM, MieStFlour}. In the present work, we replace the graph distance
 by other natural choices of distances on the vertex set or on the set of faces, and we show 
 that, in large scales, these new distances behave like the original graph distance, up to 
 a constant multiplicative factor.  In particular, we prove that the vertex set of a 
 uniformly distributed random plane triangulation with $n$ vertices equipped with the (suitably rescaled) first-passage percolation
 distance obtained by assigning independent random lengths to the edges converges in distribution 
 to the Brownian map as
 $n\to\infty$, in the Gromov--Hausdorff sense. So, in some sense, the extra randomness coming 
 from the weights assigned to the edges plays no role in the limit. This is a new illustration of 
 the universality of the Brownian map as a two-dimensional model of random geometry. We mention here that
 the study of discrete or continuous models of random geometry has been strongly motivated by
 their relevance to various domains of theoretical physics, and in particular to the so-called
 two-dimensional quantum gravity. Discrete random geometry has been the subject 
 of intensive research in the physics literature since the beginning of the eighties when Polyakov suggested  to solve 
 questions coming from string theory and quantum gravity by developing a formalism to calculate sums over
 random surfaces, as a kind of analog of the famous Feynman path integrals. We refer to the book \cite{ADJ97}
 for an overview of the use of discrete random surfaces in theoretical physics.
 
 Let us turn to a more detailed description of our main results. We recall that a planar map is 
a proper embedding of a finite connected (multi)graph in the two-dimensional sphere, viewed up to 
orientation-preserving homeomorphisms of the sphere. We will always consider rooted planar maps, which means
that there is a distinguished oriented edge whose initial vertex is called the root vertex. The 
faces are the connected components of the complement of edges, and 
a planar map is called a triangulation if all its 
faces are triangles (possibly with two sides glued together).

\paragraph{Modified distances.}  If $ \mathrm{m}$ is a rooted planar map, we let $ \mathsf{V}( \mathrm{m}), \mathsf{E}( \mathrm{m})$ and $ \mathsf{F}( \mathrm{m})$ denote
respectively the sets of vertices, edges and faces of $ \mathrm{m}$. The set $ \mathsf{V}( \mathrm{m})$ is usually equipped
with the graph distance, which is denoted by $ \mathrm{d_{gr}}$ (or $ \mathrm{d}_{ \mathrm{gr}}^{ \mathrm{m}}$ if there is a risk of ambiguity). 
We introduce the following modifications of the graph distance.

 \begin{description}
\item[Case 0.] \textsc{First-passage (bond) percolation}. Assign independent identically distributed positive random variables  $ w(e)$  to all edges $e \in \mathsf{E}( \mathrm{m})$.
We assume that the common distribution of the  ``weights'' $w(e)$ is supported on
$[\kappa,1]$ for some $\kappa\in(0,1]$. The associated first-passage percolation distance is defined on $  \mathsf{V}( \mathrm{m})$ by setting for any $x, y \in \mathsf{V}( \mathrm{m})$,
$$ \mathrm{d}_{\rm fpp}(x,y) = \inf_{ \gamma : x \to y}\  \sum_{e \in \gamma} w(e),$$ 
where the infimum runs over all paths $\gamma$ going from $x$ to $y$ in the map $ \mathrm{m}$.
\item[Case 1.] \textsc{Dual graph distance}. Consider the dual map $ \mathrm{m}^\dagger$, whose vertices are the faces of $ \mathrm{m}$, and each edge $e$ of $ \mathrm{m}$ corresponds to an edge of $ \mathrm{m}^\dagger$ connecting the two (possibly equal)
faces incident to $e$. We may then consider the graph distance on $\mathsf{V}( \mathrm{m}^\dagger)=\mathsf{F}( \mathrm{m})$, which we denote by $\mathrm{d}^\dagger_{\rm gr}$.
\item[Case 2.] \textsc{Eden model.} This is the first-passage percolation model on $ \mathrm{m}^\dagger$ corresponding to exponential edge weights. More precisely, we assign independent exponential random variables
with parameter $1$ to the edges of $ \mathrm{m}^\dagger$ (or equivalently to the edges of $ \mathrm{m}$) and the associated first-passage percolation distance on $\mathsf{F}( \mathrm{m})=\mathsf{V}( \mathrm{m}^\dagger)$ is denoted by 
$\mathrm{d}^\dagger_{\rm Eden}$. We call this the Eden model because of the close relation with the classical Eden growth model, see 
in particular \cite[Section 6]{AB14} or \cite[Section 1.2.1]{MS13}, and also \cite[Proposition 15]{CLGpeeling}. 
\end{description}
The functions $ \mathrm{d_{gr}}$ and $ \mathrm{d_{fpp}}$ are distances on $ \mathsf{V}( \mathrm{m})$ whereas $ \mathrm{d}_{ \mathrm{gr}}^\dagger$ and $ \mathrm{d}_{ \mathrm{Eden}}^\dagger$ are distances on $ \mathsf{F}( \mathrm{m})$. To compare the latter distances to the usual graph metric on $ \mathrm{m}$, we will replace faces by incident vertices, and we use the notation $x\triangleleft f$ to mean that the vertex $x$ is incident to the face $f$.

\paragraph{Finite triangulations.} We consider these ``modified distances'' when $ \mathrm{m}=\t_n$ is a 
random planar map chosen uniformly at random
in the set of all rooted plane triangulations with $n+1$ vertices (we consider type I triangulations where loops and multiple edges are allowed). 
 In each of the previous cases, we are able to prove that the modified distances behave in large scales like a deterministic constant times the graph distance on $\mathsf{V}(\t_n)$. More precisely, there exist constants $  \mathbf{c}_{0}, \mathbf{c}_{1}$ and $ \mathbf{c}_{2}$ in $(0,\infty)$ such that we have the following three convergences in probability
\begin{eqnarray}
\label{intro-1}
n^{-1/4} \sup_{ \begin{subarray}{c}x, y \in \mathsf{V}(\t_{n})\\ {\color{white} x\triangleleft f \mathrm{\ and \ }y\triangleleft g} \end{subarray}} \big| \mathrm{d_{fpp}}(x,y) - \mathbf{c}_{0}\cdot \mathrm{d_{gr}}(x,y)\big| &\xrightarrow[n\to\infty]{}&0,\\
\label{intro-12} n^{-1/4} \sup_{ \begin{subarray}{c} x, y \in \mathsf{V}(\t_{n}), \ 
f,g\in \mathsf{F}(\t_n)\\
x\triangleleft f \mathrm{\ and \ }y\triangleleft g
\end{subarray}}
 \big| \mathrm{d^\dagger_{gr}}(f,g) - \mathbf{c}_1\cdot \mathrm{d_{gr}}(x,y)\big| &\xrightarrow[n\to\infty]{}&0,\\
\label{intro-13}n^{-1/4} \sup_{ \begin{subarray}{c} x, y \in \mathsf{V}(\t_{n}), \ 
f,g\in \mathsf{F}(\t_n)\\
x\triangleleft f \mathrm{\ and \ }y\triangleleft g
\end{subarray}}
 \big| \mathrm{d^\dagger_{Eden}}(f,g) - \mathbf{c}_2\cdot \mathrm{d_{gr}}(x,y)\big| &\xrightarrow[n\to\infty]{}&0.
 \end{eqnarray}

Since the convergence of rescaled triangulations to the Brownian map \cite{LG11} implies that the typical graph distance  
between two vertices of $\t_n$ is of order $n^{1/4}$, the convergence \eqref{intro-1} shows that in large scales $\mathrm{d_{fpp}}(x,y) $ is proportional to $\mathrm{d_{gr}}(x,y)$.
In fact \eqref{intro-1} implies that the set $\mathsf{V}(\t_n)$ equipped 
with the metric $n^{-1/4}\mathrm{d_{fpp}}$ converges in distribution in the Gromov--Hausdorff sense to (a scaled version of) the Brownian map, and that this convergence
takes place jointly with that of $(\mathsf{V}(\t_n),n^{-1/4}\mathrm{d_{gr}})$ proved in \cite{LG11}
(see Corollary \ref{cor:brownianmap} below). Similarly \eqref{intro-12} shows that uniform rooted trivalent maps with $n$ faces (which are the dual maps of rooted triangulations with $n$ vertices) converge after rescaling toward the Brownian map. 

In case $0$., the constant $ \mathbf{c}_{0}$ depends on the distribution of the weights and an explicit calculation of this constant seems hopeless. However in cases 1.~and 2.~(dual graph and Eden model) the constants can be computed exactly and we have $$\mathbf{c}_1=1+2\sqrt{3} \quad \mbox{and} \quad \mathbf{c}_2=2\sqrt{3}.$$
The reason why these two models are more tractable is the fact that balls for the dual graph distance or for the Eden model can be generated and studied
via an algorithmic procedure known as the peeling process, which was first discussed by Angel \cite{Ang03}, following the
work of Watabiki \cite{Wat95} in the physics literature. These peeling explorations
for balls have been studied in detail in the case of the Uniform Infinite Planar Triangulation (UIPT) in \cite{CLGpeeling}, see also \cite{AB14}
and \cite{Bud15}. Budd's results in \cite{Bud15} apply to more general classes of random planar maps, known as Boltzmann planar maps, 
and as mentioned in the introduction of \cite{Bud15} they might allow the explicit calculation of other scaling constants arising when
considering different metrics on these  random graphs.

\paragraph{The UIPT.}  We can also state versions of our results on the UIPT, which is
the random infinite lattice first discussed by Angel and Schramm \cite{AS03} (in fact, Angel and Schramm did not consider type I
 triangulations, but that case is a special instance of the constructions in Stephenson \cite{St14}). The UIPT, which will be denoted by
 $\t_\infty$, is the local limit of uniformly distributed plane triangulations with $n$ faces when $n\to\infty$. 
 We can equip the vertex set of the UIPT with the usual graph distance $\mathrm{d}_{\mathrm{gr}}$ or with a modified distance as above. 
 For simplicity, let us consider
 only the first-passage percolation distance $\mathrm{d_{fpp}}$  defined as previously from i.i.d.~edge weights (case 0.). 
 For every $r>0$, write $B_r(\t_\infty)$ for the planar map obtained by keeping only those faces of $\t_\infty$ that contain
 at least one vertex at graph distance strictly less than $r$ from the root vertex, and define
 $ B^\mathrm{fpp}_{r}(\t_\infty)$ analogously, replacing the graph distance by the first-passage percolation distance.
 Under the
 same assumptions on the weights, we prove 
 (Theorem \ref{balls-uipt}) that, for every $\ve>0$,
\begin{equation}
\label{intro-2}
\lim_{r\to\infty} \P\Bigg(\sup_{x,y\in \mathsf{V}(B_r(\t_\infty))} \big|\mathrm{d}_\mathrm{fpp}(x,y) 
-\mathbf{c}_{0}\cdot \mathrm{d}_\mathrm{gr}(x,y)\big| > \ve r\Bigg) = 0,
\end{equation}
with the same constant $ \mathbf{c}_{0}$ as above.
It follows that the inclusions
\begin{equation} \label{eq:intro-ball} B_{(1-\ve)r/\mathbf{c}_{0}}(\t_\infty) \subset B^\mathrm{fpp}_{r}(\t_\infty)
\subset B_{(1+\ve)r/\mathbf{c}_{0}}(\t_\infty) 
\end{equation}
hold with probability tending to $1$ as $r\to\infty$. In other words, large balls for the first-passage percolation
distance are close to balls for the graph distance. Similar results hold for the graph distance 
or the Eden distance on the dual map
of the UIPT (see Theorem \ref{balls-uipt-dual}). The particular case of the Eden model answers a question \cite[Question 9.14]{MS13}
raised by Miller and Sheffield, who used the Eden growth model as a  motivation for introducing the so-called Quantum Loewner Evolution of parameter $( \frac{8}{3},0)$. Notice that the value of $ \mathbf{c}_{2}$ had already been conjectured in \cite[Remark 5]{AB14}.

\paragraph{Comparison with other models of FPP.} Let us briefly discuss the connections between our results and the vast literature on (bond) first-passage
percolation on regular lattices such as $\Z^d$. The fact \eqref{eq:intro-ball} that our first-passage percolation distance grows balls that
are close to ``deterministic'' balls for the graph distance can be seen as a counterpart to the classical shape
theorem on regular lattices (see the survey papers \cite{How04,ADH15}).  On the other hand, very little is known about
the asymptotic shape of balls for first-passage percolation on $\Z^d$, and this shape 
is not expected to be the round ball since the anisotropy of $ \mathbb{Z}^d$ might persist in the limit. When the underlying lattice is the Delaunay triangulation of a standard Poisson point process on $ \mathbb{R}^2$, rotational invariance is restored and the limit shape is a round ball as proved in \cite{VW92}. However,  the dilation factor of the ball, known as the time constant, remains out of reach. 

Similarly to the case of Delaunay triangulations, our random setting of the UIPT is in a sense
more ``isotropic'' than regular lattices (even though the graph is not embedded and so the meaning of rotational invariance is unclear) which explains intuitively why balls in the modified metric grow
roughly like balls for the graph distance. It is remarkable that one can compute the values of the ``time constants'' in the 
particular cases 1.~and 2. Let us also mention that first-passage percolation on random graphs has been considered 
in other models, either in the ``dense'' case (e.g. supercritical Erd\"os-Renyi random graphs) \cite[Chapter 8]{RemcoRGII} or in tree-like graphs \cite{Stu15}. In these cases too, the resulting first-passage percolation metric 
is typically proportional to the graph distance, with an explicit multiplicative constant.  
 
\paragraph{Extensions.} 
Our techniques should extend to much more general settings. In case 0.~in particular, we expect that the assumption on the distribution of weights  can be relaxed significantly. Assuming that this distribution is supported on
 a compact subinterval of $(0,\infty)$ avoids a number of additional technicalities. Similarly, we can handle 
 more general first-passage percolation distances on the dual map. We restricted our
 attention to cases 1.~and 2.~because these cases allow explicit calculations of the values of $ \mathbf{c}_{1}$ and $ \mathbf{c}_{2}$ (which seem hopeless in more
 general situations). Our techniques are robust and could probably handle more involved modifications of the graph distance as long as these
 modifications remain ``local''. One such example suggested to us by Grégory Miermont is to consider the Riemannian metric associated with the Riemann surface structure of the map, which is obtained by gluing equilateral triangles according to the combinatorics of the map \cite{CurKPZ,GR10}.

 We also mention that similar techniques should be applicable to random \emph{quadrangulations}, and we hope to address this setting in future work. 
The case of quadrangulations is especially interesting because of Tutte's bijection between quadrangulations with $n$ faces and general planar maps with $n$ edges. 
We conjecture that our techniques can be used in combination with the results of \cite{BJM13} to verify that Tutte's bijection is asymptotically
an isometry. 

\subsection*{Ideas of proofs and methods}
We conclude this introduction with a brief discussion of our methods. An important part of the paper (Sections 2 to 4)
is devoted to geometric properties of the UIPT that are needed for the proofs of our main theorems. Section 2 
in particular discusses a decomposition of the UIPT into layers, which is closely related to the work of Krikun \cite{Kri04}.
To explain this decomposition, we define, for every integer $r\geq 1$, the hull $B^\bullet_r(\t_\infty)$
by adding to the ball $B_r(\t_\infty)$ the finite connected components of its complement. We then call layer
each set of the form $B^\bullet_r(\t_\infty)\backslash B^\bullet_{r-1}(\t_\infty)$. In each such layer we can distinguish 
special triangles, called downward triangles, which are in one-to-one correspondance with edges of the outer
boundary $\partial B^\bullet_r(\t_\infty)$ of the layer, see Fig.~\ref{fig:treesfromlayers}. It turns out that the configuration of downward triangles has a
very nice branching structure which can be described in an explicit manner in terms of 
a critical offspring distribution $\theta$ in the domain of attraction of a $\frac{3}{2}$-stable distribution. The configuration of 
downward triangles does not determine
the UIPT, but it is easy to reconstruct the UIPT given this configuration: the holes remaining when one removes the downward triangles are filled in by independent
Boltzmann triangulations with a boundary (called free triangulations in \cite{AS03}).

The layer decomposition of the UIPT suggests to introduce two half-plane models, which we call
the LHPT for lower half-plane triangulation and the UHPT for upper half-plane triangulation, and which
are discussed in Section 3.
Roughly speaking the LHPT corresponds to what one sees ``below'' the boundary of the hull
$B^\bullet_r(\t_\infty)$ when $r$ is large (see Fig.~\ref{Figu1} below). The UHPT, which was already discussed in
\cite{Ang05,ACpercopeel}, is a kind of dual model to the LHPT, and arises
as the local limit of (infinite) triangulations with a boundary when the size of the boundary tends to infinity.
The model of interest for our purposes is the LHPT: The constants $\mathbf{c}_0$, $\mathbf{c}_1$ and $\mathbf{c}_2$
arise from an application of Kingman's ergodic theorem to the (modified) distance between the root 
vertex or the root face and the horizontal line at vertical coordinate $-n$ in the LHPT (Propositions
\ref{lem:subadditive} and \ref{subadditive-dual}). 
However, certain estimates, concerning graph distances along
the boundary in these half-plane models, are easier to derive in the UHPT model and can then 
be ``transferred'' to the LHPT using the relations between the two models. These estimates, which play
a key role in the subsequent proofs, are discussed in Section 4.

\begin{figure}[!h]
 \begin{center}
 \includegraphics[width=16cm]{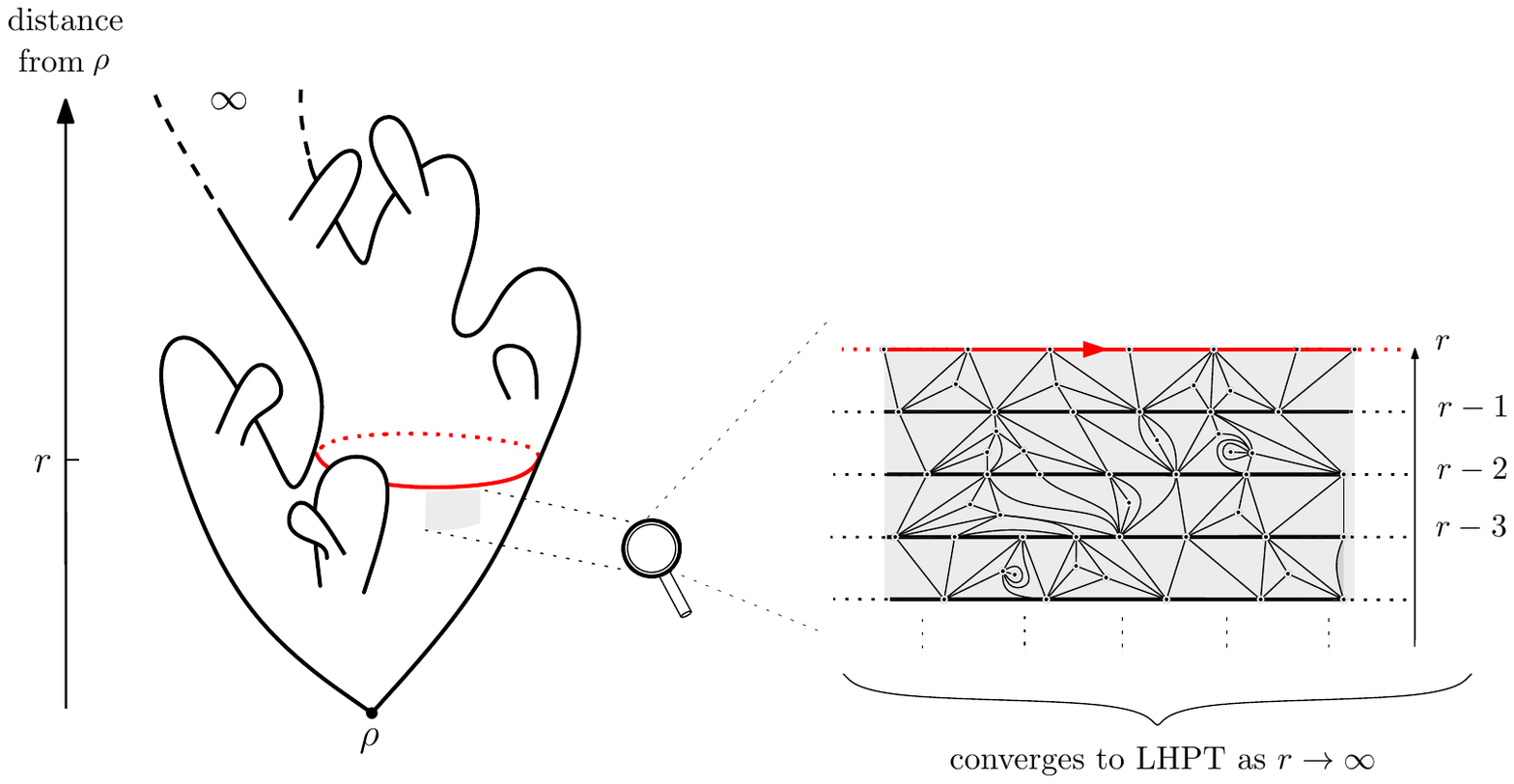}
 \caption{\label{Figu1}Illustration of the link between the UIPT and the LHPT. The latter appears as the local limit of the hull of radius $r$ of the 
 UIPT seen from the boundary of the hull.}
 \end{center}
 \vspace{-4mm}
 \end{figure}

Section 5 is the most technical part of the paper. We concentrate on the first-passage percolation distance
$\mathrm{d}_\mathrm{fpp}$ of case 0.~and explain how to apply the asymptotics known in the LHPT model
(resulting from the application of Kingman's subadditive ergodic theorem) in order to get
information on the UIPT. The key Proposition \ref{keytech} essentially shows that the $\mathrm{d}_\mathrm{fpp}$-distance
between {\em any} vertex of the boundary $\partial B^\bullet_r(\t_\infty)$ 
of the hull of radius $r$ and the boundary $\partial B^\bullet_{r-\lfloor \eta r\rfloor}(\t_\infty)$
of a smaller hull of radius $r-\lfloor \eta r\rfloor$ is close to $\mathbf{c}_0\eta r$ when $r$ is large, provided
$\eta>0$ has been fixed small enough. The proof involves a delicate comparison argument with the
LHPT model and makes use of the estimates for distances along the boundary derived in Section 4.
From Proposition \ref{keytech}, it is not too hard to verify that, with high probability, any vertex of
$\partial B^\bullet_r(\t_\infty)$ is at $ \mathrm{d_{fpp}}$-distance approximately $\mathbf{c}_0r$ from the root vertex
(Proposition \ref{dist-hull}). Our main results are then proved in Section 6 in the case 0. We use 
absolute continuity relations between the UIPT and finite triangulations to prove that the
$\mathrm{d}_\mathrm{fpp}$-distance between the root vertex of $\t_n$
and another vertex chosen uniformly at random in $\mathsf{V}(\t_n)$
is close to $\mathbf{c}_0$ times the graph distance between the same two vertices, up
to an error small in comparison with $ n^{1/4}$
(Proposition \ref{key-finitetri}). The convergence (\ref{intro-1}) follows easily. We can then 
return to the UIPT and deduce \eqref{intro-2} from \eqref{intro-1} and another application
of the absolute continuity relations between the UIPT and finite triangulations. 
Finally Section 7 explains the modifications needed to extend our results to cases 1.~and 2.~concerning the dual graph distance $\mathrm{d}^\dagger_{\rm gr}$ and the
Eden distance $\mathrm{d}^\dagger_{\rm Eden}$. As mentioned above, once the analog of \eqref{eq:intro-ball} has been proved in these cases, the values of $ \mathbf{c}_{1}$ and $ \mathbf{c}_{2}$ can be deduced from our previous work \cite{CLGpeeling} on the peeling process on random triangulations.\medskip

\noindent \textbf{Acknowledgments:} We thank Timothy Budd and Grégory Miermont for stimulating discussions.

 \section{Skeleton decomposition of random triangulations} \label{sec:skeletonmain}
In this section we present the skeleton decomposition of (type I) triangulations. This is closely related to the work of Krikun \cite{Kri04,Kri05}. Using explicit enumeration formulas found in \cite{Kri07} we are able to give simple expressions for the law of the skeleton of finite random triangulations.

We focus on the case of type I triangulations, where both loops and
multiple edges are allowed. Many of our arguments can be adapted to the case of type II triangulations, but we
do not discuss this here. As usual, all the planar maps that are considered in this work are rooted, i.e.~given with a distinguished oriented edge called the root edge, whose initial vertex is called the root vertex. A triangulation  $ \mathrm{t}$ with a boundary is a rooted planar map such that all faces are triangles, except for the face incident to the right of the root edge, which must be a simple face (its boundary is a simple cycle). The latter face will be called the bottom face
of the triangulation, and its boundary $\partial \mathrm{t}$ is called the bottom cycle. If the length of the bottom cycle is $p$, we speak of a
triangulation of the $p$-gon. The height of a vertex $x \in \mathsf{V}( \mathrm{t})$ is the minimal graph distance between $x$ and a vertex of the bottom cycle.

 \subsection{Enumeration background}
 \label{sec:enumeration}
 
For every $p\geq 1$ and $n\geq 0$, we let $ \T_{n,p}$ be the set of all (type I) triangulations of the $p$-gon with $n$ inner vertices. We list here the enumeration results that we will need. These results can be found in Krikun \cite{Kri07} (Krikun uses the number of edges as the size parameter and in order to apply his formulas we note that a triangulation of the $p$-gon with $n$ inner vertices has $3n+2p-3$ edges). The set
$ \T_{0,1}$ is empty, and,
for $n\geq 1$ and $p \geq 1$, or $n = 0$ and $p\geq 2$, we have 
 \begin{eqnarray} \label{eq:asymp} \# \T_{n,p} =4^{n - 1}  \frac{p\, (2p)!\,(2 p + 3 n - 5)!!}{(p!)^2\,n! \,
(2 p + n - 1)!!} \underset{n \to \infty}{\sim}   C(p) \,(12 \sqrt{3})^{n}\, n^{-5/2},  \end{eqnarray}
with
\begin{equation} \label{def:cp} C({p}) = \frac{3^{p-2} \, p \, (2 p)!}{4 \sqrt{2 \pi} \,  (p!)^2}  \underset{p \to \infty}{\sim} \frac{1}{36\pi \sqrt{2}} \, \sqrt{p} \ 12^p,
\end{equation}
where we write $a(p)\sim b(p)$ if the ratio $a(p)/b(p)$ tends to $1$.
Note that formula \eqref{eq:asymp} gives $ \# \T_{0,2}=1$, with 
the usual convention 
$(-1)!!=1$. This formula holds because we consider the rooted planar map consisting of 
a single (oriented) edge as a trivial triangulation of the $2$-gon. This ``edge-triangulation'' 
plays an important role as it will be used later to ``glue'' the two sides of a $2$-gon.

Consider a (rooted) plane triangulation $\mathrm{t}$ with $n$ vertices ($n\geq 3$). 
If we split the root edge of $\mathrm{t}$ into a double edge, then add a loop inside  the region bounded by this double edge,  and root the resulting map on this new loop in clockwise direction, we obtain a triangulation of the $1$-gon with $n-1$ internal vertices. The way the new loop is added should be clear from Fig.~\ref{fig:transform-root}, and we note that the 
construction works as well in the case where the initial root edge is itself a loop, as shown in the right part of Fig.~\ref{fig:transform-root}. The previous construction yields, for every $n\geq 3$, a 
bijection between  $ \T_{n-1,1}$ and the set of all rooted plane triangulations with $n$ vertices. Hence, we may and will often see rooted plane triangulations as triangulations of the $1$-gon.
\begin{figure}[!h]
 \begin{center}
 \includegraphics[width=1\linewidth]{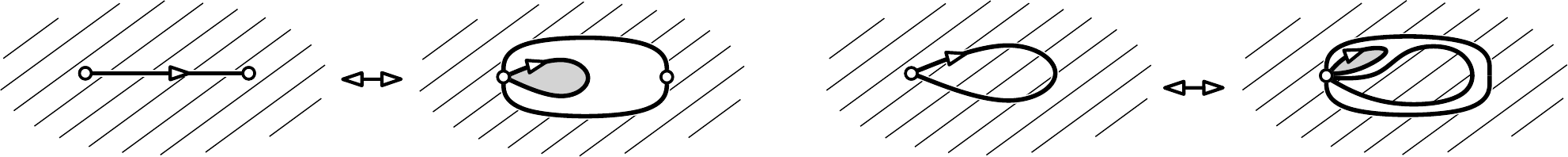}
 \caption{ \label{fig:transform-root} Transforming the root edge into a loop (in the right part, the case where the initial root edge is a loop) provides a bijection between $ \T_{n-1,1}$ and the set of all rooted plane triangulations with $n$ vertices.}
 \end{center}
 \end{figure}

Fix $p\geq 1$ and set
$Z(p) = \sum_{ n \geq 0}  (12 \sqrt{3})^{-n}\,\# \T_{n,p}$. Note that the series converges 
because of the asymptotics in \eqref{eq:asymp}. The (critical) Boltzmann distribution 
on triangulations of the $p$-gon is the probability measure on $\cup_{n\geq 0} \T_{n,p}$ that assigns mass $(12 \sqrt{3})^{-n}/Z(p)$ to every element of $ \T_{n,p}$, for every $n\geq 0$. We have the exact formulas (see \cite[Section 2.2]{ACpercopeel})
 \begin{eqnarray} Z(p) = \frac{6^p\,(2p-5)!! }{8 \sqrt{3} \,p!} \quad \text{if \ } p\geq 2,
  \qquad \qquad  Z(1) =  \frac{2 - \sqrt{3}}{4}.     \label{eq:zp} \end{eqnarray}
The generating series of $Z(p)$ can also be computed explicitly from \cite[formula (4)]{Kri07} 
and an appropriate change of variables (we omit the details):
  \begin{eqnarray} \label{eq:gfzp}  \sum_{p\geq 0} Z(p+1) z^p=  \frac{1}{2}+ \frac{ (1- 12 z)^ {3/2}-1}{24 \sqrt {3} z}. \end{eqnarray}
From the explicit formula \eqref{eq:zp} and Stirling's formula, we have 
\begin{equation}
\label{asympzp}
Z(p+1) \underset{p\to \infty}{\sim}   \frac{ \sqrt{3}}{8 \sqrt{\pi}} \,12^p \,p^{-5/2},
\end{equation}

\begin{lemma} \label{lem:bound}There exists a constant $c >0$ such that, for every $n\geq 1$ and $p \geq 1$, we have 
$$ \# \T_{n,p} \leq c \cdot C(p)\, n^{-5/2} (12 \sqrt{3})^n.$$
Furthermore, for every choice of $\alpha >0$, there exists a constant $c'=c'(\alpha)>0$ such that, for
every $n\geq 1$ and for all integers $p$ with $1\leq p\leq \alpha\sqrt{n}$, we have
$$ \# \T_{n,p} \geq c' \cdot C(p)\, n^{-5/2} (12 \sqrt{3})^n.$$
\end{lemma}

\proof We start with the first assertion. Note that this estimate does not immediately follow from \eqref{eq:asymp} because we want $c$
to be independent of $p$. However, \eqref{eq:asymp} allows us to restrict our attention to the case $p\geq 2$. 
Using Stirling's formula and the fact that 
$ c_{1} \sqrt{n}(n/e)^{n/2}\leq n!!\leq c_{2} \sqrt{n} (n/e)^{n/2}$ for every $n\geq 1$,
for some positive constants $c_1$ and $c_2$, we get, with a constant $c$ that may vary from line to line,
 \begin{eqnarray*} 
 4^{n - 1}  \frac{p\, (2p)!\,(2 p + 3 n - 5)!!}{(p!)^2\,n! \,
(2 p + n - 1)!!} 
& \leq & c \cdot 4^n p\, \frac{4^p}{ \sqrt{p}} \sqrt{ \frac{2p+3n-5}{2p+n-1}} \frac{1}{ \sqrt{n}} \left(\frac{(2 p + 3 n - 5)^{2 p + 3 n - 5}}{ n^{2n} (2 p + n-1)^{2 p + n-1} }\right)^{1/2} \\ 
& \leq & c \cdot \frac{4^n}{ \sqrt{n}} \sqrt{p} \,4^p \left( \frac{(3n)^{2p+3n-5}}{n^{2p+3n-1}}\frac{(\frac{2 p}{3n} + 1 - \frac{5}{3n})^{2 p + 3 n - 5}}{ ( \frac{2 p}{n} + 1-\frac{1}{n})^{2 p + n-1} }\right)^{1/2} \\
 & \leq & c \cdot  \frac{(12 \sqrt{3})^n}{ n^{5/2}} \ \sqrt{p} \,12^p \left( \frac{(\frac{2 p}{3n} + 1 - \frac{5}{3n})^{2 p + 3 n - 5}}{ ( \frac{2 p}{n} + 1-\frac{1}{n})^{2 p + n-1} }\right)^{1/2} \\ 
 &\leq& c \cdot C(p) \frac{(12 \sqrt{3})^n}{ n^{5/2}}  \left( \frac{(\frac{2 p-1}{3n} + 1)^{\frac{2p-1}{n} + 3}}{ ( \frac{2 p-1}{n} + 1)^{\frac{2p-1}{n} + 1} }\right)^{n/2}. \end{eqnarray*}
To complete the proof of the first assertion,  it suffices to verify that the quantity inside the big parentheses
in the last display is smaller than $1$ for any $n \geq 1$ and $p \geq 2$.  To this end, take $x = (p-\frac{1}{2})/n \geq 0$ and note that the function $f(x)= (2 x /3+1)^{2x+3} (2 x +1)^{-2 x -1}$ is bounded 
above by $1$ on $ \mathbb{R}_{+}$ (setting $u=2x+1\geq 1$, and taking logarithms, we
need to see that $(u+2)(\log(u+2)-\log 3)-u\log u\leq 0$, which we get by differentiating this
function of $u$).  

To get the second assertion of the lemma, we use similar arguments to arrive at the lower bound
$$ 4^{n - 1}  \frac{p\, (2p)!\,(2 p + 3 n - 5)!!}{(p!)^2\,n! \,
(2 p + n - 1)!!} \geq c' \cdot  C(p) \frac{(12 \sqrt{3})^n}{ n^{5/2}}   \left( \frac{(\frac{2 p}{3n} + 1 - \frac{5}{3n})^{2 p + 3 n - 5}}{ ( \frac{2 p}{n} + 1-\frac{1}{n})^{2 p + n-1} }\right)^{1/2},$$
and elementary considerations show that, under our condition $1\leq p\leq \alpha\sqrt{n}$, the ratio
$$ \frac{(\frac{2 p}{3n} + 1 - \frac{5}{3n})^{2 p + 3 n - 5}}{ ( \frac{2 p}{n} + 1-\frac{1}{n})^{2 p + n-1} }$$
is bounded below by a positive constant depending on $\alpha$.
\endproof

 \subsection{Triangulations of the cylinder and their skeleton decomposition}
 
 \label{sec:skeleton}
 
 In this section, we discuss a special class of planar maps, which we call triangulations of
 the cylinder. The reason for considering this class is the fact that hulls of
 triangulations with a boundary will be triangulations of
 the cylinder. We then describe these objects via a \textit{skeleton decomposition} which was first discussed in Krikun \cite{Kri04,Kri05} in the case of type II triangulations and quadrangulations.
 
 \begin{definition}
 \label{def:cylinder}
 Let $r\geq 1$ be an integer. A triangulation of the cylinder of height $r$ is a rooted planar map
such that all faces are triangles except for two distinguished faces called respectively the
bottom face and the top face, and such that the following properties hold. The boundaries of
the bottom face and of the top face are disjoint simple cycles.
The boundary of the 
bottom face contains the root edge, which is oriented in such a way that the bottom face lies on its right.
Finally, every vertex incident to the top face is at graph distance exactly $r$ from the boundary of the bottom 
face, and  every edge incident to the top face is also incident to a triangle whose third vertex is at distance $r-1$
 from the bottom face.
 \end{definition}
 
If $\Delta$ is a triangulation of the cylinder of height $r$, the bottom cycle is again denoted by $\partial \Delta$ and the top cycle
 (boundary of the top face) of $\Delta$ is denoted by $\partial^*\Delta$.

 \smallskip
 
Let $\Delta$ be a fixed triangulation of the cylinder of height $r$. We
 let $p\geq1$ be the length of the bottom cycle $\partial\Delta$, and let
$q\geq 1$ be the length of the top cycle $\partial^*\Delta$.
We will now describe a skeleton decomposition that encodes the triangulation $\Delta$ via an
ordered forest of $q$ (rooted) plane trees with maximal height $r$, and a collection, indexed by the vertices of the forest at
height strictly less than $r$, of triangulations with a boundary. To describe this decomposition, we first need to
introduce some notation. For $1\leq j<r$, the ball $B_j(\Delta)$ is defined as the union of all faces
of $\Delta$ that are incident to a vertex at graph distance strictly less than $j$ from the bottom cycle, and the hull
 $B^\bullet_j(\Delta)$ is obtained by adding to the ball $B_j(\Delta)$  all connected components of its complement
 except for the one containing the top cycle.
It is an easy exercise to see that $B^\bullet_j(\Delta)$ is then a triangulation of the cylinder 
 of height $j$, and we let $\partial_{j}\Delta=\partial^*B^\bullet_j(\Delta)$ denote its top cycle. By convention, $\partial_{0}\Delta = \partial \Delta$
 is the bottom cycle of $\Delta$ and $\partial_{r}\Delta = \partial^* \Delta$ is the top cycle of $\Delta$. 
 We may and will assume that $\Delta$ is drawn in the plane in such a way that the top face is the unbounded face, as in Fig.~\ref{fig:treesfromlayers},
 and we then orient all cycles $\partial_{j}\Delta$ in clockwise order (for $j=0$ this is consistent with the orientation of the root edge).

\smallskip 

If $k\in\{1,2,\ldots,r\}$, every edge of $  \partial_k \Delta$ is incident to (exactly) one
triangle whose third vertex belongs to $\partial_{k-1}\Delta$. We call such triangles the
downward triangles at height $k$. These triangles are in one-to-one correspondence with the edges 
of $\partial_k\Delta$. Let $\mathsf{E}_d(\Delta)$ be the collection of all edges of $\Delta$ that 
belong to one of the cycles $\partial_j\Delta$ for $0\leq j\leq r$. We define a genealogical order on the
set $\mathsf{E}_d(\Delta)$
by saying that, for every $k\in\{1,\ldots,r\}$, an edge $e$ of $\partial_k\Delta$ is the 
``parent'' of an edge $e'$ of $\partial_{k-1}\Delta$ if the downward triangle associated
with $e$ is the first one that one encounters when turning around $\partial_{k-1}\Delta$ inside $ \Delta\backslash B_{k-1}(\Delta)$
in clockwise order, starting from the middle of the edge $e'$. The definition of this genealogical order should be clear from the
right part of Fig.~\ref{fig:treesfromlayers}.

\smallskip
Thanks to the planar structure of $\Delta$, these genealogical relations lead to a forest of $q$ plane trees, 
whose vertices are (in one-to-one correspondence with) the edges belonging to $\mathsf{E}_d(\Delta)$, and which are rooted at the edges  
of $ \partial_r\Delta$. The maximal height in the forest is $r$,
and a vertex at height (distance from the ancestor) $r-j$, for $0\leq j\leq r$, corresponds to an edge of $ \partial_j\Delta$.
See Fig.~\ref{fig:treesfromlayers}. We write $\tau_1,\tau_2,\ldots,\tau_q$ for the trees in the forest listed around $\partial^*\Delta$ in clockwise order, in such a way that $\tau_1$ is the tree containing the vertex corresponding to the 
root edge of $\Delta$. Note that $\tau_1$ is a tree with height $r$ and a distinguished vertex
(the root edge of $\Delta$) at height $r$. 

\begin{figure}[!h]
 \begin{center}
 \includegraphics[width=1 \linewidth]{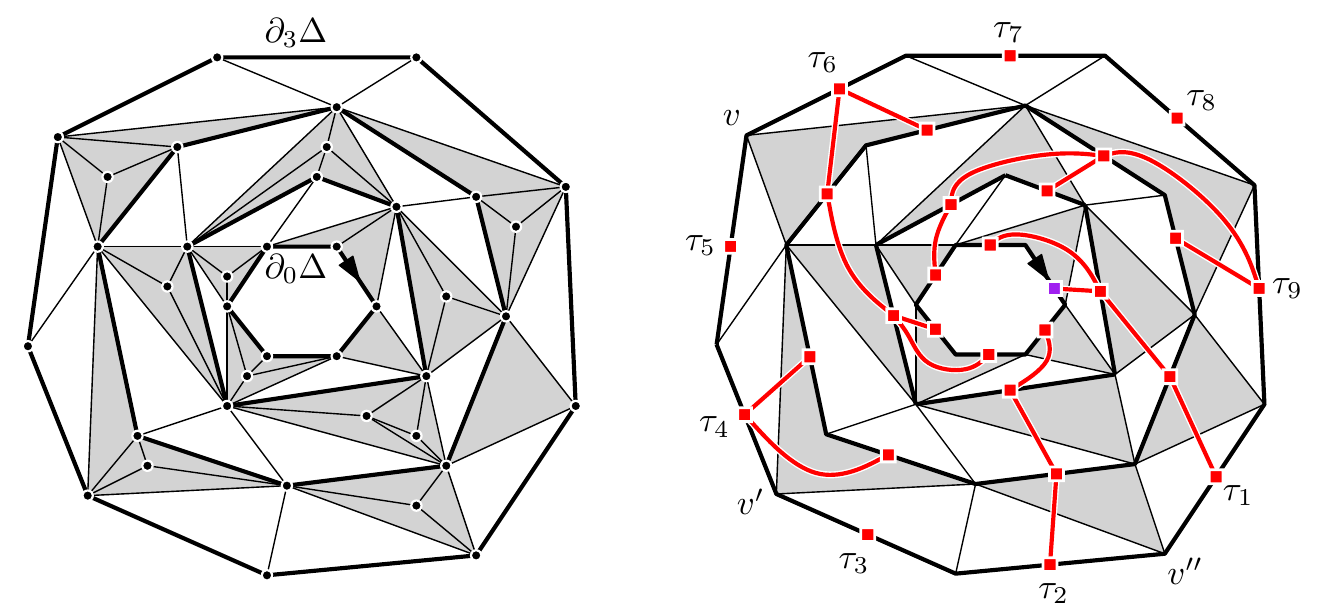}
 \caption{ \label{fig:treesfromlayers} The skeleton decomposition of a triangulation of the cylinder of height $3$. We have chosen to show a triangulation without multiple edges or loops  
 for the sake of clarity of the figure. The downward triangles are represented in white and the other triangles are in grey in the left part
 of the figure. In the right part, we have erased the edges that are not incident to downward triangles (except for those of $\partial_0\Delta=\partial \Delta$) so that
 the slots appear in grey. The forest of trees associated with the triangulation is in red in the right part of the figure (notice that the trees $\tau_3,\tau_5,\tau_7,\tau_8$ are trivial trees consisting only of the root).}
 \end{center}
 \end{figure}

The forest $(\tau_1,\tau_2,\ldots,\tau_q)$ does not give enough information to reconstruct the 
triangulation $\Delta$. Indeed this forest only characterizes the configuration
of downward triangles and, informally, we need to ``fill in'' the holes left by these
triangles. 
More precisely, if we remove all the downward triangles from $\Delta$, we are left with the top and bottom faces and a collection of  ``slots''. 
If $e$ is an edge of $ \partial_k\Delta$, where $1\leq k\leq r$, the associated slot is bounded
by the edges of $ \partial_{k-1}\Delta$ that are children of $e$ and by two ``vertical'' edges
connecting the initial vertex of $e$ (recall that $\partial_k\Delta$ is oriented in clockwise order) to vertices of $ \partial_{k-1}\Delta$ (if $e$ has no child, these two
vertical edges may, or may not, be glued in a single edge, see Fig.~\ref{fig:slots}). We may assign a root edge to the boundary
of each slot, by deciding that the root edge of the slot associated with an edge $e$ of $ \partial_k\Delta$ is the vertical edge, oriented so that 
its initial vertex is on $ \partial_{k-1}\Delta$, which is incident on its right to the downward triangle associated with $e$.  Then
the slot associated with $e$ is filled in by a well-defined triangulation of the $(c_e+2)$-gon, where 
$c_e\geq 0$ is the number of children of $e$. Note that when $c_e=0$ it may happen that
the slot is filled in by the edge-triangulation, which just means that the two vertical edges are glued together. 

\smallskip

Say that a forest $\mathcal{F}$ with a distinguished vertex
is $(p,q,r)$-admissible if
 \begin{enumerate}[(i)]
 \item the forest consists of an ordered sequence $ (\tau_{1}, \tau_{2}, \ldots , \tau_{q})$ 
 of $q$ (rooted) plane trees,
\item the maximal height of these trees is $r$,
\item the total number of vertices of the forest at generation $r$ is  $p$,
\item the distinguished vertex has height $r$,
\item the distinguished vertex is in $\tau_1$.
\end{enumerate}
If  $\mathcal{F}$ is a $(p,q,r)$-admissible forest, we write $\mathcal{F}^*$ for the set all vertices of $\mathcal{F}$ at height strictly less than $r$.

The preceding  decomposition yields a bijection between, on the one hand, triangulations $\Delta$ of the cylinder of height $r$ with a bottom cycle of length $p$ and a top cycle of length $q$, and on the other hand, 
pairs consisting of a
$(p,q,r)$-admissible pointed 
forest $\mathcal{F}$ and a collection $(M_v)_{v\in\mathcal{F}^*}$
 such that, for every $v\in\mathcal{F}^*$,  $M_v$ is
a triangulation of the $(c_v+2)$-gon, if $c_v$ stands for the number of children of $v$
in $\mathcal{F}$. We call this bijection the skeleton decomposition and we say that $\mathcal{F}$ is the skeleton of the triangulation $\Delta$. 

  \begin{figure}[!h]
  \begin{center}
  \includegraphics[width=1 \linewidth]{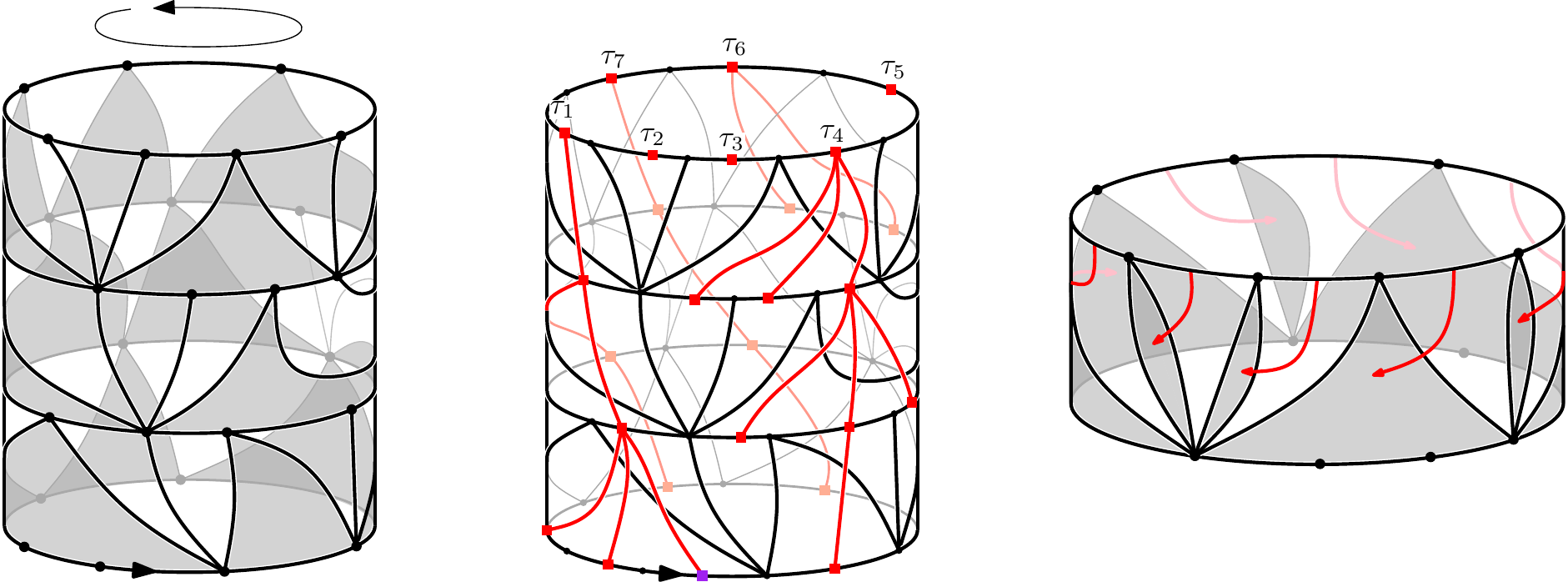}
  \caption{ \label{fig:slots} Left, a representation of a triangulation without multiple edges of the cylinder of height $3$ with the slots in grey. In the middle, the construction of the associated forest. Right, the slots  of a triangulation of the cylinder of height $1$ (possibly having multiple edges):  when an edge 
  of the top boundary has no child, the corresponding slot is bounded by a double edge, which may be glued into a single one if the slot is filled in by the edge-triangulation.}
  \end{center}
  \end{figure}
 
\begin{remark} Our skeleton decomposition is slightly simpler that the original one in Krikun \cite{Kri04}: Because we deal with with type I triangulations, where loops are allowed, we do not need to exclude the case considered in \cite[Lemma 2.2]{Kri04}.
\end{remark}

\paragraph{Left-most geodesics.}\label{sec:left-mostgeo} Let  $x$ be a vertex of $\partial_j \Delta$, where $1\leq j\leq r$.
We define the left-most geodesic from $x$ to the bottom cycle in the following way.
We first observe that half-edges incident to $x$ can be enumerated in clockwise order around $x$, starting from the 
half-edge of $\partial_j\Delta$ whose initial vertex is $x$ (recall that we have also oriented the cycles $\partial_k\Delta$).
The first edge on the
left-most geodesic from $x$ is the last edge connecting $x$ to $\partial_{j-1}\Delta$ arising in this enumeration. The path is then continued by
the obvious induction.

Left-most geodesics from distinct vertices may coalesce, and more precisely we have the following property. Let $u$ and $v$ 
be two distinct vertices of $\partial^*\Delta$.
If $\mathcal{F}$ is the skeleton of $\Delta$, let $\mathcal{F}'$ be the subforest of $\mathcal{F}$ consisting of the
trees rooted at the edges of $\partial^*\Delta$ that belong to the path going from $u$ to $v$ 
in clockwise order in $\partial^*\Delta$, and let $\mathcal{F}''$ consist of the trees of $\mathcal{F}$
that are not in $\mathcal{F}'$. Then, for every 
$k\in\{1,\ldots,r\}$, the left-most geodesics starting respectively from $u$ and from $v$  merge before step $k$
(possibly at step $k$) if and only if at least one of the two forests $\mathcal{F}'$ and $\mathcal{F}''$
has height strictly smaller than $k$. We leave the easy verification to the reader
(as an illustration, the reader may verify in Fig.~\ref{fig:treesfromlayers} that the left-most geodesics starting respectively
from $v$ and $v'$ coalesce at step $2$, but the left-most geodesics starting respectively from $v'$ and $v''$
do not coalesce). \medskip

\subsection{Hulls as triangulations of the cylinder}

 Let $\mathrm{t}$ be a triangulation of the $p$-gon, where $p\geq 1$. For every  integer $ r \geq 1$, the ball 
 $B_r(\mathrm{t})$ is  the planar map  obtained as the union of all faces of $\mathrm{t}$
 that are incident to a vertex at distance less than or equal to $r-1$ from the bottom cycle.
 Clearly, the graph distance between any vertex of $B_r(\mathrm{t})$ and the bottom cycle is
 at most $r$, and this graph distance is exactly $r$ if the vertex is in the boundary of a 
 connected component of the complement of $B_r(\mathrm{t})$ in the sphere.
 
We now let $o$
 be a distinguished vertex of $\mathrm{t}$. Then
 $\overline{\mathrm{t}}:=(\mathrm{t},o)$ is a pointed (and rooted) triangulation of the $p$-gon. 
Let $h$ be the height of $o$ (recall that this
 is the minimal distance between $o$ and a vertex of the bottom cycle). Suppose that $r<h$, so that the distinguished vertex $o$  lies in the complement of $B_r(\mathrm{t})$.
 The hull $B^\bullet_r(\,\overline{\mathrm{t}}\,)$ is obtained by adding to $B_r(\mathrm{t})$ all connected components of 
 the complement of $B_r(\mathrm{t})$,
 except for the one containing $o$. Then under the preceding assumptions,  it is easy to see that the hull $B^\bullet_r(\,\overline{\mathrm{t}}\,)$ is a triangulation of the cylinder of
 height $r$, in the sense of Definition \ref{def:cylinder}.  For future reference, we note that $B^\bullet_j(B^\bullet_r(\,\overline{\mathrm{t}}\,))=B^\bullet_j(\,\overline{\mathrm{t}}\,)$, for $1\leq j<r$.

\subsection{The skeleton decomposition of random triangulations}
\label{subsec:skeletondecom}

We now consider random triangulations. We fix $p\geq 1$ and, for every $n\geq 1$, we let $\t^{(p)}_n$
 be uniformly distributed over the set $\T_{n,p}$ of all triangulations of the $p$-gon with $n$
 inner vertices. We also write $\ov\t^{(p)}_n$ for the pointed 
 triangulation obtained by choosing an inner vertex uniformly at random in $\t^{(p)}_n$.
 For every integer $r\geq 1$, we can make sense of the hull $B^\bullet_r(\ov\t^{(p)}_n)$
 provided that the distance between the distinguished vertex and the bottom cycle is at least
 $r+1$. For definiteness, if the latter condition does not hold, we define the hull $B^\bullet_r(\ov\t^{(p)}_n)$ as
 the edge-triangulation. Our goal is to evaluate the asymptotics as $n\to\infty$ of the probability that the hull $B^\bullet_r(\ov\t^{(p)}_n)$ is equal to
 a given triangulation $\Delta$ of the cylinder of height $r$, in terms of the skeleton decomposition of $\Delta$.
 
Let $r\geq 1$, and let $\Delta$ be a triangulation of the cylinder of height  $r$. Write
 $p$ for the length of the bottom cycle of $\Delta$ and $q$ for the length of the top cycle. 
 Also let $N$ be the total number of vertices of $\Delta$. By the skeleton decomposition
 of Section \ref{sec:skeleton}, we can associate with $\Delta$ a $(p,q,r)$-admissible
 forest $\mathcal{F}=(\tau_1,\ldots,\tau_q)$. As in Section \ref{sec:skeleton}, we write $(M_v)_{v\in\mathcal{F}^*}$ for the triangulations
with a boundary filling in the slots in $\Delta$. We also let $ \mathsf{Inn}(M_v)$ stand for the number of
inner vertices of $M_v$, for every $v\in \mathcal{F}^*$. Recall the constants $C(p)$ and $Z(p)$ defined in \eqref{def:cp} and \eqref{eq:zp}.

 \begin{lemma}
 \label{law-hull}
 We have 
 $$\lim_{n\to\infty} \P(B^\bullet_r(\ov\t^{(p)}_n)=\Delta)= \frac{12^{-q}{C}(q)}{12^{-p}{C}(p)} \prod_{v \in \mathcal{F}^*} \theta(c_{v}) \times \prod_{v \in \mathcal{F}^*} \frac{(12\sqrt{3})^{-\mathsf{Inn}(M_v)}}{Z(c_v+2)}$$
where $c_{v}$ is the number of children of $v$ in the forest 
$\mathcal{F}$, and $\theta$ is the critical offspring distribution whose generating function $g_\theta$ is given by 
$$g_\theta(x)= \sum_{i=0}^\infty x^i \theta(i) = 1 - \left( 1+ \frac{1}{ \sqrt{1-x}}\right)^{-2}, \qquad x \in [0,1].$$
In particular we have 
\begin{equation}
\label{asymp-theta}
\theta(k) \build{\sim}_{k\to\infty}^{} \frac{3}{2\sqrt{\pi}} \,k^{-5/2}.
\end{equation}
\end{lemma}

\proof In this proof, we set $\rho=12\sqrt{3}$ and $\alpha=12$ to simplify notation. We observe that
the property $B^\bullet_r(\ov\t^{(p)}_n)=\Delta$ holds if and only if $\t^{(p)}_n$ is obtained from $\Delta$ by gluing 
an arbitrary triangulation of the $q$-gon 
with $n-(N-p)$ inner vertices on the top cycle of $\Delta$ (to perform this gluing, we may assume
that the top cycle is rooted at the root of the first forest in $\mathcal{F}$) and if the distinguished vertex
of $\ov\t^{(p)}_n$ is  chosen among the inner vertices of the glued triangulation. It follows that, for
$n\geq N-p$,
 \begin{equation} \label{eq:hulldiscrete} \P(B^\bullet_r(\ov\t^{(p)}_n)=\Delta)=\frac{\# \T_{n-(N-p),q}}{\#\T_{n,p}}\times \frac{n-(N-p)}{n}.  \end{equation}
Using \eqref{eq:asymp}, we now get
\begin{eqnarray}
\lim_{n\to\infty} \P(B^\bullet_r(\ov\t^{(p)}_n)=\Delta)= \frac{C(q)}{C(p)} \,\rho^{-N+p}.
\label{eq:loihullUIPHT}\end{eqnarray}
Notice that the total number of vertices of $\Delta$ can be decomposed as $N= S+ \sum_{v \in \mathcal{F}^*}  \mathsf{Inn}( M_v)$ where $S=\#\mathsf{E}_d(\Delta)$ is the total number of vertices in the 
cycles $\partial_{j}\Delta$, $0\leq j\leq r$, with the notation of Section \ref{sec:skeleton}. We have $$S = \sum_{i=1}^q \# \tau_{i} = q+ \sum_{v \in \mathcal{F}^*} c_{v}.$$ Hence the right-hand side of
\eqref{eq:loihullUIPHT} is also equal to
$$ 
\frac{C(q)}{C(p)}  {\rho^{p-q}}\prod_{v \in \mathcal{F}^*} \rho^{- \mathsf{Inn}( M_v)} \rho^{-c_{v}}.$$
 We then observe that $\sum_{ v \in  \mathcal{F}^*} (c_{v}-1) = p-q$ and so we can multiply the
quantity in the last display by
 $$\Big(\frac{\alpha}{\rho}\Big)^{p-q-\sum_{v\in\mathcal{F}^*} (c_v-1)}$$
 to get that
 \begin{eqnarray} 
\lim_{n\to\infty} \P(B^\bullet_r(\ov\t^{(p)}_n)=\Delta)  &=& \frac{\alpha ^{-q} C(q)}{\alpha ^{-p} C(p)} \prod_{v \in \mathcal{F}^*} \Big(\rho^{-1} \alpha^{-c_{v}+1} 
\rho^{- \mathsf{Inn}({M}_{v})}\Big)\nonumber \\ 
 &=& \frac{\alpha ^{-q} C(q)}{\alpha ^{-p} C(p)} \prod_{v \in \mathcal{F}^*} \Big(\theta( c_{v})  \cdot 
 \frac{\rho^{- \mathsf{Inn}({M}_{v})} }{Z({c_{v}+2})}\Big), \label{eq:enumetree} \end{eqnarray}
where we have set
 \begin{equation} \theta(k) =  \frac{1}{\rho}\alpha^{-k+1} Z({k+2}), \quad \mbox{ for every }k\geq 0.  \label{def:theta}\end{equation} Recall that $\rho=12 \sqrt{3}$ and $\alpha = 12$ so that, using   \eqref{eq:zp} and \eqref{eq:gfzp}, we get for $0<x\leq 1$,
\begin{align*}
\sum_{k=0}^\infty x^k\theta(k)
&=\frac{\alpha}{\rho}\sum_{k=0}^\infty  \Big(\frac{x}{\alpha}\Big)^k\, Z(k+2)\\
&=\frac{\alpha}{\rho}\sum_{i=1}^\infty  \Big(\frac{x}{\alpha}\Big)^{i-1}\, Z(i+1)\\
&=\frac{\alpha^2}{x \rho}\left(\sum_{i=0}^\infty \Big( \frac{x}{\alpha}\Big)^iZ(i+1) - Z(1)\right)\\
&=  \frac{4 \sqrt{3}}{x}\Big(\frac{1}{2} + \frac{(1-x)^{3/2}-1}{2\sqrt{3}\,x} - \frac{2 - \sqrt{3}}{4}\Big)\\
&=  \frac{2(1-x)^{3/2} + 3x -2}{x^2}\\
&= 1-\Big(1+\frac{1}{\sqrt{1-x}}\Big)^{-2}.
\end{align*}
From this explicit formula for the generating function, one gets  that $\theta$ is a probability
distribution with mean $1$. 
Finally the asymptotics \eqref{asymp-theta} are derived from the formula \eqref{def:theta}  for $\theta(k)$
and \eqref{asympzp}.
This completes the proof.
 \endproof
 
 \begin{remark} It is interesting to notice that although we consider a slightly different model, namely type I triangulations, we recover the same 
 offspring distribution $\theta$ in the case of type II triangulations \cite{Kri04} . This may be explained geometrically by the relations between 
 type I and type II triangulations. \end{remark}
 
 The distribution $\theta$ plays an important role in this work. It will be convenient to write
 $Y=(Y_r)_{r\geq 0}$ for a Galton--Watson process with offspring distribution $\theta$ that starts 
 from $k$ under the probability measure $\P_k$. Then, for every integer $r\geq 1$, the generating function of the  distribution 
 of $Z_r$ under $\P_1$ is the $r$-th iterate $g^{(r)}_\theta$ of $g_\theta$. A simple induction shows that, for every
 integer $r\geq 1$, for every $x\in[0,1)$,
 \begin{equation}
\label{iterate-gtheta}
\mathbb{E}_1[x^{Y_r}] = g^{(r)}_\theta(x)= 1 - \Big(r+\frac{1}{\sqrt{1-x}}\Big)^{-2}.
\end{equation}
The function $g^{(r)}_\theta$ can be extended to a holomorphic function on $\C\backslash [1,\infty)$. For this extension, we have
$$g^{(r)}_\theta(z)-z= (1-z)\Big(1-\frac{1}{1+r\sqrt{1-z}}\Big)\Big(1+\frac{1}{1+r\sqrt{1-z}}\Big) \sim 2r\,(1-z)^{3/2},$$
as $z\to 1$, $z\in \C\backslash [1,\infty)$. By a well-known result of singularity analysis (see Corollary VI.1 p.392 in \cite{FS09}), 
it follows that
\begin{equation}
\label{asymp-giterate}
\P_1(Y_r=k)\build{ \sim}_{k\to\infty}^{} \frac{3r}{2\sqrt{\pi}}\,k^{-5/2}.
\end{equation}
Of course we recover \eqref{asymp-theta} when $r=1$.

\medskip

We let $\F_{p,q,r}$ be the set of all $(p,q,r)$-admissible forests. We will also consider the set  $\F'_{p,q,r}$ of all (pointed) forests satisfying the same properties
as $(p,q,r)$-admissible forests, except that we do not require property (v) saying that the distinguished vertex belongs to
the first tree in the forest. Finally, we will write  $\F''_{p,q,r}$ for the set of all (non-pointed) forests satisfying the same properties
(i)--(iii)
as $(p,q,r)$-admissible forests except that no special vertex is distinguished (and so properties (iv) and (v) become irrelevant).

\begin{lemma}
\label{full-mass}
For every $p\geq 1$ and $r\geq 1$,
$$\sum_{q=1}^\infty \sum_{\mathcal{F}\in \F_{p,q,r}} \frac{12^{-q}{C}(q)}{12^{-p}{C}(p)} \prod_{v \in \mathcal{F}^*} \theta(c_{v}) = 1.$$
\end{lemma}

\proof We start with some simple observations. If $\mathcal{F}'\in\F'_{p,q,r}$, we can obtain a forest $\mathbf{p}(\mathcal{F}')\in \F_{p,q,r}$ by
cyclically permuting the trees so that the first one will contain the distinguished vertex, and 
if $\mathcal{F}\in \F_{p,q,r}$ there are exactly $q$ choices of $\mathcal{F}'\in\F'_{p,q,r}$ such that
$\mathbf{p}(\mathcal{F}')=\mathcal{F}$. Hence, we have also
$$\sum_{q=1}^\infty \sum_{\mathcal{F}\in \F_{p,q,r}} \frac{12^{-q}{C}(q)}{12^{-p}{C}(p)} \prod_{v \in \mathcal{F}^*} \theta(c_{v}) = \sum_{q=1}^\infty \sum_{\mathcal{F}\in \F'_{p,q,r}} \frac{1}{q}\,\frac{12^{-q}{C}(q)}{12^{-p}{C}(p)} \prod_{v \in \mathcal{F}^*} \theta(c_{v}).$$
Then, with every 
$\mathcal{F}'\in\F'_{p,q,r}$, we can associate $\mathcal{F}''\in\F''_{p,q,r}$ by simply ``forgetting'' the
distinguished vertex, and each $\mathcal{F}''\in\F''_{p,q,r}$ corresponds through this operation
to $p$ possible choices  of $\mathcal{F}'\in\F'_{p,q,r}$. Hence,
\begin{equation}
\label{mass-forest}
\sum_{q=1}^\infty \sum_{\mathcal{F}\in \F'_{p,q,r}} \frac{1}{q}\,\frac{12^{-q}{C}(q)}{12^{-p}{C}(p)} \prod_{v \in \mathcal{F}^*} \theta(c_{v})
= \sum_{q=1}^\infty \sum_{\mathcal{F}\in \F''_{p,q,r}} \frac{p}{q}\,\frac{12^{-q}{C}(q)}{12^{-p}{C}(p)} \prod_{v \in \mathcal{F}^*} \theta(c_{v})
= \sum_{q=1}^\infty \sum_{\mathcal{F}\in \F''_{p,q,r}} \frac{h(q)}{h(p)}\, \prod_{v \in \mathcal{F}^*} \theta(c_{v}),
\end{equation}
where using \eqref{def:cp} we put for every $k\geq 1$, 
$$h(k)= 4^{-k} {2k \choose k}= {36} {\sqrt{2\pi}}\,\frac{12^{-k} C(k)}{k} .$$
The right-hand side of \eqref{mass-forest} can also be written as
$$\sum_{q=1}^\infty \frac{h(q)}{h(p)} \,\mathbb{P}_q(Y_r=p)$$
with the notation introduced before the statement of the lemma. 
So in order to prove the lemma, we need to verify that, for every $p\geq 1$,
\begin{equation}
\label{statio}
\sum_{q=1}^\infty h(q) \,\mathbb{P}_q(Y_r=p)= h(p).
\end{equation}
This is equivalent to saying that $(h(k))_{k\geq 1}$ is an infinite stationary  measure for the 
Galton--Watson process $Y$. To see this, 
set, for every $x\in[0,1)$,
$$\Pi(x)= \sum_{k=1}^\infty h(k)\,x^k = \frac{1}{\sqrt{1-x}} -1.$$
To verify that $(h(k))_{k\geq 1}$ is a stationary distribution for $Y$, it is enough
\cite[Chapter II]{AN72} to
check that $\Pi(g_\theta(x))-\Pi(g_\theta(0))= \Pi(x)$, for every $x\in [0,1)$. This follows
from the explicit formulas for $g_\theta$ and $\Pi$.
\endproof

For fixed integers $p\geq 1$ and $r\geq 1$, define
$$\F_{p,r}:=\bigcup_{q=1}^\infty \F_{p,q,r}.$$
Also write $\C_{p,r}$ for the (countable) set of all triangulations of the cylinder of height $r$ with a bottom cycle of length $p$. 

Lemma \ref{full-mass} allows us to define a probability measure $\mathbf{P}_{p,r}$ on $\F_{p,r}$ by setting, for every
forest $\mathcal{F}\in \F_{p,q,r}$,
\begin{equation}
\label{law-forest}
\mathbf{P}_{p,r}(\mathcal{F}) := \frac{12^{-q}{C}(q)}{12^{-p}{C}(p)} \prod_{v \in \mathcal{F}^*} \theta(c_{v})
.
\end{equation}
We then define a probability measure $\P_{p,r}$ on the set $\C_{p,r}$, by requiring that under $\P_{p,r}$ the skeleton 
is distributed according to $\mathbf{P}_{p,r}$ and, conditionally on the skeleton, the
triangulations with a boundary filling in the slots are independent and 
Boltzmann distributed (with boundary lengths prescribed by the skeleton).
We can then restate Lemma \ref{law-hull} by saying that, if $\Delta\in \C_{p,r}$,
\begin{equation}
\label{local-pUIPT}
\lim_{n\to\infty} \P(B^\bullet_r(\ov\t^{(p)}_n)=\Delta)=\P_{p,r}(\Delta),
\end{equation}
 or, in other words, the law of $B^\bullet_r(\ov\t^{(p)}_n)$ converges weakly to $\P_{p,r}$ as $n\to\infty$.
 
 For $1\leq r<r'$, there is a canonical projection $\mathbf{p}_{r,r'}$ from $\C_{p,r'}$ onto $\C_{p,r}$, which maps 
 $\Delta\in\C_{p,r'}$ to $B^\bullet_r(\Delta)$. The preceding convergence then implies that the probability measures $\P_{p,r}$, $r\geq 1$, are consistent
 in the sense that $\P_{p,r}=\P_{p,r'}\circ \mathbf{p}_{r,r'}^{-1}$  for every $1\leq r < r'$. It follows that we can define a random infinite triangulation of the plane, which we denote by $\t^{(p)}_\infty$,
 with a (simple) boundary of length $p$, such that 
 $\t^{(p)}_\infty$ has a unique end, and the law of $B^\bullet_r(\t^{(p)}_\infty)$ is equal to $\P_{p,r}$ for every
 integer $r\geq 1$. Note that here infinity plays the role of the distinguished point, so that the hull
 $B^\bullet_r(\t^{(p)}_\infty)$ is obtained by ``filling in'' the finite holes in the ball $B_r(\t^{(p)}_\infty)$, which is itself the
  union of all faces that are incident to a vertex at graph distance strictly less than
 $r$ from the boundary.
 
 We call $\t^{(p)}_\infty$ the (type I) uniform infinite triangulation of the $p$-gon. It follows from \eqref{local-pUIPT} that $\t^{(p)}_\infty$
 is the local limit of $\t^{(p)}_n$ as $n\to\infty$. When $p=1$, the boundary
 is a loop, and, after performing the inverse of the transformation described in Fig.~\ref{fig:transform-root}, we get a random infinite planar triangulation, which we denote by $ \t_{\infty}$.
 Then $\t_\infty$ is the local limit of uniform rooted (plane) triangulations with $n$ vertices when $n\to\infty$, and therefore
 is identified with the type I uniform infinite planar
 triangulation (UIPT), which was already constructed in \cite[Proposition 6.2]{St14}. Notice that this approach via the skeleton
 decomposition gives a simple method for
 constructing the type I UIPT. A similar method was already used by Krikun \cite{Kri05} to construct the Uniform Infinite Planar Quadrangulation.
 
 \subsection{The comparison principle}
 \label{subsec:compa}
 
 Our goal in this section is to obtain a comparison principle showing that
 the law of the skeleton of the hull  $B^\bullet_r(\t^{(p)}_\infty)$ is ``not too different''
 from the law of a finite sequence of independent Galton--Watson trees. We write
 $L^{(p)}_r$ for the  length of the top cycle of $B^\bullet_r(\t^{(p)}_\infty)$.
 When $p=1$, we will write $L_r=L^{(1)}_r$ to simplify notation.
 
 We first discuss a spatial Markov property of the law of $\t^{(p)}_\infty$. To this end, let 
 $1\leq r  < s$ be integers, and let $\Delta\in \C_{p,s}$. Let $q$ be the length 
 of the cycle $\partial_r\Delta$. Then $\Delta$
is obtained by gluing a triangulation $\Delta''\in \C_{q,s-r}$ on the top cycle of
 another triangulation $\Delta'\in \C_{p,r}$, whose top cycle has length $q$. From the explicit formula for 
the probability measure $\mathbf{P}_{p,r}$ on $\F_{p,r}$, it is a simple matter to
 verify that
 $$\P_{p,s}(\Delta)= \P_{p,r}(\Delta')\times \P_{q,s-r}(\Delta'').$$
 From this equality, it follows that, conditionally on $\{L^{(p)}_r=q\}$,
 $B^\bullet_s(\t^{(p)}_\infty)\backslash B^\bullet_r(\t^{(p)}_\infty)$ is distributed 
 according to $\P_{q,s-r}$ and is independent of $B^\bullet_r(\t^{(p)}_\infty)$ --- we slightly abuse notation
 by writing $B^\bullet_s(\t^{(p)}_\infty)\backslash B^\bullet_r(\t^{(p)}_\infty)$ for the triangulation of the cylinder of
 height $s-r$ consisting of the faces of $\t^{(p)}_\infty$ that lie in $B^\bullet_s(\t^{(p)}_\infty)$ but not in
 $B^\bullet_r(\t^{(p)}_\infty)$, and this triangulation is rooted at the edge of $\partial^*B^\bullet_r(\t^{(p)}_\infty)$
 that is the root of the first tree of the skeleton of $B^\bullet_r(\t^{(p)}_\infty)$.
 By letting $s\to\infty$, we also get that, conditionally on $\{L^{(p)}_r=q\}$, the triangulation 
 $\t^{(p)}_\infty\backslash B^\bullet_r(\t^{(p)}_\infty)$ is distributed as $\t^{(q)}_\infty$
 and is independent of $B^\bullet_r(\t^{(p)}_\infty)$. 
 
 The following lemma provides useful estimates about the distribution of $L_r=L^{(1)}_r$.
 
 \begin{lemma}
\label{bdlawperi}
There exists a constant $C_{0} >0$ such that for any integer $\alpha\geq 0$ and any choice of the
integers $r,p \geq 1$ we have 
  \begin{equation}  \mathbb{P}(L_{r} =p) \leq \frac{C_{0}}{r^2} \label{localbound}.\end{equation}
and
 \begin{equation} \label{eq:boundLs} \mathbb{P}(L_{r} > \alpha r^2)  \leq C_{0} e^{-\alpha/5}, \end{equation}
\end{lemma}

The proof of this lemma is postponed to the end of the section.

We now fix a positive constant $a\in(0,1)$,  which may be chosen  arbitrarily small. 
For every integer $r\geq 1$, we let $N_r^{(a)}$ be uniformly distributed over $\{\lfloor ar^2\rfloor+1,\ldots, \lfloor a^{-1}r^2\rfloor\}$. We also
consider a sequence $\tau_1,\tau_2,\ldots$ of independent
Galton--Watson trees with offspring distribution $\theta$, which is independent of $N^{(a)}_r$. 
For every $j\geq 0$, we write
$[\tau_i]_{j}$ for the tree $\tau_i$ truncated at generation $j$ (that is, we remove all vertices
at height strictly greater than $j$).

If $1\leq r < s$, we let $\mathcal{F}^{(1)}_{r,s}$ be the skeleton of $B_{s}^\bullet(\t_{\infty}^{(1)}) \backslash B_{r}^\bullet(\t_{\infty}^{(1)})$. We also write $\wt{\mathcal{F}}^{(1)}_{r,s}$
for the (non-pointed) forest obtained by a random cyclic permutation of the trees of $\mathcal{F}^{(1)}_{r,s}$ (so that
the first tree in $\wt{\mathcal{F}}^{(1)}_{r,s}$ is the tree of index $K$ in $\mathcal{F}^{(1)}_{r,s}$,
where $K$ is chosen uniformly over $\{1,2,\ldots, L_s\}$) and also
 ``forgetting'' the distinguished vertex at generation $s-r$. On the event
 $\{L_r=p\}\cap\{L_s=q\}$, $\wt{\mathcal{F}}^{(1)}_{r,s}$ is
 a random element of the set $\F''_{p,q,s-r}$ introduced before
 Lemma \ref{full-mass}. 

\begin{proposition}
\label{prop:compa}
There exists a constant $C_1$, which only depends on the real $a$, such that, for every sufficiently large integer $r$,
for every choice of $s \in \{r+1,r+2, \ldots \}$,
for every choice of the integers $p$ and $q$ with 
$ar^2< p\leq a^{-1}r^2$ and $ar^2< q\leq a^{-1}r^2$, for
every forest $\mathcal{F}\in \F''_{p,q,s-r}$, we have
\begin{equation}
\label{compabd}
\P\Big(\wt{\mathcal{F}}^{(1)}_{r,s}=\mathcal{F}\Big)
\leq C_1\,\P\Big( ([\tau_1]_{s-r},\ldots,[\tau_{N_r^{(a)}}]_{s-r})=\mathcal{F}\Big).
\end{equation}
\end{proposition}

\begin{proof}
By a standard formula for Galton--Watson trees, we have first
\begin{align}
\label{compatech}
\P\Big( ([\tau_1]_{s-r},\ldots,[\tau_{N_r^{(a)}}]_{s-r})=\mathcal{F}\Big)
&= \P(N_r^{(a)}=p)\,\P\Big( ([\tau_1]_{s-r},\ldots,[\tau_{p}]_{s-r})=\mathcal{F}\Big)\nonumber\\
&= \frac{1}{ \lfloor a^{-1}r^2\rfloor-\lfloor ar^2\rfloor}
\times \prod_{v\in\mathcal{F}^*} \theta(c_v),
\end{align}
using again the notation $\mathcal{F}^*$ for the set of all vertices of $\mathcal{F}$ at height strictly less than $s-r$. 
On the other hand, let $\mathcal{F}^\circ$ be any (pointed) forest in 
$\F_{p,q,s-r}$ such that the sequence of trees in $\mathcal{F}^\circ$ coincides with that in $\mathcal{F}$
up to a cyclic permutation. By the observations of the beginning of this section,
we know that, conditionally on $L_r=p$,   $B^\bullet_s(\t^{(1)}_\infty)\backslash B^\bullet_r(\t^{(1)}_\infty)$ is distributed 
 according to $\P_{p,s-r}$, and thus
$$\P\Big({\mathcal{F}}^{(1)}_{r,s}=\mathcal{F}^\circ\,\Big|\,
L_r=p\Big)= \mathbf{P}_{p,s-r}(\mathcal{F}^\circ)=\frac{12^{-q}{C}(q)}{12^{-p}{C}(p)} \prod_{v \in \mathcal{F}^*} \theta(c_{v}),$$
Note that the right-hand side of the previous display only depends on
 $\mathcal{F}$ and not on the choice of $\mathcal{F}^\circ$. By arguments similar to those 
 of the proof of Lemma \ref{full-mass}, we have then
 $$\P\Big(\wt{\mathcal{F}}^{(1)}_{r,s}=\mathcal{F}\,\Big|\,
L_r=p\Big)
=  \frac{p}{q}\,\P\Big({\mathcal{F}}^{(1)}_{r,s}=\mathcal{F}^\circ\,\Big|\,
L_r=p\Big)=\frac{h(q)}{h(p)}\,\prod_{v\in\mathcal{F}^*} \theta(c_v)
.$$
By our conditions on $p$ and $q$, the ratio $h(q)/h(p)$ is bounded above by a constant $C_2$ (depending on $a$).  Using the bound \eqref{localbound},
we thus get 
$$\P\Big(\wt{\mathcal{F}}^{(1)}_{r,s}=\mathcal{F}\Big) \leq \frac{C_0}{r^2}\;C_2 \,\prod_{v\in\mathcal{F}^*} \theta(c_v)
.$$
The bound of the proposition now follows by comparing the right-hand side of
the previous display with the right-hand side of \eqref{compatech}.
\end{proof}

\begin{proof}[Proof of Lemma \ref{bdlawperi}] 
We observe that 
$$\P(L_r=p)= \sum_{\mathcal{F}\in \F''_{1,p,r}} \P(\wt{\mathcal{F}}^{(1)}_{0,r}=\mathcal{F})
= \sum_{\mathcal{F}\in \F''_{1,p,r}} \frac{h(p)}{h(1)}\,\prod_{v\in\mathcal{F}^*} \theta(c_v),$$
where 
$\wt{\mathcal{F}}^{(1)}_{0,r}$ is defined as above from the skeleton ${\mathcal{F}}^{(1)}_{0,r}$ of the triangulation of the cylinder $B^\bullet_r(\t^{(1)}_\infty)$. Hence
\begin{equation}
 \label{bdlawperi1}
 \P(L_r=p)
= \frac{h(p)}{h(1)}\,\P_p(Y_{r}=1)
\end{equation} where, as previously, $(Y_n)_{n\geq 0}$ stands for a Galton--Watson process with offspring distribution $\theta$ that starts from $k$ under the probability measure $\mathbb{P}_k$. 
From the explicit form of $h$ there exists a constant $C_3$ such that
 \begin{eqnarray}h(p)\leq \frac{C_3}{ \sqrt{p}}, 
 \label{bound-h}
 \end{eqnarray}
for every $p \geq 1$. On the other hand, from \eqref{iterate-gtheta},
we have
 \begin{equation}
 \label{heighttree}
 \mathbb{P}_1(Y_r=0)= 1- (r+1)^{-2}.
 \end{equation} 
and 
 \begin{eqnarray}
\mathbb{P}_{p}(Y_{r} =1) &= &
\lim_{x\downarrow 0} \;x^{-1} \Big( \E_p[x^{Y_r}] - \mathbb{P}_p(Y_r=0)\Big)\nonumber\\
 &=&
 \lim_{x\downarrow 0} \;x^{-1} \Big( \Big(1 - \Big(r+\frac{1}{\sqrt{1-x}}\Big)^{-2}\Big)^p
 - \Big(1 - (r+1)^{-2}\Big)^p\Big)\nonumber\\
 &=&\frac{p}{(r+1)^3}\, (1-(r+1)^{-2})^{p-1}.\label{probab==1}
 \end{eqnarray}
 It follows that, with some constants $C_{4},C_{5}>0$, 
 $$ \mathbb{P}(L_{r}=p) \leq \frac{C_{3}}{h(1)}\frac{ \sqrt{p}}{(r+1)^3} (1-(r+1)^{-2})^{p-1} \leq \frac{C_{4}}{r^2} \sqrt{ \frac{p}{r^2}} e^{-(p-1)/(r+1)^2} \leq \frac{C_{5}}{r^2} \sqrt{ \frac{p}{r^2}} e^{-p/(4r^2)}.$$
The bound \eqref{localbound} immediately follows. As for \eqref{eq:boundLs}, we use the fact that the function $ x \mapsto \sqrt{x} e^{-x/4}$ is decreasing when $x \geq 2$ so that, if $\alpha \geq 2$, we have, with some constant $C_{6}$,
 $$ \mathbb{P}(L_{r} > \alpha r^2) \leq \sum_{q = \alpha r^2 +1}^\infty \frac{C_{5}}{r^2} \sqrt{\frac{q}{r^2}} e^{-q/(4r^2)} \leq  \frac{C_{5}}{r^2} \int_{\alpha r^2}^\infty \mathrm{d}x\, \sqrt{\frac{x}{r^2}} e^{-x/(4r^2)}   \leq C_{6} e^{-\alpha/5}. $$
\end{proof}

\begin{remark}
\label{conv-dist-L}
 The preceding calculations
give for every $x>0$,
$$ \mathbb{P}_{\lfloor xr^2\rfloor}(Y_{r} =1) \build{\sim}_{r\to\infty}^{} {\frac{x}{r}\;\exp(-x)},$$
and, noting that $\frac{h(p)}{h(1)}\sim 2/\sqrt{\pi p}$ as $p\to\infty$, we obtain that, for every $x>0$,
$$\lim_{r\to\infty} r^2\,\P(L_r=\lfloor xr^2\rfloor)= \frac{2}{\sqrt{\pi}}\,\sqrt{x}\,\exp(-x).$$
In this way we recover (in a stronger form) the fact that $r^{-2}L_r$ converges in distribution to a 
Gamma distribution with parameter $3/2$ \cite[Theorem 2]{CLGpeeling}.
\end{remark}

We can apply Proposition \ref{prop:compa} to get information on the probability 
of coalescence of left-most geodesics from distinct vertices of $\partial^*B^\bullet_n(\t^{(1)}_\infty)$
when $n$ is large.  Let $u^{(n)}_0$ be chosen uniformly at random 
over the vertices of $\partial^*B^\bullet_n(\t^{(1)}_\infty)$, and enumerate all vertices
of $\partial^*B^\bullet_n(\t^{(1)}_\infty)$ in clockwise order as $u^{(n)}_0,u^{(n)}_1,\ldots,u^{(n)}_{L_n-1}$.
Fix $\delta >0$. We claim that, if $\eta\in(0,1/2)$ is chosen small enough, the probability of the intersection of 
\begin{equation}
\label{goodevent}
\{an^2\leq L_n\leq a^{-1}n^2\} \cap \{an^2\leq L_{n-\lfloor \eta n\rfloor}\leq a^{-1}n^2\} 
\end{equation}
with the event where the left-most geodesics from $u^{(n)}_0$ and $u^{(n)}_{\lfloor a n^2/2\rfloor}$
coalesce before hitting $\partial^* B^\bullet_{n-\lfloor \eta n\rfloor}(\t^{(1)}_\infty)$ is bounded above by $\delta$ for all large enough $n$. 
Indeed, using the remarks at the end of subsection \ref{sec:skeleton}, we need to bound the probability of 
(the intersection of the event in \eqref{goodevent} with) the event where the height of the subforest consisting of the
first $\lfloor an^2/2\rfloor$ trees of $\wt{\mathcal{F}}^{(1)}_{n-\lfloor \eta n\rfloor,n}$ is strictly smaller than $\lfloor \eta n\rfloor$, or the 
same holds for the height of the complementary 
subforest. Proposition \ref{prop:compa} shows that, up to a multiplicative constant
depending only on $a$, this probability is bounded above by twice the probability that a forest of $\lfloor an^2/2\rfloor$
independent Galton--Watson trees with offspring distribution $\theta$ has height smaller than $\lfloor \eta n\rfloor$, and our claim
now follows from \eqref{heighttree}.

 \section{Half-plane models} \label{sec:half-planes}

{In this section, we introduce the two half-plane models that are local limits of large random 
triangulations of the cylinder rooted either on the bottom cycle (as previously) or
on the top cycle. In the first case, we get the upper half-plane model, which was already discussed for type II triangulations in \cite{Ang05} and, in the second case, we get the lower half-plane model. We then 
obtain a relation between these
two half-plane models (Proposition \ref{size-bias-UHPT}). The lower half-plane model is most relevant for our study --- the time constants $ \mathbf{c}_{0}, \mathbf{c}_{1}$ and $ \mathbf{c}_{2}$ of the Introduction will arise in an application of the ergodic subadditive theorem on this random lattice (Propositions \ref{lem:subadditive} and \ref{subadditive-dual}). However
some of the delicate estimates that we will need about the geometry of the lower half-plane model are easier to derive
first for the upper half-plane model and can then be transferred to the lower half-plane model using Proposition \ref{size-bias-UHPT} and  Corollary \ref{abso-con}.}

\subsection{The upper half-plane triangulation}
\label{subsec:UHPT}

We construct a triangulation of the upper half-plane $\R\times \R_+$, whose vertex set contains all
points of the form $(i,j)$ for $i\in\Z$ and $j\in \Z_+$. A key ingredient of this
construction is an infinite tree, which is closely related to the Galton--Watson tree with offspring
distribution $\theta$ conditioned on non-extinction. This tree will be embedded in the half-plane so that
its vertices are exactly all points of the form $(\frac{1}{2}+i,j)$ for $i\in\Z$ and $j\in \Z_+$. Let us start by describing this tree 
and its embedding.

The tree has an infinite {\it spine} which consists of all vertices of the discrete half-line $\{(\frac{1}{2},j),\;j\in\Z_+\}$, with an
edge between any two successive vertices on this half-line. Informally, one may think that time runs backwards when we move upward the spine,
so that the vertex $(\frac{1}{2},j)$ is the parent of the vertex $(\frac{1}{2},j-1)$ for every $j\geq 1$. Write $\ov\theta$
for the size-biased distribution associated with $\theta$, namely $\ov\theta(k)=k\theta(k)$ for every $k\geq 1$. Every vertex of the
form $(\frac{1}{2},j)$, $j\geq 1$ has (independently of the others) a random number $m_j$ of children distributed according 
to $\ov\theta$, and these are the vertices $(\frac{1}{2}+k,j-1)$ for $\ell_j-m_j\leq k\leq \ell_j-1$, where 
$\ell_j$ is uniform over $\{1,2,\ldots,m_j\}$ (put differently, the rank of $(\frac{1}{2},j-1)$ among the children
of $(\frac{1}{2},j)$ is uniform). Of course the pairs $(\ell_j,m_j)$, $j\in\Z_+$, are assumed to be independent.

Then, to every vertex of the form $(\frac{1}{2}+k,j)$ ($k\not=0$) that is a child of a vertex 
of the spine,
we attach (independently and independently of $(\ell_i,m_i)_{i\in\Z_+}$) 
a
Galton--Watson tree with offspring distribution $\theta$ truncated at height $j$,  in such a way that
vertices at height $r\in\{0,1,\ldots,j\}$ in this truncated tree will be points of the form $(\frac{1}{2}+i,j-r)$. An easy calculation shows 
that on both sides of the spine infinitely many of these trees will hit the maximal possible height. It follows that
we may draw these trees in the upper half-plane, in such a way that edges do not cross and
vertices (including those of the spine) are exactly all points of the form $(\frac{1}{2}+i,j)$ for $i\in\Z$ and $j\in \Z_+$. In particular, 
every vertex $(\frac{1}{2}+i,j)$ with 
$i\not =0$ is a descendant of some vertex of the spine. 
Furthermore this embedding is unique.
Rather than giving a 
more formal construction, we refer the reader to Fig.~\ref{fig:defUHPT} from which the definition of
our infinite tree should be clear. 

\begin{figure}[!h]
 \begin{center}
 \includegraphics[width=0.8\linewidth]{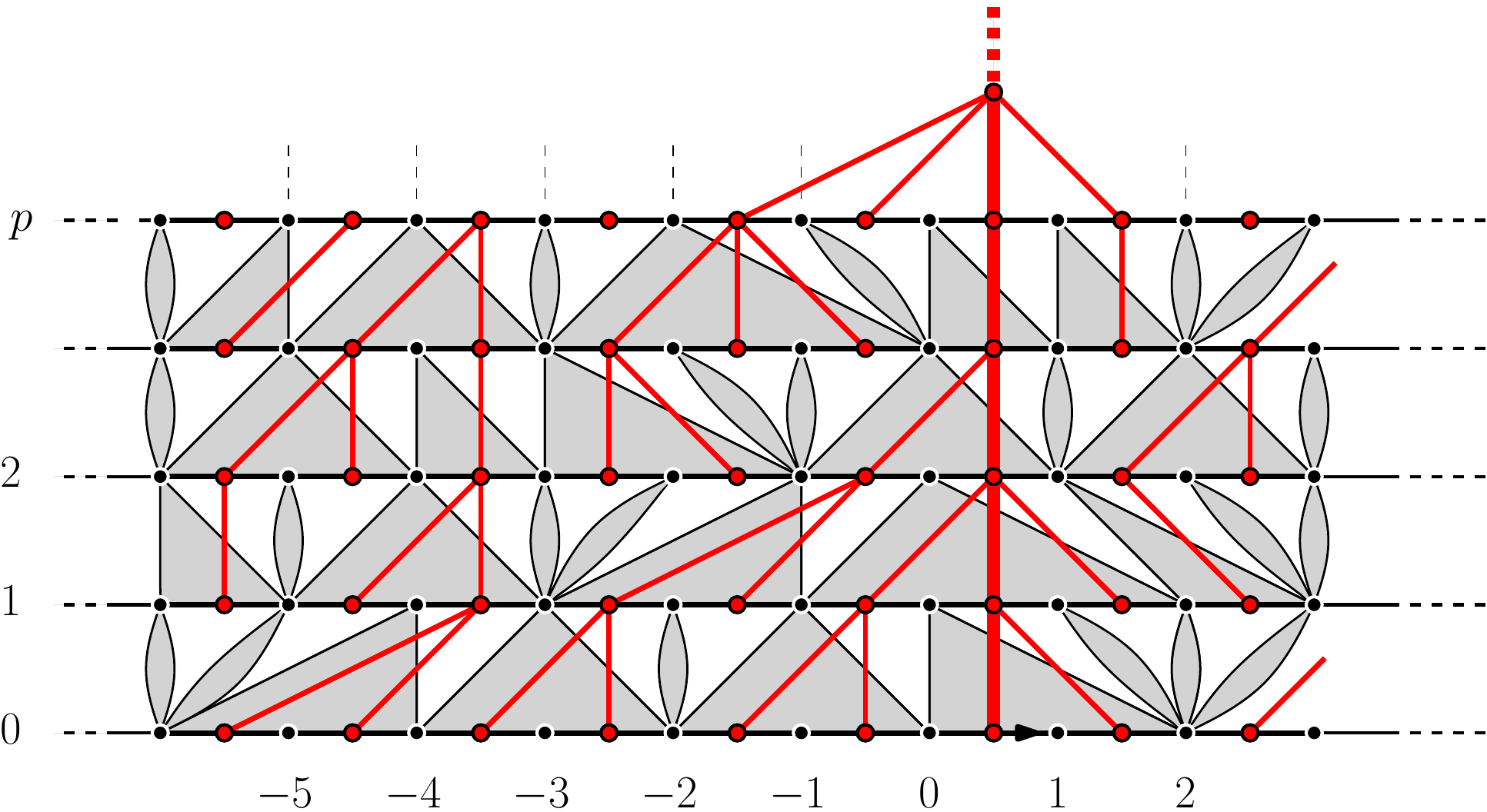}
 \caption{ \label{fig:defUHPT} Illustration of the construction of the upper half-plane triangulation. In red, the underlying tree giving the ``skeleton'' of the map and in grey, the slots to be filled in by independent Boltzmann triangulations. The thick red line represents the spine of the infinite tree.}
 \end{center}
 \end{figure}

Let us now construct our infinite triangulation of the upper half-plane. We start by constructing special triangles, which
we call the downward triangles by analogy with the previous sections, whose vertices will be elements of
$\Z\times \Z_+$. We first require that, for every $(i,j)\in\Z\times \Z_+$, the horizontal edge $[i,i+1]\times\{j\}$ connecting $(i,j)$ to $(i+1,j)$ is an edge
of the triangulation. For every such horizontal edge with $j\geq 1$, we construct a {\it downward triangle}
containing this edge, whose third vertex is the vertex $(k,j-1)$, where $k$ is the 
minimal integer such that the tree vertex $(\frac{1}{2}+k,j-1)$ is a child of $(\frac{1}{2}+i',j)$ for some $i'> i$.
We do this construction in such a way that edges are all distinct (and of course do not cross). Note in
particular that if a (tree) vertex $(\frac{1}{2} + i,j)$ with $j\geq 1$ has no child, there will be a double edge 
from $(i,j)$ to $(k,j-1)$ for some $k\in\Z$.
The configuration of downward triangles is then completely determined by the infinite tree.
As in the previous sections, the complement of the union of downward triangles in the half-plane consists of
slots, each vertex of the form $(i,j)$, $j\geq 1$, being at the ``top'' of a slot 
bounded by a cycle whose length is $2$ plus the number of children of $(\frac{1}{2}+i,j)$ in the infinite tree. 
We then fill in the slots by independent Boltzmann triangulations with the prescribed perimeters, and, as
previously, we make the convention that if a slot with perimeter $2$ is filled in by the edge-triangulation, 
this means that the double edge bounding this slot is glued into a single edge.
The resulting triangulation of the upper half-plane in called the UHPT for Upper Half-Plane Triangulation. 
It is rooted at the edge between $(0,0)$ and $(1,0)$, which is oriented from left to
right. We write $\u$ for the UHPT and $\partial \u$ for its (bottom) boundary.

\begin{proposition}
\label{conv-UHPT}
We have
$$\t^{(p)}_\infty \build{\la}_{p\to\infty}^{\rm(d)} \u,$$
in the sense of local limits of rooted planar maps.
\end{proposition}
\begin{remark} In \cite{Ang05}, Angel uses a similar local convergence
to define the type II Uniform Infinite Half Planar Triangulation. Our approach is however different from the peeling construction given in \cite{Ang05,ACpercopeel}.\end{remark} 

\begin{proof} We first observe that, for every fixed $r\geq1$ and $j\geq 1$,
\begin{equation}
\label{asymp-perime}
\P(L^{(p)}_r =  j) \build{\la}_{p\to\infty}^{} 0.
\end{equation}
Indeed, the same arguments as in the proof of Lemma \ref{bdlawperi}
give
$$\P(L^{(p)}_r =  j)=\frac{h(j)}{h(p)}\mathbb{P}_j(Y_r=p)\leq \frac{h(j)}{h(p)}\,\frac{1}{p} \mathbb{E}_j[Y_r]= \frac{j\,h(j)}{p\,h(p)},$$
yielding the desired result since $p\,h(p)\la \infty$ as $p\to\infty$. 

In order to prove the proposition, it is enough to prove that, for every $r\geq 1$, if
$\mathcal{B}_r(\u)$, respectively $\mathcal{B}_r(\t^{(p)}_\infty)$, denotes the
planar map obtained by keeping only the faces of $\u$, resp. of $\t^{(p)}_\infty$,
that are incident to a vertex at graph distance strictly less than $r$ from the root vertex,
we have
\begin{equation}
\label{conv-balls}
\P(\mathcal{B}_r(\t^{(p)}_\infty)=A)
\build{\la}_{p\to\infty}^{}\P(\mathcal{B}_r(\u)=A),
\end{equation}
for any rooted planar map $A$. To get this convergence,
fix $r\geq 1$ and write $\mathcal{F}^{(p)}_{0,r}=(\mathscr{T}^{(p)}_0,\mathscr{T}^{(p)}_1,\ldots,\tc^{(p)}_{L^{(p)}_r-1})$ for the skeleton  of $B^\bullet_r(\t^{(p)}_\infty)$. 
We will prove that, for every $k\geq 1$, if $(\tau_{-k},\ldots,\tau_0,\ldots,\tau_k)$ is a finite
collection of plane trees having maximal height $r$ and a distinguished vertex at height $r$
that belongs to $\tau_0$,
\begin{align}
\label{conv-UHPT-tech}
&\P(\{\tc^{(p)}_{L^{(p)}_r-k}=\tau_{-k},\ldots,\tc^{(p)}_{L^{(p)}_r-1}=\tau_{-1},\tc^{(p)}_{0}=\tau_0,\tc^{(p)}_1=\tau_1,\ldots, \tc^{(p)}_k=\tau_k\}
\cap \{L^{(p)}_r \geq 2k+1\})\nonumber\\
&\quad \build{\la}_{p\to\infty}^{} \P(\Gamma_{(-k,r)}=\tau_{-k},\ldots,\Gamma_{(0,r)}=\tau_0,\ldots,\Gamma_{(k,r)}=\tau_k)
\end{align}
where $\Gamma_{(i,j)}$ stands for the subtree of descendants of $(\frac{1}{2}+i,j)$ in the infinite tree (here we view 
$\Gamma_{(i,j)}$ as an abstract plane tree, and we ``forget'' the embedding in the plane), and 
it is understood that $\Gamma_{(0,r)}$ has a distinguished vertex corresponding to $(\frac{1}{2},0)$, 
so that when we write the equalities $\tc^{(p)}_0=\tau_0$ or $\Gamma_{(0,r)}=\tau_0$, we 
mean an equality of pointed trees. We observe that,  if $k$ is large, we can find a collection $\mathbf{F}_k$
of forests $(\tau_{-k},\ldots,\tau_0,\ldots,\tau_k)$ such that the probability of the event $\{(\Gamma_{(-k,r)},\ldots,\Gamma_{(k,r)})\in\mathbf{F}_k\}$
is close to $1$, and, on the latter event,
the ball $\mathcal{B}_r(\u)$ is a deterministic function of the trees
$\Gamma_{(-k,r)},\ldots,\Gamma_{(k,r)}$ and of the triangulations with a boundary filling in the 
slots associated with the vertices of these trees. Similarly, on the event $\{(\tc^{(p)}_{L^{(p)}_r-k},\ldots,\tc^{(p)}_{L^{(p)}_r-1},\tc^{(p)}_0,\ldots,\tc^{(p)}_k)\in\mathbf{F}_k\}\cap\{L^{(p)}_r\geq 2k+1\}$, 
the ball $\mathcal{B}_r(\t^{(p)}_\infty)$ will be the same deterministic function of the trees
$\tc^{(p)}_{L^{(p)}_r-k},\ldots,\tc^{(p)}_{L^{(p)}_r-1},\tc^{(p)}_0,\ldots,\tc^{(p)}_k$ and of the associated triangulations
with a boundary. The desired convergence \eqref{conv-balls} thus follows from \eqref{conv-UHPT-tech}, using also
\eqref{asymp-perime}. 

It remains to prove \eqref{conv-UHPT-tech} and, to this end, we fix a forest $(\tau_{-k},\ldots,\tau_k)$ satisfying the assumptions stated above.
We first note that, if 
$\mathsf{V}^*(\tau_i)$ stands for the collection of all vertices of $\tau_i$
at height strictly less than $r$, we have
\begin{equation}
\label{UPHT-lem2-t1}
\P(\Gamma_{(-k,r)}=\tau_{-k},\ldots,\Gamma_{(0,r)}=\tau_0,\ldots,\Gamma_{(k,r)}=\tau_k)
=\prod_{v\in \mathsf{V}^*(\tau_{-k})\cup\cdots\cup \mathsf{V}^*(\tau_k)}
\theta(c_v),
\end{equation}
where $c_v$ denotes the number of children of $v$. The preceding equality holds 
because by construction the trees $\Gamma_{(i,r)}$, $i\not =0$ are independent
Galton--Watson trees with offspring distribution $\theta$ truncated at height $r$, and
the tree $\Gamma_{(0,r)}$ is a size-biased Galton--Watson tree with offspring distribution $\theta$ 
truncated at height $r$ and
given with a distinguished vertex at height $r$. See \cite{LPP95b}
for the definition and properties of size-biased Galton--Watson trees, noting that, if we ``forget'' the distinguished vertex, the right-hand side of the
preceding formula has an extra multiplicative factor equal to the size of
generation $r$ in $\tau_0$.

Consider then the left-hand side of \eqref{conv-UHPT-tech}. To simplify notation, write
$\f_{(k)}$ for the forest $(\tau_{-k},\ldots,\tau_k)$ and $\mathrm{m}_k$ for the number of vertices of
$\f_{(k)}$ at generation $r$. For any 
forest $\mathcal{F}=(\sigma_0,\ldots,\sigma_\ell)\in \F_{p,\ell,r}$, with $\ell\geq 2k+1$, write
$\Phi_k(\mathcal{F})= (\sigma_{\ell-k}, \ldots,\sigma_{\ell-1},\sigma_0,\ldots,\sigma_k)$ where it is understood that,
in $\Phi_k(\mathcal{F})$ as in $\mathcal{F}$, $\sigma_0$ comes with a distinguished vertex at height $r$.
Then, using \eqref{law-forest} and the fact that the law of $B^\bullet_r(\t^{(p)}_\infty)$ is $\mathbf{P}_{p,r}$, 
we can rewrite the left-hand side of \eqref{conv-UHPT-tech} as
\begin{align}
\label{UPHT-lem2-t2}
&\sum_{\ell=2k+1}^\infty\  \sum_{\mathcal{F}\in \F_{p,\ell,r}: \Phi_k(\mathcal{F})=\mathcal{F}_{(k)}}\ 
\frac{12^{-\ell}C(\ell)}{12^{-p}C(p)} \prod_{v\in \f^*} \theta(c_v)
= \Bigg(\prod_{v\in \mathsf{V}^*(\tau_{-k})\cup\cdots\cup \mathsf{V}^*(\tau_k)}
\theta(c_v)\Bigg) \nonumber\\
&\qquad\times
\Bigg(\sum_{\ell=2k+1}^\infty  \frac{12^{-\ell}C(\ell)}{12^{-p}C(p)} \build{
\sum_{\sigma_{k+1},\sigma_{k+2},\ldots,\sigma_{\ell-k-1}}}_{\#\sigma_{k+1}(r)+\cdots+\#\sigma_{\ell-k-1}(r)=p-\mathrm{m}_k}^{}
\prod_{v\in \mathsf{V}^*(\sigma_{k+1})\cup\cdots\cup \mathsf{V}^*(\sigma_{\ell-k-1})} \theta(c_v)\Bigg),
\end{align}
where the second
sum in the last line is over all choices of the plane trees $\sigma_{k+1},\sigma_{k+2},\ldots,\sigma_{\ell-k-1}$
having a total number of vertices at height $r$ equal to $p-\mathrm{m}_k$. Set $\varphi(\ell)
=12^{-\ell}C(\ell)$ to simplify notation, and write $A_p$ for the quantity inside parentheses in the second line
of \eqref{UPHT-lem2-t2}. Then we have 
\begin{align*}
A_p=\sum_{\ell=2k+1}^\infty  \frac{\varphi(\ell)}{\varphi(p)} \; \mathbb{P}_{\ell-(2k+1)}(Y_r=p-\mathrm{m}_k)
&=\sum_{\ell=0}^\infty \frac{\varphi(\ell +2k+1)}{\varphi(p)}\; \mathbb{P}_{\ell}(Y_r=p-\mathrm{m}_k)\\
&\geq \sum_{\ell=0}^\infty \frac{\varphi(\ell)}{\varphi(p)}\; \mathbb{P}_{\ell}(Y_r=p-\mathrm{m}_k),
\end{align*}
since $\varphi$ is monotone increasing.
Fix $\ve\in(0,1/2)$. Recalling that $\varphi(\ell)=\ell\, h(\ell)$, we have then
\begin{equation}
\label{UHPT-t3}
A_p \geq (1-\ve)\,\sum_{\ell=\lfloor (1-\ve)p\rfloor+1}^\infty \frac{h(\ell)}{h(p)}\;\mathbb{P}_{\ell}(Y_r=p-\mathrm{m}_k).
\end{equation}
On the other hand, we claim that
\begin{equation}
\label{claimUHPT}
\lim_{p\to \infty} \sum_{\ell=0}^{\lfloor (1-\ve)p\rfloor} \frac{h(\ell)}{h(p)}\;\mathbb{P}_{\ell}(Y_r=p-\mathrm{m}_k) = 0.
\end{equation}
To see this, observe that the law of $Y_r-\ell$ under $\P_\ell$ is the law of the sum of 
$\ell$ centered i.i.d. random variables whose tail asymptotics are given by \eqref{asymp-giterate}. As a consequence of
\cite[Corollary 2.1]{DDS08}, there exist constants $C_\ve$ and $C'_\ve$ such that, for every sufficiently large $p$
and every $\ell\in\{0,1,\ldots, \lfloor (1-\ve)p\rfloor\}$, 
$$\mathbb{P}_{\ell}(Y_r=p-\mathrm{m}_k) \leq C_\ve\,\ell\, \mathbb{P}_{1}(Y_r=p-\ell +1-\mathrm{m}_k) \leq C'_\ve\,\ell\, p^{-5/2},$$
using  \eqref{asymp-giterate} in the last bound (to be precise,  \cite[Corollary 2.1]{DDS08} gives this only for  ``large''
values
$\ell\geq \ell_0$ for some integer $\ell_0$, but the values $\ell\leq \ell_0$ are easily treated by a direct argument). 
Since $h(k)\sim 1/\sqrt{\pi k}$ as $k\to \infty$, we then get, for all sufficiently large $p$,
$$\sum_{\ell=0}^{\lfloor (1-\ve)p\rfloor} \frac{h(\ell)}{h(p)}\;\mathbb{P}_{\ell}(Y_r=p-\mathrm{m}_k)
\leq C''_\ve \,p^{-2}\sum_{\ell=0}^{\lfloor (1-\ve)p\rfloor} \ell^{1/2}$$
which tends to $0$ as $p\to\infty$, proving our claim \eqref{claimUHPT}.

Using \eqref{UHPT-t3} and \eqref{claimUHPT}, we have then
$$\liminf_{p\to\infty} A_p \geq (1-\ve)\,\liminf_{p\to\infty} \sum_{\ell=0}^{\infty} \frac{h(\ell)}{h(p)}\;\mathbb{P}_{\ell}(Y_r=p-\mathrm{m}_k) 
= (1-\ve) \liminf_{p\to\infty} \frac{h(p-\mathrm{m}_k)}{h(p)},$$
by \eqref{statio}. Since $h(p-m_k)/h(p)$ tends to $1$ and $\ve$ was arbitrary, we have indeed proved that
$$\liminf_{p\to\infty} A_p \geq 1.$$
From \eqref{UPHT-lem2-t1} and \eqref{UPHT-lem2-t2}, we get that the liminf of the quantities in the left-hand side of \eqref{conv-UHPT-tech} is greater than or equal 
to the right-hand side, for any choice of the forest $(\tau_{-k},\ldots,\tau_k)$. On the other hand the sum of the
quantities in the right-hand side over possible choices of $(\tau_{-k},\ldots,\tau_k)$ is equal to $1$. It follows
that the convergence \eqref{conv-UHPT-tech} holds. This completes the proof  of the proposition.
\end{proof}

\subsection{The lower half-plane triangulation}

We now discuss the lower half-plane triangulation or LHPT, which can be obtained as the local limit
in distribution of the hulls $B^\bullet_r(\t^{(p)}_\infty)$ when $r\to\infty$, provided that these hulls are 
re-rooted at an edge chosen uniformly on the top cycle. 

The construction of the LHPT is similar to that of the UHPT in the previous section. The vertex set now contains all
points of $\Z\times\Z_-$, and the role of the infinite tree is
played by a doubly infinite sequence $(\tc_i)_{i\in\Z}$ of independent Galton--Watson trees 
with offspring distribution $\theta$. These trees are then embedded in the lower half-plane 
so that the root of $\tc_i$ is $(\frac{1}{2}+i,0)$ for every $i\in \Z$, and the collection of all vertices of the trees
is exactly the set of all points of the form $(\frac{1}{2}+i,j)$, where $(i,j)\in\Z\times \Z_-$
(vertices at height $k$ in a tree being of the form $(\frac{1}{2}+i,-k)$). Here there are several ways of doing this embedding, but 
for definiteness we may agree that the collection of vertices of the trees $\tc_i$ for $i\geq 0$
is $(\frac{1}{2}+\Z_+)\times \Z_-$. See Fig.~\ref{fig:LHPT}. 

Given the trees, the downward triangles of the LHPT are constructed in a very similar way
to what was done in the previous section. 
For every
horizontal edge $[i,i+1]\times\{j\}$ connecting $(i,j)$ to $(i+1,j)$, where $i\in\Z$ and $j\in\Z_-$, we construct a {downward triangle}
containing this edge whose third vertex is the vertex $(k,j-1)$, where $k$ is the 
minimal integer such that $(\frac{1}{2}+k,j-1)$ is a child of $(\frac{1}{2}+i',j)$ for some $i'> i$. We then fill in the slots left by
the downward triangles by independent Boltzmann triangulations with a boundary to get the
LHPT, which is denoted by $\l$. By convention  $\l$ is rooted at the
edge between $(0,0)$ and $(1,0)$, which is oriented from left to
right.  

\begin{figure}[!h]
 \begin{center}
 \includegraphics[width=0.9\linewidth]{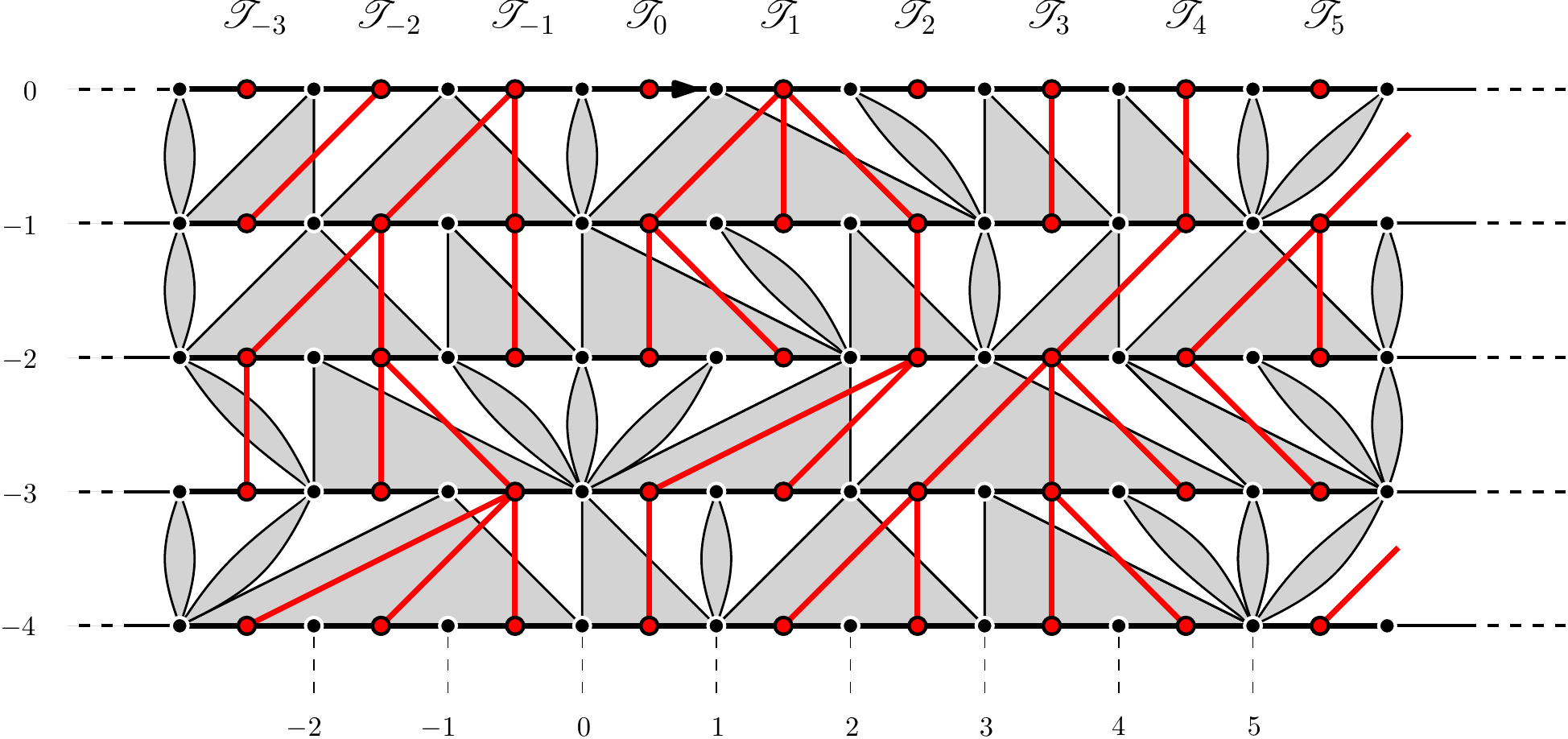}
 \caption{\label{fig:LHPT}Illustration of the construction of the LHPT.}
 \end{center}
 \end{figure}

\begin{proposition}
\label{conv-LHPT}
Let $p\geq 1$, and, for every $r\geq 1$, let $\wt B^\bullet_r(\t^{(p)}_\infty)$ stand for the hull $B^\bullet_r(\t^{(p)}_\infty)$
re-rooted at an edge chosen uniformly at random on the top boundary and oriented so that the top face is lying on its left-hand side. Then,
$$\wt B^\bullet_r(\t^{(p)}_\infty) \xrightarrow[r\to\infty]{ \mathrm{(d)}} \l,$$
in the sense of local limits of rooted planar maps.
\end{proposition}

This can be proved by arguments very similar to the proof of Proposition \ref{conv-UHPT}. We omit the details
as this statement is not needed in what follows.

Let us now discuss the connections between $\u$ and $\l$. For every integer $r\geq 1$, we let 
$\u_{[0,r]}$ stand for the (infinite) rooted planar map obtained by keeping only the first $r$ layers of $\u$. 
More precisely, in our construction of the UHPT we keep only those vertices and edges that lie in the strip
$\R\times [0,r]$. Alternatively, we may view $\u_{[0,r]}$ as the hull of radius $r$, corresponding to
distances from the bottom boundary. Similarly, we write $\l_{[0,r]}$ for the rooted planar map obtained by
keeping only the first $r$ layers of $\l$, that is, the part of $\l$ lying in the strip $\R\times [-r,0]$. We will also use the notation $\u_r$, resp. $\l_r$, 
for the horizontal line $\u_r=\{(i,r):i\in\Z\}$, resp. $\l_r=\{(i,-r):i\in\Z\}$. 

One may expect that the two infinite planar maps $\u_{[0,r]}$ and $\l_{[0,r]}$ are
closely related, and that informally $\l_{[0,r]}$ should correspond to $\u_{[0,r]}$ re-rooted at an edge of
its upper boundary. To give a precise statement, we introduce some notation.

We fix $r\geq 1$ and recall our notation $\Gamma_{(i,r)}$ for the subtree of descendants of $(\frac{1}{2}+i,r)$ in the infinite tree 
associated with $\u$. We already noticed that the trees $\Gamma_{(i,r)}$, $i\not =0$ are independent
Galton--Watson trees with offspring distribution $\theta$ truncated at height $r$, and
the tree $\Gamma_{(0,r)}$ is a size-biased Galton--Watson tree with offspring distribution $\theta$ 
truncated at height $r$ and
given with a {uniform} distinguished vertex at height $r$. 
For every $i\in\Z$, let $\Gamma_{(i,r)}(r)$ stand for the set of vertices of $\Gamma_{(i,r)}$ at height $r$. Also let
$K_r\geq 1$ be the first index $i\geq 1$ such that $\Gamma_{(i,r)}(r)\not=\varnothing$. See Fig.~\ref{fig:re-rooting}. 

\begin{figure}[!h]
 \begin{center}
 \includegraphics[width=0.8\linewidth]{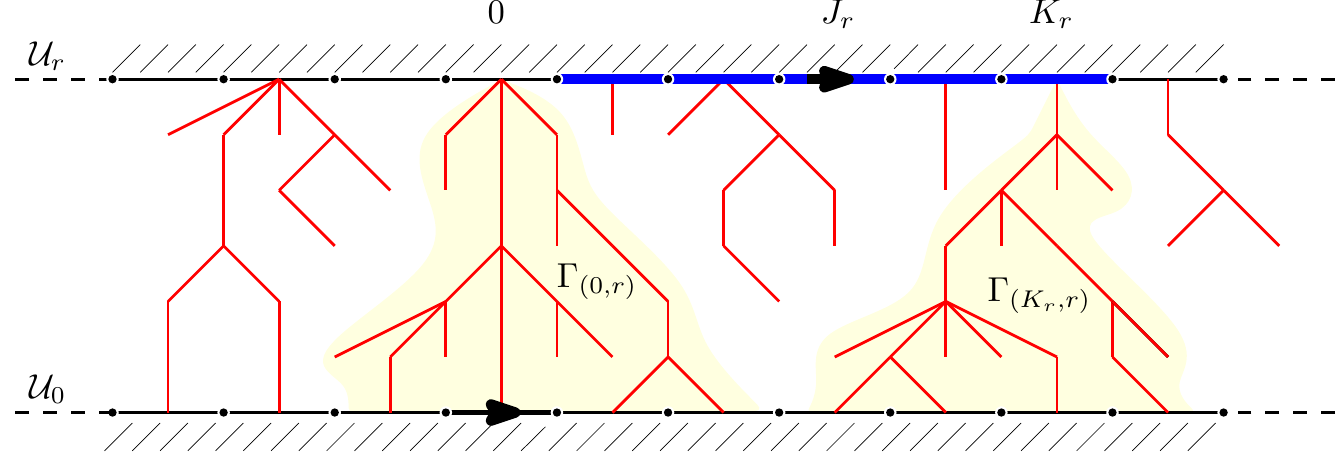}
 \caption{\label{fig:re-rooting} From $ \mathcal{U}_{[0,r]}$ to $ \widetilde{\mathcal{U}}_{[0,r]}$. The thick part of the  line 
 $\mathcal{U}_r$ corresponds to the possible edges at which the map can be re-rooted.}
 \end{center}
 \end{figure}

We also let $i_r< 0$ be the largest integer $i< 0$ such that $\tc_i$ has height
at least $r$. As previously, let $\tc_{i_r}(r)$ be the set of all vertices of $\tc_{i_r}$
at height $r$.

\begin{proposition}
\label{size-bias-UHPT}
Let $\wt \u_{[0,r]}$ stand for the infinite rooted planar map obtained by re-rooting 
$\u_{[0,r]}$ so that the root edge is the horizontal edge from $(J_r,r)$ to $(J_r+1,r)$, where the index
$J_r$ is chosen uniformly at random in $\{1,\ldots,K_r\}$. Then, for any nonnegative measurable function $f$,
$$\E[K_r\;f(\wt\u_{[0,r]})] = \E[\#\tc_{i_r}(r)\;f(\l_{[0,r]})].$$
\end{proposition}

\proof 
Clearly, it is enough to verify that the distribution of the configuration of downward triangles is the same for 
$\wt\u_{[0,r]}$, under the measure having density $K_r$ with respect to $\P$, and for  $\l_{[0,r]}$,
under the measure having density $\#\tc_{i_r}(r)$ with respect to $\P$. In both cases,
the configuration of downward triangles is coded by a doubly infinite sequence of trees, and we
need to verify that these two sequences have the same distribution. By construction, the sequence of trees
associated with $\wt\u_{[0,r]}$ is $(\Gamma'_i)_{i\in \Z}$, where $\Gamma'_i=\Gamma_{(i+J_r,r)}$ if $i\not =-J_r$, and 
$\Gamma'_{-J_r}=\Gamma^{\rm unp}_{(0,r)}$, where we use the notation 
$\Gamma^{\rm unp}_{(0,r)}$ to represent the tree $\Gamma_{(0,r)}$ without its distinguished vertex at height $r$. Recall also that the trees $(\tc_i)_{i\in\Z}$ coding the configuration of downward triangles in $\l$
are just independent Galton--Watson trees with offspring distribution $\theta$.

Fix four integers $k_1,k_2,\ell_1,\ell_2$ such that $0< k_1\leq \ell_1$ and $0\leq k_2\leq \ell_2$. 
Consider a finite sequence $( \tau_{-\ell_{1}},\tau_{-\ell_1+1},\ldots ,  \tau_{\ell_{2}})$ of $\ell_{1}+\ell_{2}+1$ trees having height less than or equal to $r$ and such that the trees $  \tau_{i}$ for $-k_{1}<i<k_{2}$ have height strictly less than $r$, whereas $\tau_{-k_{1}}$ and $\tau_{k_{2}}$ have height $r$. By construction, we have
$$\{\Gamma'_{-\ell_1}=\tau_{-\ell_1},\ldots, \Gamma'_{\ell_2}=\tau_{\ell_2}\} = \{\Gamma_{(-\ell_1+k_1,r)}=\tau_{-\ell_1},\ldots, \Gamma^{\rm unp}_{(0,r)}=\tau_{-k_1},\ldots, 
\Gamma_{(\ell_2+k_1,r)}=\tau_{\ell_2}\} \cap \{J_{r}=k_{1}\}.$$ 
Since we have $K_r=k_1+k_2$ on the  first event in the right-hand side, it follows that 
\begin{eqnarray*} \P(\Gamma'_{-\ell_1}=\tau_{-\ell_1},\ldots, \Gamma'_{\ell_2}=\tau_{\ell_2})
&=&\frac{1}{k_1+k_2}\,\P(\Gamma_{(-\ell_1+k_1,r)}=\tau_{-\ell_1},\ldots, \Gamma^{\rm unp}_{(0,r)}=\tau_{-k_1},\ldots, 
\Gamma_{(\ell_2+k_1,r)}=\tau_{\ell_2}) \\
&=& \frac{\#\tau_{-k_1}(r)}{k_1+k_2}\,\P([\tc_{-\ell_1}]_r=\tau_{-\ell_1},\ldots, [\tc_{\ell_2}]_r=\tau_{\ell_2})  \end{eqnarray*}
where $[\tc_i]_r$ denotes the tree $\tc_i$ truncated at height $r$, and we used the fact that $\Gamma^{\rm unp}_{(0,r)}$ is a size-biased Galton--Watson tree with offspring distribution $\theta$ truncated at height $r$. The statement of the proposition follows since $K_r=k_1+k_2$  on $\{\Gamma'_{-\ell_1}=\tau_{-\ell_1},\ldots, \Gamma'_{\ell_2}=\tau_{\ell_2}\}$
and $i_r=-k_1$ on $\{[\tc_{-\ell_1}]_r=\tau_{-\ell_1},\ldots, [\tc_{\ell_2}]_r=\tau_{\ell_2}\}$. 
\endproof

\begin{corollary}
\label{abso-con}
For every $\ve>0$, we can choose $\delta>0$ small enough, so that for every 
$r\geq 1$, for every measurable set $A$, the property $\P(\wt\u_{[0,r]} \in A)\leq\delta$
implies $\P(\l_{[0,r]}\in A)\leq\ve$. 
\end{corollary}

\proof Since $\tc_{i_r}$ is just a Galton--Watson tree 
with offspring distribution $\theta$ conditioned on non-extinction at generation $r$, the generating function 
of $\#\tc_{i_r}(r)$ is derived from \eqref{iterate-gtheta} and \eqref{heighttree},
$$\E[x^{\#\tc_{i_r}(r)}] = (r+1)^{2}\Big( (r+1)^{-2}- \Big(r+\frac{1}{\sqrt{1-x}}\Big)^{-2}\Big),\qquad x\in [0,1].$$
From this, it is elementary to verify that $(r+1)^{-2}\#\tc_{i_r}(r)$ converges in distribution to a 
random variable $U$ with Laplace transform $\E[e^{-\lambda U}] = 1-(1+\lambda^{-1/2})^{-2}$. Since $U>0$ a.s.,
we can find $\eta >0$ such that $\P(\#\tc_{i_r}(r) < \eta (r+1)^2)\leq \ve/2$ for every $r\geq1$.

We take $\delta= \eta^2\ve^2/16$ and consider a measurable set $A$ such that $\P(\wt\u_{[0,r]} \in A)\leq\delta$.
By Proposition \ref{size-bias-UHPT}, we have for every ${r}\geq 1$,
\begin{equation}
\label{dist-bd-tech}
\E[K_r\,\mathbf{1}_{A}(\wt\u_{[0,r]})] = \E[\#\tc_{i_r}(r)\,\mathbf{1}_{A}(\l_{[0,r]})],
\end{equation}
and \eqref{heighttree} shows that the distribution of $K_r$ is given by $\P(K_r\geq j)=(1-(r+1)^{-2})^{j-1}$ for every $j\geq 1$. Straightforward
calculations then give $ \mathbb{E}[(K_r)^2]\leq 4\,(r+1)^4$ and the Cauchy--Schwarz inequality implies that the
left-hand side of \eqref{dist-bd-tech} is bounded above by
$$\E[(K_r)^2]^{1/2}\,\P(\wt\u_{[0,r]}\in A)^{1/2} \leq \frac{\eta\ve}{2}\;(r+1)^2.$$
The right-hand side of \eqref{dist-bd-tech} is bounded below by
$$\eta\, (r+1)^2\,\P(\{\l_{[0,r]}\in A\}\cap \{\#\tc_{i_r}(r) \geq \eta (r+1)^2\}).$$
By combining the preceding two bounds we get
$$\P(\{\l_{[0,r]}\in A\}\cap \{\#\tc_{i_r}(r) \geq \eta (r+1)^2\})\leq \frac{\ve}{2}.$$
By our choice of $\eta$, we also know that
$\P(\#\tc_{i_r}(r) < \eta (r+1)^2)\leq \ve/2$, and we get $\P(\l_{[0,r]}\in A)\leq \ve$. This 
completes the proof. \endproof

 \section{Estimates for distances along the boundary}
 
{ In this section, we derive asymptotic estimates for distances on the boundary of the UHPT or of the LHPT, which roughly say that the graph distance (with respect to the half-plane triangulation) between boundary vertices grows like the square root of their distance along the boundary.  We are interested in the case of the LHPT for future
 applications, but, for technical reasons, we start with the case of the UHPT. To derive these estimates in Section \ref{sec:distanceUHPT}, we first study the layers of balls centered at the root vertex of the UHPT, in the spirit of Section \ref{sec:skeleton}.}
 
 \subsection{Layers of balls in the UHPT}
 
 Let $r\geq 1$ be an integer, which will be fixed throughout this subsection. The ball  $\mathcal{B}_r(\u)$ is defined
 as the union of all 
 triangles of $\u$ which are incident to a vertex at graph distance smaller than or equal to $r-1$ from the root vertex, and
the
 hull $\mathcal{B}^\bullet_r(\u)$ is the complement of the unique infinite component of the complement of $\mathcal{B}_r(\u)$. 
 Then $\mathcal{B}^\bullet_r(\u)$ is a triangulation with a simple boundary consisting of (finitely many) edges on the 
 boundary of $\u$, including the root edge, and a simple path formed by non-boundary edges of $\u$ that connects the two
 extreme vertices of $\mathcal{B}_r(\u)$ lying on the boundary of $\u$. It will be useful to keep the information given by these two
 extreme vertices. So we view $\mathcal{B}^\bullet_r(\u)$ as a triangulation with a simple boundary, given with two 
 distinguished vertices on the boundary, which are distinct and distinct from the root vertex.
 
 \begin{lemma}
 \label{law-hull-UHPT}
 Let $A$ be a triangulation with a boundary, given with two distinct distinguished vertices on the boundary other than the root vertex. We write $\wt\partial A$ for the part of $\partial A$ that consists of the path between
 the two distinguished vertices that contains the root edge.
 Assume that $\P(\mathcal{B}^\bullet_r(\u)=A)>0$. Let $m\geq 2$ be the number of
 edges of $\wt\partial A$ 
 and  let $q\geq 1$ be the number of edges of $\partial A\backslash \wt\partial A$. Also let $N\geq 0$ be the number of vertices of $A$ that do not belong to $\wt\partial A$. Then,
 $$\P(\mathcal{B}^\bullet_r(\u)=A) = 12^{q-m}\,(12\sqrt{3})^{-N}.$$
 \end{lemma}
 
 \proof
 If $A'$ is another triangulation with a boundary, we use the notation $A\sqsubset A'$ to mean that $A$
 can be obtained as a subtriangulation of $A'$ with root edges coinciding and in such a way that
$\wt\partial A$
 is part of $\partial A'$ and no other edge of
 $A$ is on $\partial A'$. By Proposition \ref{conv-UHPT}, under the condition $\P(\mathcal{B}^\bullet_r(\u)=A)>0$,
 $$\P(\mathcal{B}^\bullet_r(\u)=A) = \lim_{p\to\infty} \P(A\sqsubset \t^{(p)}_\infty).$$
 On the other hand, the fact that $\t^{(p)}_\infty$ is the local limit in distribution of the finite triangulations
 $\t^{(p)}_n$ ensures that
 $$\P(A\sqsubset \t^{(p)}_\infty)=\lim_{n\to\infty} \P(A\sqsubset \t^{(p)}_n).$$
 Fix $p>m$, and note that the property $A\sqsubset \t^{(p)}_n$ will hold if and only
 if $\t^{(p)}_n$ is obtained by gluing a triangulation with a boundary of length $q+(p-m)$ to $A$, in such a way
 that a part (of length $q$) of the boundary of the glued triangulation is identified to 
 $\partial A\backslash \wt\partial A$
 (this should be made more precise by saying that the root edge of the glued triangulation is glued
 to a specific edge of $\partial A\backslash \wt\partial A$). This argument shows that, for $n$ large enough,
 $$\P(A\sqsubset \t^{(p)}_n)= \frac{\#\T_{n-N,p+q-m}}{\#\T_{n,p}}.$$
  In a way similar to the derivation of \eqref{eq:loihullUIPHT},
 we now use \eqref{eq:asymp} to get
 $$\P(A\sqsubset \t^{(p)}_\infty)= \frac{C(p+q-m)}{C(p)} \,(12\sqrt{3})^{-N},$$
 and then, 
  $$\P(\mathcal{B}^\bullet_r(\u)=A) = \lim_{p\to\infty} \P(A\sqsubset \t^{(p)}_\infty)= 12^{q-m}\,(12\sqrt{3})^{-N},$$
  which completes the proof.
  \endproof
 
  \smallskip
  
  Our next goal is to describe the conditional distribution of the ``layer'' $\mathcal{B}^\bullet_{r+1}(\u)\backslash \mathcal{B}^\bullet_r(\u)$
  given $\mathcal{B}^\bullet_r(\u)$. We call internal edge of $\partial \mathcal{B}^\bullet_{r}(\u)$ any edge of
  $\partial \mathcal{B}^\bullet_{r}(\u)$ that does not belong to $\partial \u=\u_0$. We order the internal edges of
  $\partial \mathcal{B}^\bullet_{r+1}(\u)$ in clockwise order and denote them as $E_1,E_2,\ldots,E_Q$, where
  $Q\geq 1$ is the number of internal edges of $\partial \mathcal{B}^\bullet_{r+1}(\u)$.  
  We also let $(-L',0)$ be the left-most vertex of $\partial \mathcal{B}^\bullet_{r}(\u)\cap \partial\u$, and
  $(R',0)$ be the right-most vertex of $\partial \mathcal{B}^\bullet_{r}(\u)\cap \partial\u$. We define
  similarly $L''$ and $R''$ replacing $\mathcal{B}^\bullet_{r}(\u)$ by $\mathcal{B}^\bullet_{r+1}(\u)$.
  See Fig.~\ref{fig:hull-UHPT} for an illustration. 
  
  Any internal edge of $\partial \mathcal{B}^\bullet_{r+1}(\u)$ connects two vertices at distance $r+1$
  from the root vertex and is incident to a ``downward'' triangle whose third vertex belongs to 
  $\partial \mathcal{B}^\bullet_{r}(\u)$. Write $V_1,\ldots,V_Q$ for the vertices of $\partial \mathcal{B}^\bullet_{r}(\u)$ that
  belong to the downward triangles associated with $E_1,\ldots,E_Q$ respectively. Note that $V_1,\ldots,V_Q$
  are not necessarily distinct.
   For $1\leq j\leq Q+1$,
  write $S_i$ for the number of edges of $\partial \mathcal{B}^\bullet_{r}(\u)$ lying between $V_{j-1}$ and $V_j$, 
  where by convention $V_0$ is the vertex $(-L',0)$, and 
  $V_{Q+1}$ is the vertex  $(R',0)$. Note that
  $S_1+\cdots+ S_{Q+1}=:\mathbf{P}_r$ is the number of internal edges of $\partial \mathcal{B}^\bullet_{r}(\u)$.
  
\begin{figure}[!h]
 \begin{center}
 \includegraphics[width=0.9\linewidth]{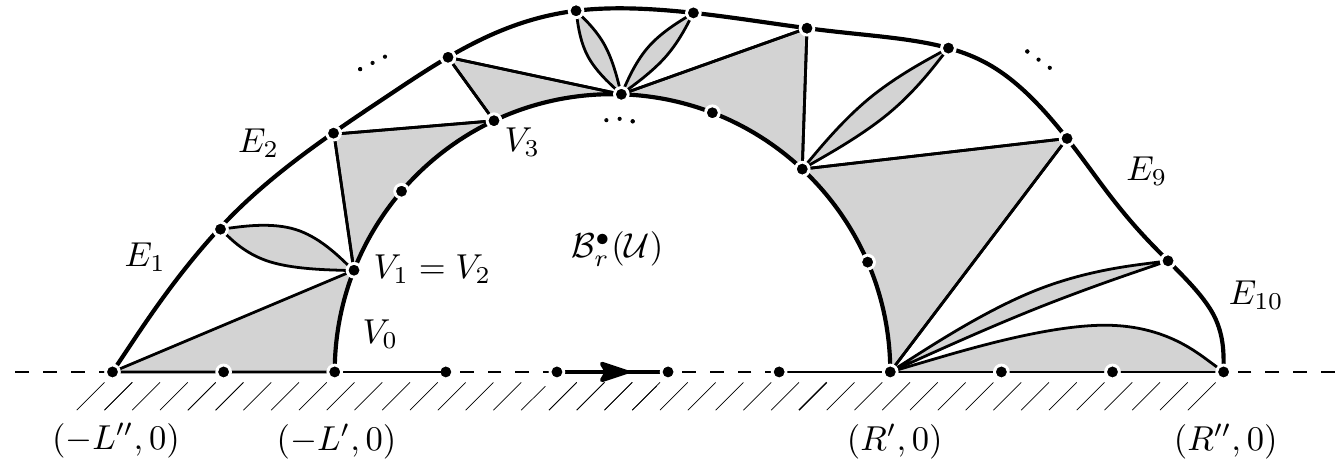}
 \caption{\label{fig:hull-UHPT}Illustration of the setting of Proposition \ref{law-hull-UHPT2}: the layer $ \mathcal{B}_{r+1}^\bullet( \mathcal{U}) \backslash \mathcal{B}_{r}^\bullet( \mathcal{U})$ is displayed. As usual the ``slots'' are represented in grey and not filled in for the clarity of the figure. Comparing with the skeleton decomposition of Section \ref{sec:skeleton}, we notice the particular roles played by the two extremes slots of the layer.}
 \end{center}
 \end{figure}
  
  \begin{proposition}
  \label{law-hull-UHPT2}
  Let  $q\geq 1$ be an integer. Then, for any choice of the nonnegative  integers $s_1,\ldots,s_{q+1}$
  and $k_1,k_2$,
 \begin{align*}
 &\P(Q=q,S_1=s_1,\ldots,S_{q+1}=s_{q+1}, L''-L'=k_1+1,R''-R'=k_2+1\mid \mathcal{B}^\bullet_r(\u))\\
  &\qquad=\frac{1}{4}\;\mathbf{1}_{\{s_1+\cdots+s_{q+1}=\mathbf{P}_r\}}\;\theta(s_1+k_1)\,\theta(s_2)\,\theta(s_3)\cdots\theta(s_q)\,\theta(s_{q+1}+k_2).
  \end{align*}
  \end{proposition}
  
\begin{remark} Note that the above conditional distribution depends on $ \mathcal{B}_{r}^\bullet( \mathcal{U})$ 
only through its internal perimeter $\mathbf{P}_r$. This can be interpreted via the spatial Markov property of the UHPT. \end{remark}
  
  \proof We again write $\alpha=12$ and $\rho=12\sqrt{3}$
  to simplify notation. Fix a triangulation with a boundary $A$, with two distinguished vertices on the 
  boundary as previously, such that $\P(\mathcal{B}^\bullet_r(\u)=A)>0$. Let $p$ be the number of
  internal edges of $\partial A$. The proposition will follow if we can verify that, for any choice of the nonnegative  integers $s_1,\ldots,s_{q+1}$
  and $k_1,k_2$ such that $s_1+\cdots+s_{q+1}=p$, 
\begin{align*}
 &\P(Q=q,S_1=s_1,\ldots,S_{q+1}=s_{q+1}, L''-L'=k_1+1,R''-R'=k_2+1\mid \mathcal{B}^\bullet_r(\u)=A)\\
  &\qquad=\frac{1}{4}\;\theta(s_1+k_1)\,\theta(s_2)\,\theta(s_3)\cdots\theta(s_q)\,\theta(s_{q+1}+k_2).
  \end{align*}
  The left-hand side of the preceding display can be written as
  $$\sum_{A'} \P(\mathcal{B}^\bullet_{r+1}(\u)=A' \mid \mathcal{B}^\bullet_r(\u)=A)$$
  where the sum is over all triangulations with a boundary $A'$ having two distinct 
  distinguished boundary vertices (also distinct from the root vertex) that satisfy
  the following properties:
  \begin{enumerate}
  \item[$\bullet$] $A\sqsubset A'$ in the sense
  explained in the proof of Lemma \ref{law-hull-UHPT};
  \item[$\bullet$] $\wt \partial A\subset \wt \partial A'$, and there are $k_1+1$ boundary edges, resp. $k_2+1$ boundary edges, between 
  the left-most vertex of $\wt\partial A$ and the left-most vertex of $\wt\partial A'$, resp. 
  between 
  the right-most vertex of $\wt\partial A$ and the right-most vertex of $\wt\partial A'$;
  \item[$\bullet$] $\partial A'\backslash \wt\partial A'$ has $q$ edges, each of which is 
  incident is a ``downward triangle'' whose third vertex is incident to $\partial A\backslash \wt\partial A$, and the configuration of these downward triangles is characterized by the numbers $s_1,\ldots,s_{q+1}$ as explained before
  the proposition (see also Fig.~\ref{fig:hull-UHPT}).
  \end{enumerate}
 Fix $A'$ satisfying the preceding properties, and note that
$$ \P(\mathcal{B}^\bullet_{r+1}(\u)=A' \mid \mathcal{B}^\bullet_r(\u)=A)= \frac{\P(\mathcal{B}^\bullet_{r+1}(\u)=A')}
 {\P(\mathcal{B}^\bullet_r(\u)=A)},$$
 because the event $\{\mathcal{B}^\bullet_{r+1}(\u)=A'\}$ is contained in $\{\mathcal{B}^\bullet_r(\u)=A\}$. 
 We use Lemma \ref{law-hull-UHPT} to evaluate the ratio of the probabilities appearing in the 
 previous display. It readily follows that
 \begin{equation}
 \label{hull-UHPT-1}
 \frac{\P(\mathcal{B}^\bullet_{r+1}(\u)=A')}
 {\P(\mathcal{B}^\bullet_r(\u)=A)}= \alpha^{q-p}\,\alpha^{-(k_1+k_2+2)}\,\rho^{-N}
 \end{equation}
 where $N$ denotes the number of vertices of $A'$ that are not vertices of $A$ or of $\partial \u$. 
 
 At this point we note that the triangulation $A'$ is completely determined if in addition to the
 preceding properties we know the triangulations with a boundary that fill in the slots
 of $A'\backslash A$
 left by the downward triangles. Note that there are $q+1$ such slots, and that the first one
 and the last one play a particular role since their boundary contains edges of $\wt \partial A$. 
 More precisely, for $2\leq i\leq q$, the boundary of the slot contains $s_i+2$
 edges, whereas for $i=1$ it contains $s_1+k_1+2$ edges, and for $i=q+1$ it contains 
 $s_{q+1}+k_2+2$ edges. Write $\mathcal{M}_i$ for the triangulations with a boundary
 filling the $i$-th slot, and let ${\rm Inn}(\mathcal{M}_i)$ be the number of inner vertices of $\mathcal{M}_i$. Then, 
 we have
 $$N= q-1 + \sum_{i=1}^{q+1} {\rm Inn}(\mathcal{M}_i).$$
 Set $\wt s_i=s_i$ if $2\leq i\leq q$ and $\wt s_1 =s_1+k_1$, $\wt s_{q+1}=s_2+k_2$
 to simplify notation. Then simple manipulations show that formula
 \eqref{hull-UHPT-1} can be rewritten in the form
 $$ \frac{\P(\mathcal{B}^\bullet_{r+1}(\u)=A')}
 {\P(\mathcal{B}^\bullet_r(\u)=A)}= \alpha^{-3}\rho^2 \prod_{i=1}^{q+1} \Big(\frac{1}{\rho}
 \alpha^{-(\wt s_i-1)} \,\rho^{-{\rm Inn}(\mathcal{M}_i)}\Big)$$
 (observe that $\sum_{i=1}^{q+1} (\wt s_i-1)= p-(q+1)+k_1+k_2$). Next note that $ \alpha^{-3}\rho^2=\frac{1}{4}$
 and recall the definition \eqref{def:theta} of the probability distribution $\theta$. We arrive at the formula
 \begin{equation}
 \label{hull-UHPT-2}
 \frac{\P(\mathcal{B}^\bullet_{r+1}(\u)=A')}
 {\P(\mathcal{B}^\bullet_r(\u)=A)}=
 \frac{1}{4}\,\prod_{i=1}^{q+1}  \Big( \theta(\wt s_i)\;\frac{\rho^{-{\rm Inn}(\mathcal{M}_i)}}{Z(\wt s_i+2)}\Big).
 \end{equation}
 It remains to sum over all possible choices of $A'$. But as explained earlier, this amounts to
 summing over possible choices of the triangulations $\mathcal{M}_1,\ldots,\mathcal{M}_{q+1}$ with boundaries of
 the prescribed lengths. By the very definition of $Z(\cdot)$, we obtain
$$\sum_{A'} \frac{\P(\mathcal{B}^\bullet_{r+1}(\u)=A')}
 {\P(\mathcal{B}^\bullet_r(\u)=A)}=  \frac{1}{4}\,\prod_{i=1}^{q+1}  \theta(\wt s_i),$$
 which completes the proof. \endproof
 
 \begin{remark}
A simplified version of the arguments of the preceding proof also gives 
the distribution of the hull $\mathcal{B}^\bullet_1(\u)$. Let $Q$ denote the number of internal edges of 
$\partial \mathcal{B}^\bullet_1(\u)$, and write $(R,0)$, respectively $(-L,0)$, for the 
right-most vertex, resp.~the left-most vertex, of
$\partial \mathcal{B}^\bullet_1(\u)\cap \partial \u$. Then, for every $q\geq 1$ and $k_1,k_2\geq 0$,
\begin{equation}
\label{hullradius1}
\P(Q=q,L=k_1+1,R=k_2+1)= \frac{1}{4}\,\theta(0)^{q-1}\,\theta(k_1)\,\theta(k_2)
\end{equation}
(note that $\theta(0)=3/4$). We leave the details to the reader.
\end{remark}
 
 \begin{corollary}
 \label{bd-UHPT-boundary}
 For every integers $c,k\geq 0$,
\begin{equation}
\label{bdry-formula}
\P(S_{Q+1}=c,R''-R'=k+1\mid \mathcal{B}^\bullet_r(\u))=\mathbf{1}_{\{c\leq \mathbf{P}_r\}}\,\frac{1}{2}
\big(1+h(\mathbf{P}_r-c)\big)\,\theta(c+k),
\end{equation}
 where we recall that $h(j)=4^{-j}{2j\choose j}$ for every integer $j\geq 0$. 
 Consequently, for every integer $k\geq 0$,
\begin{equation}
\label{bd-bdry}
\P(R''-R'=k+1\mid \mathcal{B}^\bullet_r(\u))\leq \theta([k,\infty)).
\end{equation}
 \end{corollary}
 
 \proof
 From the identity of Proposition \ref{law-hull-UHPT2}, we get that
\begin{align*}
&\P(S_{Q+1}=c,R''-R'=k+1\mid \mathcal{B}^\bullet_r(\u))\\
&\qquad=
 \frac{1}{4}\;\Bigg(\sum_{q=1}^\infty\;\sum_{s_1+\cdots+s_{q}=\mathbf{P}_r-c}\;\theta([s_1,\infty))\,\theta(s_2)\cdots
 \theta(s_{q-1})\,\theta(s_q)\Bigg) \theta(c+k).
 \end{align*}
 where the right-hand side is $0$ if $\mathbf{P}_r<c$. We first observe that, for
 every integer $p\geq0$,
 \begin{equation}
 \label{bd-UHPT-tech}
 \sum_{q=1}^\infty\;\sum_{s_1+\cdots+s_{q}=p}\;\theta((s_1,\infty))\,\theta(s_2)\cdots\theta(s_{q-1})\,\theta(s_{q}) = 1.
 \end{equation}
 This identity is immediate: If $X_1,X_2,\ldots$ is a sequence of i.i.d. random variables with distribution $\theta$, and
 if $H_p:=\min\{q\geq 1: X_1+\cdots+X_q>p\}$, the left-hand side of \eqref{bd-UHPT-tech} is just $\sum_{q=1}^\infty \P(H_p=q)=1$. 
 To get formula \eqref{bdry-formula}, we then need to verify that, for every integer $p\geq 0$,
  \begin{equation}
 \label{bd-UHPT-tech2}
\sum_{q=1}^\infty\sum_{s_1+\cdots+s_{q}=p}\;\theta(s_1)\,\theta(s_2)\cdots\theta(s_{q-1})\,\theta(s_{q})= 1+ 2\,h(p).
 \end{equation}
For $p=0$, this follows immediately from the fact that $\theta(0)=3/4$.
 So we restrict our attention to $p\geq 1$. We first observe that 
 $$\sum_{s_1+\cdots+s_{q}=p}\;\theta(s_1)\,\theta(s_2)\cdots\theta(s_{q-1})\,\theta(s_{q})= Q_q(0,p)$$
 where $Q_q(i,j)$ denotes the transition kernel of the random walk with jump distribution $\theta$. Hence, for $p\geq 1$,
the left-hand side of \eqref{bd-UHPT-tech2} is equal to $G(0,p)$, where $G(i,j)=\sum_{q=0}^\infty Q_q(i,j)$
is the Green kernel of the same random walk. We then observe that, for $x\in[0,1)$,
$$\sum_{k=0}^\infty x^k\,G(0,k)= \sum_{q=0}^\infty g_\theta(x)^q= \frac{1}{1-g_\theta(x)}= \Big(1+\frac{1}{\sqrt{1-x}}\Big)^2.$$
By expanding $(1+\frac{1}{\sqrt{1-x}})^2$ as a power series, we obtain that, for every $k\geq 0$,
$$G(0,k)=\mathbf{1}_{\{k=0\}} + 1 + 2\,h(k),$$
giving the right-hand side of \eqref{bd-UHPT-tech2} when $k=p\geq 1$. This completes the proof of
\eqref{bdry-formula}. The bound \eqref{bd-bdry} follows by summing over $c$, noting that $h(j)\leq 1$ for every $j\geq 0$.
\endproof
  
  \subsection{Distances along the boundary of the UHPT}
   \label{sec:distanceUHPT}
  We now use the results of the preceding subsection to get bounds on distances between vertices of $\partial \u$
 and the root vertex. Recall that $\partial \u$ is identified to $\Z\times \{0\}$, so that $(0,0)$
  is the root vertex. 
  
  For every integer $r\geq 1$, we now write $(-L^\u_r,0)$ for the left-most vertex in $\partial \mathcal{B}^\bullet_r(\u)\cap\partial \u$
  and $(0,R^\u_r)$ for the right-most vertex in $\partial \mathcal{B}^\bullet_r(\u)\cap\partial \u$. We also set $R^\u_0=0$
  and $L^\u_0=0$ by convention, and we also agree that $ \mathcal{B}^\bullet_0(\u)$ is the 
  edge-triangulation. Recall that
  $\mathrm{d}^\u_{\rm gr}$ is the graph distance on the vertex set of $\u$.
  
  \begin{proposition}
  \label{dist-bdry}
  The sequences $(r^{-2}L^\u_r)_{r\geq 1}$ and $(r^{-2}R^\u_r)_{r\geq 1}$ are bounded in probability: For every 
  $\ve>0$, there exists a constant $K$ such that
  $$\sup_{r\geq 1} \P(L^\u_r\geq K\,r^2) \leq \ve \quad \mbox{ and } \quad \sup_{r\geq 1} \P(R^\u_r\geq K\,r^2) \leq \ve,$$
and consequently, for every $r\geq 1$,
\begin{equation}
\label{dist-bdry-tec1}
\P\Big(\min_{|j|\geq K\,r^2} \mathrm{d}^\u_{\rm gr}((0,0),(j,0)) >r\Big) \geq 1-2\ve.
\end{equation}
For every integer $m\geq 1$, set $T_m:=\min\{r\geq 1: R^\u_r > m\}$. There exists a constant $K'$
such that, for every $m\geq 1$ and every $j\geq 1$,
$$\P( R^\u_{T_m}-m>j)\leq K'\,\sqrt{\frac{m}{m+j}}.$$
  \end{proposition}
  
  \proof Let us start with the first statement. By symmetry, it suffices to consider the sequence $(r^{-2}R^\u_r)_{r\geq 1}$. By combining \eqref{bd-bdry}
  and \eqref{asymp-theta}, we get the existence of a constant $H$ such that, for every $r\geq 1$,
  and every $k\geq 1$,
 \begin{equation}
 \label{bd-increment}
 \P(R^\u_{r+1}-R^\u_r=k\mid R^\u_1,\ldots,R^\u_r) \leq H\, k^{-3/2}.
 \end{equation}
  From this bound, for instance by using a coupling with a stable subordinator with index $1/2$, one derives the
  fact that the sequence $(r^{-2}R^\u_r)_{r\geq 1}$ is bounded in probability. We then note that
  the condition $j>R^\u_r$, or $j<-L^\u_r$ implies by definition that 
  $(j,0)\notin \mathcal{B}^\bullet_r(\u)$. Therefore,
  $$\min_{|j|> R^\u_r\vee L^\u_r} \mathrm{d}^\u_{\rm gr}((0,0),(j,0)) > r,$$
  giving \eqref{dist-bdry-tec1} since,
  by the first part of the proposition,
  $$\P(R^\u_r\vee L^\u_r < K\,r^2) \geq 1-2\ve.$$
  
  Let us turn to the proof of the last assertion. It is enough to establish the existence of a constant
  $K'$, which does not depend on $m$, such that, for every integer $k\geq 0$, for every $\ell\in\{0,1,\ldots,m\}$ and for every integer $j\geq 0$,
  \begin{equation}
  \label{dist-bdry-tech1}
  \P(R^\u_{k+1}-R^\u_{k}> \ell +j\mid \mathcal{B}^\bullet_k(\u)) \leq K'\sqrt{\frac{m}{m+j}}\,\P(R^\u_{k+1}-R^\u_{k}> \ell \mid \mathcal{B}^\bullet_k(\u)).
  \end{equation}
  Indeed, if \eqref{dist-bdry-tech1} holds, we have, for every $j\geq 0$,
  \begin{align*}
  \P(R^\u_{T_m} -m>j)&=\sum_{k=0}^\infty  \P(R^\u_k\leq m,R^\u_{k+1}-R^\u_k> m+j-R^\u_k)\\
  &=\sum_{k=0}^\infty \E[\mathbf{1}_{\{R^\u_k\leq m\}}\,\P(R^\u_{k+1}-R^\u_k>m+j-R^\u_k\mid \mathcal{B}^\bullet_k(\u))]\\
  &\leq K'\sqrt{\frac{m}{m+j}}\,\sum_{k=0}^\infty \E[\mathbf{1}_{\{R^\u_k\leq m\}}\,\P(R^\u_{k+1}-R^\u_k> m-R^\u_k\mid \mathcal{B}^\bullet_k(\u))]\\
 & =K'\sqrt{\frac{m}{m+j}}\,.
 \end{align*}
  Let us prove \eqref{dist-bdry-tech1}. The case $k=0$ follows very easily from \eqref{hullradius1} and so we concentrate on the case
  $k\geq 1$. From Corollary \ref{bd-UHPT-boundary}, we have, for every $j\geq 0$,
  $$\P(R^\u_{k+1}-R^\u_{k}> \ell +j\mid \mathcal{B}^\bullet_k(\u))
  = \frac{1}{2}\sum_{i=\ell+j}^\infty \sum_{c=0}^{\mathbf{P}_k} \big(1+h(\mathbf{P}_k-c)\big)\,\theta(c+i)=\frac{1}{2}\sum_{c=0}^{\mathbf{P}_k} \big(1+h(\mathbf{P}_k-c)\big)\,\theta([c+\ell+j,\infty)) ,$$
  where $\mathbf{P}_k$ is the number of internal edges of $\partial\mathcal{B}^\bullet_k(\u)$. Since $1\leq 1+h(\mathbf{P}_k-c)\leq 2$, the bound 
  \eqref{dist-bdry-tech1} will follow if we can verify that, for every $p\geq 1$, for every $\ell\in\{0,1,\ldots,m\}$ and every $j\geq 0$,
  $$\sum_{c=0}^{p} \theta([c+\ell+j,\infty))\leq K''\,\sqrt{\frac{m}{m+j}}\,\sum_{c=0}^{p} \theta([c+\ell,\infty)),$$
  with a suitable constant $K''$. The bound in the last display is derived by elementary arguments relying on 
  \eqref{asymp-theta}: Note that, if $0\leq c\leq m$, the ratio $\theta([c+\ell+j,\infty))/\theta([c+\ell,\infty))$ is
  bounded above by (a constant times) $(m/(m+j))^{3/2}$, and on the other hand
  $$\sum_{c=m}^{\infty} \theta([c+\ell+j,\infty))$$
  is bounded above by a constant times $(m+j)^{-1/2}$, whereas 
  $$\sum_{c=0}^{m} \theta([c+\ell,\infty))$$
  is bounded below by a positive constant times $m^{-1/2}$. 
   \endproof
   
    We will need a reinforcement of property \eqref{dist-bdry-tec1}, which is given in the next proposition. 
    
   \begin{proposition}
   \label{dist-bdry2}
   Let $\ve>0$. For every integer $A>0$, we can choose an integer $K>0$ sufficiently large so that, 
   for every $r\geq 1$,
   $$\P\Big(\min_{0\leq i\leq Ar^2,\;j\geq Kr^2} \mathrm{d}_{\rm gr}^\u ((i,0),(j,0)) \leq r\Big) \leq \ve.$$
   \end{proposition}
   
   \proof
   Fix an integer $A>0$.
   By the last assertion of the previous proposition, we can choose an integer $A'>A$ large enough so that,
   for every $r\geq 1$,
   $$\P\Big( R^\u_{T_{Ar^2}} \geq A'r^2\Big) <\frac{\ve}{2}.$$
  Fix $r\geq 1$ and set, for every $j\geq 0$,
   $$\wt R^\u_j:= R^\u_{T_{Ar^2}+j} - R^\u_{T_{Ar^2}}.$$
   Since $T_{Ar^2}$ is a stopping time of the process $(R^\u_j)_{j\geq 0}$, it follows from
   \eqref{bd-increment} that we have also, for every $j\geq 0$ and $k\geq 1$,
   $$ \P(\wt R^\u_{j+1}-\wt R^\u_j=k\mid \wt R^\u_1,\ldots,\wt R^\u_j) \leq H\, k^{-3/2}.$$
   As in the proof of Proposition \ref{dist-bdry}, this implies that the sequence $(j^{-2}\wt R^\u_j)_{j\geq 1}$ is bounded in probability, 
   and since $H$ does not depend on our choice of $r$, we even get bounds that are uniform in $r$. Hence 
   we can choose an integer $K>A'$, which does not depend on the choice of $r$, so that
      $$\P(\wt R^\u_r \geq (K-A')r^2) < \frac{\ve}{2}.$$
   
   Finally, we know that any vertex $(i,0)$ with $0\leq i\leq Ar^2$
   belongs to the hull of radius $T_{Ar^2}$ (because by definition $R^\u_{T_{Ar^2}}>Ar^2$).
 On the event $\{R^\u_{T_{Ar^2}} < A'r^2\}\cap \{\wt R^\u_r < (K-A')r^2\}$,
 we have   $R^\u_{T_{Ar^2}+r} < Kr^2$, which implies that vertices $(j,0)$ with 
$ j\geq Kr^2$ do not belong to the hull of radius $T_{Ar^2}+r$. Hence, on the latter event
we must have $\mathrm{d}_{\rm gr}^\u ((i,0),(j,0))>r$ whenever $0\leq i\leq Ar^2$ and 
$ j\geq Kr^2$. This completes the proof. \endproof
 
\begin{remark} The known results for quadrangulations \cite{CCuihpq}
strongly suggest that $(r^{-2} L_{\lfloor tr \rfloor }^{ \mathcal{U}}, r^{-2} R_{\lfloor tr \rfloor}^{ \mathcal{U}})_{t \geq 0}$ should 
converge in distribution to the pair formed by the last hitting time processes of two independent Bessel processes of dimension $5$.
 \end{remark}

\subsection{Distances along the boundary of the LHPT}

We will now deduce an analog of Proposition \ref{dist-bdry2} for distances along the
boundary of the LHPT $\l$. Our main technical tool will be the absolute continuity
property of Proposition \ref{size-bias-UHPT}. We will also use left-most geodesics, which are defined 
in a way similar to the end of Section \ref{sec:skeleton}. For every $i\in\Z$ and every integer $r\geq 1$, the
left-most geodesic from $(i,0)$ in $\l$ is the infinite geodesic path $\omega$ in $\l$ that starts from $\omega(0)=(i,0)$, visits
a vertex $\omega(n)\in \l_n$ at each step $n\geq 0$, and is obtained by choosing at step $n+1$ the left-most
edge between $\omega(n)$ and $\l_{n+1}$. For every $r\geq 1$, the  first $r$ edges on this path
give the
left-most geodesic from $(i,0)$ to $\l_r$ in $\l$. Similarly, in the UHPT $\u$, we can define the
left-most geodesic from $(i,r)$ to $\partial \u$, for every $i\in\Z$ and $r\geq 1$. Furthermore, if $1\leq i<j$, 
the left-most geodesics from $(i,r)$ to $\partial \u$ and from $(j,r)$ to $\partial \u$ coalesce before hitting $\partial \u$
(possibly when they hit $\partial \u$), if and only if none of the trees $\Gamma_{(i,r)},\Gamma_{(i+1,r)},\ldots,\Gamma_{(j-1,r)}$
has height $r$. 

Recall that $\mathrm{d}^\l_{\rm gr}$ is the graph distance on the vertex set of $\l$.

\begin{proposition}
\label{dist-bdry3}
For every 
  $\ve>0$, there exists an integer $K\geq 1$ such that
  for every $r\geq 1$,
$$\P\Big(\min_{|j|\geq K\,r^2} \mathrm{d}^\l_{\rm gr}((0,0),(j,0)) \geq r\Big) \geq 1-\ve.$$
Consequently, we have also with $K'=4\,K$, for every $r\geq 1$,
$$\P\Big(\min_{|j|\geq 2K'\,r^2}\  \min_{-K'r^2\leq i\leq K'r^2}\mathrm{d}^\l_{\rm gr}((i,0),(j,0)) \geq r\Big) \geq 1-2\ve.$$
\end{proposition}

\proof Let us start with the first assertion.
For obvious symmetry reasons, it is enough to consider positive values of $j$. 
For every integer $K\geq 1$, and for every $r\geq 1$, we consider the measurable set $A_r(K)$ such that
$\l_{[0,r]}\in A_r(K)$ if and only if there is a path in $\l_{[0,r]}$
of length strictly smaller than $r$ that connects $(0,0)$ to $(0,j)$, for some 
$j\geq K\,r^2$. Note that $\l_{[0,r]}\in A_r(K)$ if and only if 
$$\min_{j\geq K\,r^2} \mathrm{d}^\l_{\rm gr}((0,0),(j,0))<  r,$$
since a path in $\l$ that starts from $(0,0)$ and has length strictly smaller than $r$
must stay in $\l_{[0,r]}$. We therefore need to prove that,
if $K$ is chosen sufficiently large, we have, for every $r\geq 1$,
$$\P(\l_{[0,r]}\in A_r(K))< \ve.$$
Thanks to Corollary \ref{abso-con}, it is then enough to verify that, if
$K$ is sufficiently large, we have for every $r\geq 1$,
$$\P(\wt\u_{[0,r]}\in A_r(K))< \ve,$$
where $\wt\u_{[0,r]}$ is defined in Proposition \ref{size-bias-UHPT}.
Recalling the notation of this proposition, we have
$$\{\wt\u_{[0,r]}\in A_r(K)\} \subset \Big\{\min_{i\geq Kr^2} \mathrm{d}_{\rm gr}^\u((J_r,r),(J_r+i,r))<r\Big\}.$$
To bound the probability of the event in the right-hand side, let $(I_r,0)$
be the endpoint of the left-most geodesic from $(1,r)$ to $\partial \u$ in $\u$. 
By the observations preceding the proposition, and the fact that
$1\leq J_r\leq K_r$, we get that the left-most geodesic from $(J_r,r)$ to $\partial\u$ coalesces with the 
one from $(1,r)$ before reaching $\partial \u$ (possibly when hitting $\partial \u$). Consequently, we have
$$\mathrm{d}^\u_{\rm gr}((J_r,r),(I_r,0))= r.$$
Let $(I'_r(K),0)$ be the endpoint of the left-most geodesic from $(Kr^2,r)$ to $\partial \u$ in $\u$. See Fig.~\ref{fig:lowerbound} for
an illustration of the preceding definitions. 
Suppose that both $J_r<Kr^2$ and there exists $i\geq Kr^2$
such that $\mathrm{d}_{\rm gr}^\u((J_r,r),(J_r+i,r))<r$. Then the endpoint  
of  the left-most geodesic from $(J_r+i,r)$ to $\partial \u$ is of the form $(j,0)$
with $j\geq I'_r(K)$ since $J_r+i\geq Kr^2$. The triangle inequality then shows that
$$\mathrm{d}^\u_{\rm gr}((I_r,0),(j,0))\leq \mathrm{d}^\u_{\rm gr}((I_r,0),(J_r,r)) + \mathrm{d}^\u_{\rm gr}((J_r,r),(J_r+i,r)) + 
\mathrm{d}^\u_{\rm gr}((J_r+i,r),(j,0)) < 3r.$$

\begin{figure}[!h]
 \begin{center}
 \includegraphics[width=0.8\linewidth]{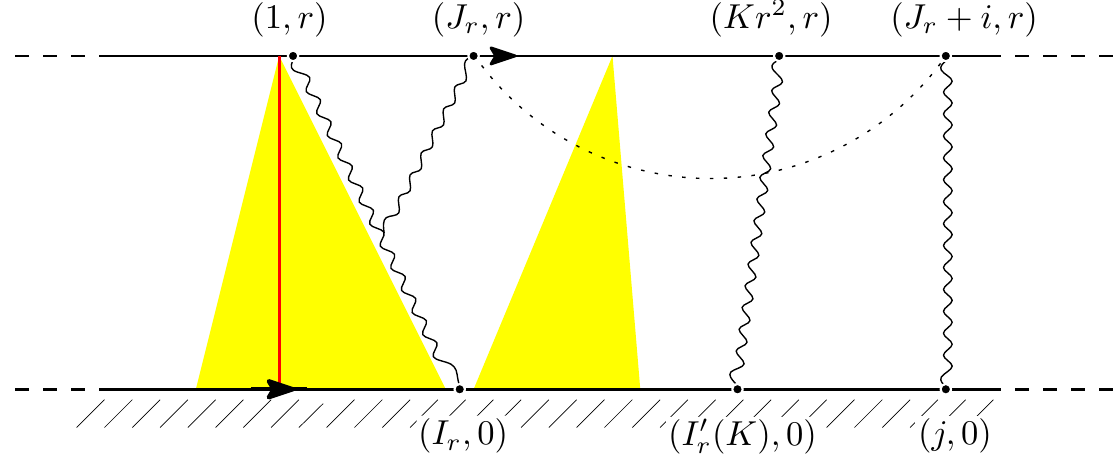}
 \caption{\label{fig:lowerbound}Illustration of the Proof of Proposition \ref{dist-bdry3}. The wavy lines represent left-most geodesics, the dotted line is the unlikely path linking $(J_{r},r)$ to $(J_{r}+i,r)$. Finally the first two trees $\Gamma_{(\ell,r)}$ with $\ell \geq 0$ that hit $\partial \u$ are represented in yellow.}
 \end{center}
 \end{figure}

Summarizing, we have
\begin{equation}
\label{estim0}
\P(\wt\u_{[0,r]}\in A_r(K))\leq \P(J_r\geq Kr^2) + \P\Big(\min_{j\geq I'_r(K)}\mathrm{d}^\u_{\rm gr}((I_r,0),(j,0))< 3r\Big).
\end{equation}
Since we already noticed that $\E[(J_r)^2]\leq 4(r+1)^4$, the Markov inequality immediately
tells us that we can find $K_0$ such that $\P(J_r\geq Kr^2)\leq \ve/4$ if $K\geq K_0$. We then
bound the second term in the right-hand side of \eqref{estim0}. 
We first claim that, we can find an integer $B$ such that, for
every $r\geq 1$,
\begin{equation}
\label{estim1}
\P(I_r\geq B r^2)\leq \frac{\ve}{4}.
\end{equation}
Indeed, from the construction of the downward triangles of the UHPT in Section \ref{subsec:UHPT}, 
it is not hard to see that $I_r$ is equal to the maximal integer $m\geq 1$ such that
$(m-\frac{1}{2},0)$ is a vertex of $\Gamma_{(0,r)}$. Consequently,  $I_r$ is bounded above by the size $N_r$ of 
generation
$r$ of the tree $\Gamma_{(0,r)}$, which is a size-biased Galton--Watson tree
with offspring distribution $\theta$. Note that, for $x\in[0,1)$,
$$\E[x^{N_r}]=\sum_{k=1}^\infty k\,\P_1(Y_r=k)\,x^k= x \times \frac{\mathrm{d}}{\mathrm{d}x} \Big(1-(r+\frac{1}{\sqrt{1-x}})^{-2}\Big)
= x (1+r\sqrt{1-x})^{-3}$$
from which one easily gets that $r^{-2}N_r$ converges in distribution, yielding the estimate 
\eqref{estim1}. 

We then use Proposition \ref{dist-bdry2}, which allows us to find an integer
$K$ sufficiently large so that, for every $r\geq 1$,
\begin{equation}
\label{estim2}
\P\Big(\min_{0\leq i\leq Br^2,\;j\geq Kr^2/3} \mathrm{d}_{\rm gr}^\u ((i,0),(j,0)) \leq 3r\Big) \leq \frac{\ve}{4}.
\end{equation}

We finally claim that, by choosing $K$ even larger if needed, we have also for every $r\geq 1$,
\begin{equation}
\label{estim3}
\P(I'_r(K)< Kr^2/3) \leq \frac{\ve}{4}.
\end{equation}

Assuming that \eqref{estim3} holds, we can combine \eqref{estim1}, \eqref{estim2} and \eqref{estim3} to get that 
the second term in the right-hand side of \eqref{estim0} is bounded above by $3\ve/4$, and since we also know that 
$\P(J_r\geq Kr^2)\leq \ve/4$, we deduce from \eqref{estim0} that $\P(\wt\u_{[0,r]}\in A_r(K))\leq\ve$,
which was the desired result for the first assertion of the proposition.

It remains to verify that the estimate \eqref{estim3} holds. From the construction of
the UHPT, the quantity $I'_r(K)$ is bounded below by 
$$M_r(K):=\sum_{\ell=1}^{Kr^2-1} \#\Gamma_{(\ell,r)}(r).$$
We have, for every $x\in[0,1)$,
$$ \mathbb{E}[x^{M_r(K)}]= \Big(1-\Big(r+\frac{1}{\sqrt{1-x}}\Big)^{-2}\Big)^{Kr^2-1},$$
and it follows by straightforward calculations that $r^{-2} M_r(K)$ converges in
distribution to a random variable $U_K$ with Laplace 
transform $  \mathbb{E}[e^{-\lambda U_K}]= \exp(-K(1+\lambda^{-1/2})^{-2})$. Then
$U_K/K$ converges in probability to $1$ as $K\to\infty$, and thus we can fix $K$ sufficiently large
so that $ \mathbb{P}(U_K>K/2)>1-\ve/4$. For this value of $K$, the estimate \eqref{estim3}
holds for all sufficiently large $r$. We can deal with the remaining values of 
$r$ by taking $K$ even larger if necessary (using now the law of large numbers). This completes the proof of
the first assertion.

The proof of the second assertion is now easy. If we assume that for some 
$j\geq 2K'r^2$ and for some $i\in\{-K'r^2,\ldots,K'r^2\}$, we have 
$\mathrm{d}^\l_{\rm gr}((i,0),(j,0)) <r$, then any geodesic from $(i,0)$ to $(j,0)$ 
must stay in $\l_{[0,r]}$ (otherwise it has length greater than $r$) and by a planarity
argument it must intersect the leftmost geodesic from $(K'r^2,0)$ to the horizontal line
$\l_r$, and it follows that $\mathrm{d}^\l_{\rm gr}((K'r^2,0),(j,0)) <2r$. However, by the first assertion
of the proposition (with $r$ replaced by $2r$) and a translation argument, the probability that 
$\mathrm{d}^\l_{\rm gr}((K'r^2,0),(j,0)) <2r$ for some $j\geq 2K'r^2$ is bounded above by $\ve$. The desired
result follows, noting that we must also consider the case $j<-2K'r^2$.
\endproof

\medskip
The preceding proposition provides lower bounds on distances between
vertices on the boundary of the LHPT. We state another proposition that 
gives upper bounds for the same quantities.
We recall the notation $\l_i$ for the line $\{-i\}\times \Z$. Also recall the definition of the left-most geodesic in $\l$
 starting from $v$, for every $v\in\partial \l$.

\begin{proposition}
\label{upper-distLHPT}
Let $\delta>0$ and $\gamma>0$. We can choose an integer $A\geq 1$
such that the following holds for every sufficiently large $n$. With
probability at least $1-\delta$, we have:
\begin{enumerate}
\item[$\bullet$] for every $i\in\{-n+1,-n+2,\ldots,n\}$, the left-most geodesic starting from $(i,0)$ coalesces with the left-most geodesic starting from
$(-n+\lfloor 2\ell n/A\rfloor,0)$, for some $0\leq \ell \leq A$, before hitting $\l_{\lfloor \gamma 
\sqrt{n}\rfloor}$;
\item[$\bullet$] for every $i,j\in\{-n+1,-n+2,\ldots,n\}$, with $i<j$, there is a path 
from $(i,0)$ to $(j,0)$ that stays in $\l_{[0,{\lfloor \gamma \sqrt{n}\rfloor}]}$ and has length
smaller than
$$\Big(\big\lfloor \frac{A(j-i)}{2n}\big\rfloor+2\Big) (1+2\gamma\sqrt{n}).$$
\end{enumerate}
\end{proposition}

\proof
Let $U^{(n)}_1<U^{(n)}_2<\cdots<U^{(n)}_{m_n}$
be all indices in $\{1,2,\ldots, 2n-1\}$ such that
the height of $\tc_{-n+i}$ is greater than or equal to $\lfloor\gamma\sqrt{n}\rfloor$. 
If we set for every $t\in [0,2n]$,
$$N^{(n)}_t:=\#\{i\in\{1,\ldots,m_n\}: U^{(n)}_i\leq t\},$$
it follows from \eqref{heighttree} that
$(N^{(n)}_{\lfloor nt\rfloor})_{0\leq t\leq 2}$ converges in
distribution in the Skorokhod sense to a Poisson 
process with parameter $\gamma^{-2}$. Write 
$U^{(n)}_0=0$ and $U^{(n)}_{m_n+1}=2n$ by convention.
The preceding observations imply that we can choose
$\eta>0$ small enough so that, for every sufficiently large
$n$, the property
\begin{equation}
\label{spacing}
U^{(n)}_{i+1}- U^{(n)}_i > \eta n\ , \quad \forall i\in\{0,1,\ldots,m_n\}
\end{equation}
holds with probability at least $1-\delta$. 

By the coalescence property of left-most geodesics, if $U^{(n)}_j< i\leq i'\leq U^{(n)}_{j+1}$, the leftmost geodesic
from $(-n+i,0)$ coalesces with that from $(-n+i',0)$ before
hitting the line $\l_{\lfloor \gamma 
\sqrt{n}\rfloor}$.
Now set $A=\lfloor 2/\eta\rfloor +1$, so that $2/A<\eta$. On the event
where \eqref{spacing} holds, each interval $]U^{(n)}_j, U^{(n)}_{j+1}]$,
for $0<j\leq m_n$,
contains at least one of the points $\lfloor 2\ell n/A\rfloor$, $1\leq \ell \leq A$.
The first assertion of the proposition follows.

\begin{figure}[!h]
 \begin{center}
 \includegraphics[width=0.9\linewidth]{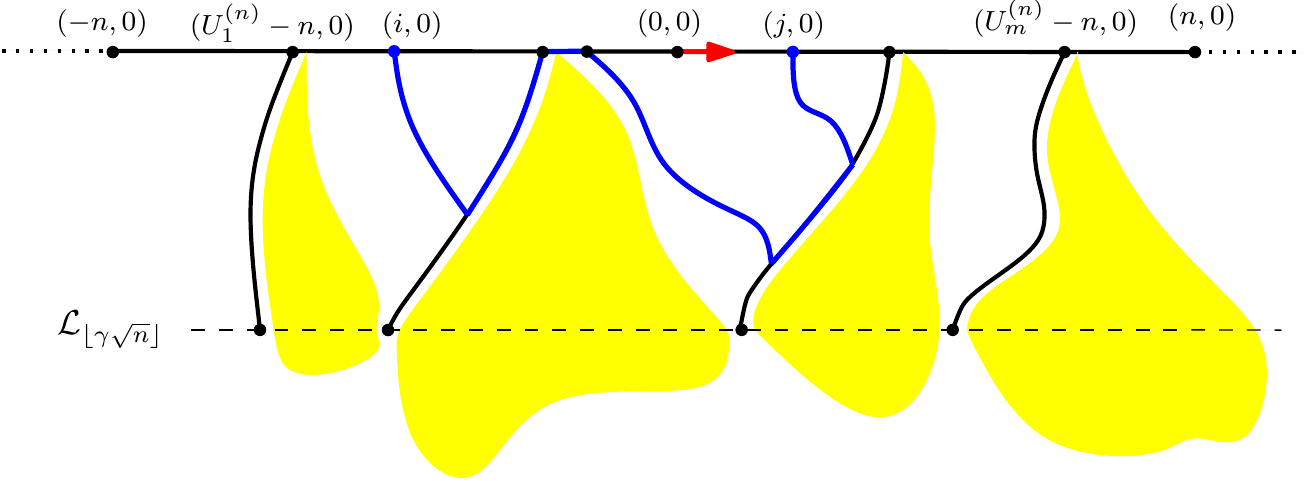}
 \caption{\label{fig:upper-distLHPT}Illustration of the proof of Proposition \ref{upper-distLHPT}. The trees 
 reaching height $\lfloor\gamma \sqrt{n}\rfloor$ are represented in yellow. The left-most geodesics are represented in black and the (non-geodesic) path connecting $(i,0)$ to $(j,0)$, which is constructed in the proof, is represented in blue.}
 \end{center}
 \end{figure}
To get the second assertion, let $k\leq \ell$ be such that $U^{(n)}_k+1\leq i\leq U^{(n)}_{k+1}$
and $U^{(n)}_\ell+1\leq j\leq U^{(n)}_{\ell+1}$. If $k=\ell$, the left-most geodesics
from $(i,0)$ and $(j,0)$ coalesce before hitting $\l_{\lfloor \gamma 
\sqrt{n}\rfloor}$, and we construct a path from $(i,0)$ to $(j,0)$, with length bounded above by
$2\gamma\sqrt{n}$, by concatenating the parts of these two geodesics before their coalescence 
time. If $k<\ell$, we construct a path from $(i,0)$ to $(j,0)$ as follows. We first construct a path from
$(i,0)$ to $(U^{(n)}_{k+1},0)$ with length smaller than $2\gamma\sqrt{n}$ by concatenating 
left-most geodesics, and
we add to this path the edge between $(U^{(n)}_{k+1},0)$ and $(U^{(n)}_{k+1}+1,0)$. We then 
concatenate the left-most geodesics from $(U^{(n)}_{k+1}+1,0)$ and from $(U^{(n)}_{k+2},0)$
up to their coalescence time
to get a path from $(U^{(n)}_{k+1}+1,0)$ to $(U^{(n)}_{k+2},0)$
with length smaller than $2\gamma\sqrt{n}$, and we add to this path the edge between 
$(U^{(n)}_{k+2},0)$ and $(U^{(n)}_{k+2}+1,0)$. We continue inductively and, when we reach 
$(U^{(n)}_\ell +1,0)$, we add the path obtained by the concatenation of the two
left-most geodesics from $(U^{(n)}_\ell +1,0)$ and from $(j,0)$ (see Fig.~\ref{fig:upper-distLHPT}). This construction yields a path 
from $(i,0)$ to $(j,0)$ with length smaller than $(\ell-k+1)(1+2\gamma\sqrt{n})$. Finally,
on the event where \eqref{spacing} holds, we have $(\ell-k-1)\eta n\leq j-i$, hence
$$\ell-k \leq \lfloor\frac{j-i}{\eta n}\rfloor +1\leq \big\lfloor \frac{A(j-i)}{2n}\big\rfloor +1,$$
giving the desired result.
\endproof

\medskip

We will now derive a result similar to the preceding proposition for the UIPT $\t_\infty^{(1)}$ of the $1$-gon.  Recall that we denote the length of  $\partial^* B^\bullet_{r}( \mathcal{T}^{(1)}_{\infty})$ by $L_{r}$. 

For every integer $n\geq 1$, we
write $u^{(n)}_0$ for a vertex chosen uniformly at random on $\partial^* B^\bullet_n(\t_\infty^{(1)})$,
and $u^{(n)}_1,\ldots,u^{(n)}_{L_n-1}$ for the other vertices of 
$\partial^* B^\bullet_n(\t_\infty^{(1)})$
enumerated in clockwise order, starting from $u^{(n)}_0$. We extend the definition 
of $u^{(n)}_i$ by periodicity, requiring that $u^{(n)}_i=u^{(n)}_{L_n+i}$
for every $i\in \Z$. Note that, for every $i\in\Z$, $u^{(n)}_i$ is
also uniformly distributed over $\partial^* B^\bullet_n(\t_\infty^{(1)})$.

\begin{proposition}
\label{densit2}
Let $\gamma\in(0,1/2)$ and $\delta>0$. For every integer $A\geq 1$, let
$H_{n,A}$ be the event where any left-most
geodesic to the root starting from a vertex of $\partial^* B^\bullet_n(\t_\infty^{(1)})$ coalesces
before time 
$\lfloor\gamma\, n\rfloor$ with one of the left-most geodesics to the root
starting from $u^{(n)}_{\lfloor kn^2/A\rfloor}$, $0\leq k \leq \lfloor n^{-2}L_{n} \, A\rfloor $. Then, we can
choose $A$ large enough so that, for every sufficiently large $n$, 
$$\P(H_{n,A})\geq 1-\delta.$$
\end{proposition}

\proof Recall the notation introduced before Proposition \ref{prop:compa}. We may assume
that the first tree in the forest $\wt\f^{(1)}_{n-\lfloor\gamma n\rfloor,n}$ 
is the tree rooted at the edge between $u^{(n)}_0$
and $u^{(n)}_1$. Then write $\wt\f^{(1)}_{n-\lfloor\gamma n\rfloor,n}
= (\tau^{(n)}_1,\ldots,\tau^{(n)}_{L_n})$. By the remarks of the
end of subsection \ref{sec:skeleton}, we know that, for $1\leq i<j\leq n$,
the left-most geodesics to the root from $u^{(n)}_i$ and from $u^{(n)}_j$ 
coalesce before time $\lfloor\gamma\, n\rfloor$ (possibly exactly at time $\lfloor\gamma\, n\rfloor$) as soon as all the trees 
$\tau^{(n)}_{i+1},\tau^{(n)}_{i+2},\ldots,\tau^{(n)}_{j}$ have height strictly smaller than
$\lfloor\gamma\, n\rfloor$. To verify that $H_{n,A}$ holds, it is therefore sufficient to verify that, for any $i\in\{1,\ldots,L_n\}$
there exists an index $k$ with $0\leq k \leq \lfloor n^{-2} L_{n}A\rfloor$
such that all trees $\tau^{(n)}_j$ with $\lfloor kn^2/A\rfloor\wedge i < j\leq
\lfloor kn^2/A\rfloor\vee i$ have height strictly smaller than $\lfloor\gamma\, n\rfloor$. Write $H'_{n,A}$ for the event where the
latter property holds.
By \eqref{eq:boundLs} and \eqref{localbound}  we can find $a >0$ such that the event 
$$ \{\lfloor an^2\rfloor< L_n
\leq \lfloor a^{-1}n^2\rfloor\}\cap \{\lfloor an^2\rfloor< L_{n-\lfloor \gamma n\rfloor}
\leq \lfloor a^{-1}n^2\rfloor\}$$ has probability at least $1- \delta/2$. On the other hand, if we want to bound the probability
of the intersection of the latter event with the complement of $H'_{n,A}$, Proposition \ref{prop:compa} shows that we may replace the 
forest $\wt\f^{(1)}_{n-\lfloor\gamma n\rfloor,n}$ by a forest of independent
Galton--Watson trees with offspring distribution $\theta$ truncated at height $\lfloor\gamma\, n\rfloor$
(at the cost of the multiplicative constant $C_1$). But then the desired result follows by the same arguments as in the first part of the proof of
Proposition \ref{upper-distLHPT}. \endproof

\section{First-passage percolation on the UIPT}
The geometric estimates gathered in the last sections will now be used to study the behavior of modified distances in random triangulations. As explained in the Introduction, we first concentrate on the case of the first-passage percolation $ \mathrm{d_{fpp}}$. After establishing an easy subadditive result (Proposition \ref{lem:subadditive}), we derive the key result of this section (Proposition \ref{dist-hull}), which deals with  the modified distance between the root vertex and
an arbitrary vertex of the boundary of the hull of radius $r$ in the UIPT.

\subsection{Subadditivity in the lower half-plane model}

We consider the LHPT $\l$, and,
conditionally on $\l$, we assign i.i.d.~random weights to the edges of $\l$.
We assume that the common distribution of these random variables is supported 
on the interval $[\kappa,1]$ for some $\kappa\in(0,1]$.
From these random weights, we define the first-passage percolation distance 
$\mathrm{d}_\mathrm{fpp}^\l$ as explained in the Introduction.

Recall our notation 
$\l_r=\{(i,-r):i\in\Z\}$ for the lower boundary of $\l_{[0,r]}$. Also recall that $\rho=(0,0)$ is the root vertex.

\begin{proposition}
\label{lem:subadditive}
There exists a constant $\mathbf{c}_0 \in [\kappa,1]$ such that
$$ r^{-1} \mathrm{d}_\mathrm{fpp}^\l\big(\rho, \l_r\big) \xrightarrow[r\to\infty]{\rm a.s.} \mathbf{c}_0.$$
\end{proposition}

\proof
For integers $0\leq m<n$, we define $\l_{[m,n]}$ as the part of $\l$ lying in the strip $\R\times [-n,-m]$. The first-passage percolation distance $\mathrm{d}_{\mathrm{fpp}}^{\l_{[m,n]}}$ on the vertex set 
of $\l_{[m,n]}$ is defined by considering the minimal weight of paths that stay
in $\l_{[m,n]}$ (thus, if $v$ and $v'$ are two vertices of $\l_{[m,n]}$ we have
$\mathrm{d}_\mathrm{fpp}^\l(v,v')\leq \mathrm{d}_{\mathrm{fpp}}^{\l_{[m,n]}}(v,v')$).

Then let $n,m\geq 1$ and let $x_m$ be the left-most vertex of $\l_m$
such that $\mathrm{d}_\mathrm{fpp}^\l\big(\rho, \l_m\big) = \mathrm{d}_\mathrm{fpp}^\l(\rho,x_m)$. We have
$$\mathrm{d}_\mathrm{fpp}^\l\big(\rho, \l_{m+n}\big)\leq \mathrm{d}_\mathrm{fpp}^\l\big(\rho, \l_m\big)
+ \mathrm{d}_{\mathrm{fpp}}^{\l_{[m,m+n]}}(x_m,\l_{m+n}).$$
Remark that $x_{m}$ is a function of $ \mathcal{L}_{[0,m]}$ only and that, by the independence of the layers in $ \mathcal{L}$, the random variable $\mathrm{d}_{\mathrm{fpp}}^{\l_{[m,m+n]}}(x_m,\l_{m+n})$
is independent of $\l_{[0,m]}$ and has the same distribution as $\mathrm{d}_\mathrm{fpp}^\l\big(\rho, \l_n\big)$. We can then 
 apply Liggett's version of Kingman's subadditive ergodic theorem \cite{Lig85} to get the statement of the
 proposition (the fact that the limit is constant is easy from a zero-one law argument, and the property
 $\kappa\leq \mathbf{c}_0\leq 1$ is obvious). \endproof

\subsection{From the lower half-plane to the UIPT}

We now discuss the first-passage percolation distance $\mathrm{d}_\mathrm{fpp}$  on the UIPT of the $1$-gon $\t_\infty^{(1)}$. We assume that this distance is defined in terms of i.i.d.~weights on the edges of the 
UIPT, these weights having the same distribution as in the
previous section. Note in particular that $\mathrm{d}_\mathrm{fpp}\leq \mathrm{d}_\mathrm{gr}$
since we assume that weights are bounded above by $1$. We still write $\mathbf{c}_0$ for the constant arising 
in Proposition \ref{lem:subadditive}. 

Let us state our main technical result, which builds upon Proposition \ref{lem:subadditive} and the geometric estimates of the last section. To simplify notation, we write
$B^\bullet_{n}=B^\bullet_{n}(\t_\infty^{(1)})$ and $ \partial^* B_{n}^\bullet = \partial^* B_{n}^\bullet(  \mathcal{T}^{(1)}_{\infty})$ only in this subsection.

\begin{proposition}
\label{keytech}
Let $\varepsilon\in(0,1)$ and $\delta>0$. We can find $\eta \in (0,1/2)$ 
such that, for every sufficiently large $n$, the property
$$(1-\varepsilon)\mathbf{c}_0\,\eta n \leq
\mathrm{d}_{\rm fpp}(v,\partial^* B^\bullet_{n-\lfloor \eta n\rfloor}) \leq 
(1+\varepsilon)\mathbf{c}_0\, \eta n\, ,\quad\forall v\in \partial^* B^\bullet_{n},$$
holds with probability at least $1-\delta$.
\end{proposition}

\proof 
Let us briefly outline the main steps of the proof.
Recall that $u^{(n)}_j$, for $0\leq j\leq L_n-1$, are the vertices of $\partial^* B^\bullet_{n}$
enumerated as explained before Proposition \ref{densit2}, and that
we have extended the definition of $u^{(n)}_j$ to all $j\in\Z$ by periodicity.
The first step of the proof if to use Proposition \ref{dist-bdry3} to observe that an FPP shortest path from $u^{(n)}_j$
(for some fixed $j$) to $\partial^* B^\bullet_{n-\lfloor \eta n\rfloor}$ 
that stays in $B^\bullet_n$ cannot ``turn around'' the layer $B^\bullet_n \backslash B^\bullet_{n-\lfloor \eta n\rfloor}$,
and more precisely that it must stay in the region bounded by the left-most geodesics coming 
from $u^{(n)}_{j-\lfloor cn^2\rfloor}$ and $u^{(n)}_{j+\lfloor cn^2\rfloor}$ respectively , for some $c>0$. Proposition \ref{prop:compa}
then allows us to compare the distribution of the trees of the skeleton of $B^\bullet_n \backslash B^\bullet_{n-\lfloor \eta n\rfloor}$
that code the latter region, with the distribution of independent 
Galton--Watson trees. This makes it possible to transfer the result known in the LHPT case (Proposition \ref{lem:subadditive}) 
to FPP distances in $B^\bullet_n \backslash B^\bullet_{n-\lfloor \eta n\rfloor}$. Finally, to get uniformity in the  starting vertex $u^{(n)}_j$, we use 
the coalescence property given in Proposition \ref{densit2}, which roughly speaking says that it is enough to consider 
a fixed number $A$ (large but independent of $n$) of values of $j$. 

Let us turn to the details of the argument. As a consequence of the bounds of Lemma \ref{bdlawperi}, we can fix $a\in(0,1/2)$ small enough (depending on $\delta$)
so that, for every $n\geq 2$, for every $\eta\in(0,1/2)$, the event 
$$\e_{n}(\eta) := \{\lfloor an^2\rfloor+1\leq L_{n}\leq \lfloor a^{-1}n^2\rfloor\}\cap \{\lfloor an^2\rfloor+1\leq L_{n-\lfloor\eta n\rfloor}\leq \lfloor a^{-1}n^2\rfloor\}$$ 
holds with probability at least $1 - \delta/4$.

For $\eta\in(0,1/2)$ and $j\in\Z$, let $\h_{n,j}(\eta)$ be the intersection of $\e_n(\eta)$ with the event where the leftmost geodesics 
starting respectively from $u^{(n)}_{j-\lfloor a n^2/4\rfloor}$ and from $u^{(n)}_{j+\lfloor a n^2/4\rfloor}$
do not coalesce before hitting $B^\bullet_{n-\lfloor \eta n\rfloor}$. By the considerations of the
end of subsection \ref{subsec:compa}, we know that, if $\eta>0$ is small enough, we have
for every sufficiently large $n$, for every $j\in\Z$,
\begin{equation}
\label{keytech22}
\P(\e_n(\eta)\cap (\h_{n,j}(\eta))^c)\leq a^2\delta/80.
\end{equation}
On the event 
$\h_{n,j}(\eta)$, we define $\g^{(n)}_j(\eta)$ as the subregion of $B^\bullet_n\backslash B^\bullet_{n-\lfloor\eta n\rfloor}$
that contains $u^{(n)}_j$ and is bounded on one side by the leftmost
geodesic from $u^{(n)}_{j-\lfloor a n^2/4\rfloor}$ and on the other side by the
leftmost geodesic from $u^{(n)}_{j+\lfloor a n^2/4\rfloor}$. 
We also write $\partial_\ell\g^{(n)}_j(\eta)$ for the part of the boundary of $\g^{(n)}_j(\eta)$ that
is contained in the union of the leftmost geodesics from $u^{(n)}_{j-\lfloor a n^2/4\rfloor}$ and 
from $u^{(n)}_{j+\lfloor a n^2/4\rfloor}$.

Let $\a_{n,j}(\eta)$ be the intersection of  $\h_{n,j}(\eta)$ with the event where, for some $i$ with $j-an^2/16\leq i\leq j+an^2/16$, there exists
a path from $u^{(n)}_i$ to $\partial_\ell\g^{(n)}_j(\eta)$ that stays in
$B^\bullet_n\backslash B^\bullet_{n-\lfloor \eta n\rfloor}$ and has length smaller
that $4\eta n/\kappa$ (recall that the weights of our first-passage percolation belong to $[\kappa,1]$).
 We claim that, by choosing $\eta\in (0, \frac{1}{2})$ even smaller if necessary, we have also, for every sufficiently large $n$ and for every $j\in\Z$,
\begin{equation}
\label{keytech23}
\P(\a_{n,j}(\eta))\leq a^2\delta/80.
\end{equation}
Let us prove this claim. Obviously it is enough to take $j=0$ (recall that $u^{(n)}_0$
is chosen uniformly at random over $\partial^* B^\bullet_n$). For every $i\in \Z$, write
$\tc^{(n,\lfloor\eta n\rfloor)}_i$ for the tree in the skeleton of
$B^\bullet_n\backslash B^\bullet_{n-\lfloor \eta n\rfloor}$ that is rooted at the edge between
$u^{(n)}_{i-1}$ and $u^{(n)}_{i}$. Then
note that, on the event $ \h_{n,j}(\eta)$,  the region $\g^{(n)}_j(\eta)$ is determined, as a 
planar map, by the trees $\tc^{(n,\lfloor\eta n\rfloor)}_i$ for $-\lfloor an^2/4\rfloor < i\leq \lfloor
an^2/4\rfloor$, and by the Boltzmann distributed triangulations with a boundary
used to fill in the slots associated with vertices of these trees at height strictly less than $\lfloor \eta n\rfloor$. 
Let $\f$ be a forest with height  $\lfloor\eta n\rfloor$ such that the number of trees in $\f$ and the number of vertices at generation $\lfloor\eta n\rfloor$
both lie between $\lfloor an^2\rfloor+1$
and $ \lfloor a^{-1}n^2\rfloor$. Proposition \ref{prop:compa} shows that the probability 
of the event where the finite collection of random trees $(\tc^{(n,\lfloor\eta n\rfloor)}_1,\ldots, \tc^{(n,\lfloor\eta n\rfloor)}_{L_n})$
is equal to $\f$ is bounded above, up to a multiplicative 
constant depending only on $a$, by the analogous probability for independent
Galton--Watson trees with offspring distribution $\theta$
truncated at generation $\lfloor\eta n\rfloor$. It follows that
the probability of the event $\a_{n,0}(\eta)$
can also be bounded by (a constant times) the 
probability of the similar event in the LHPT model. More precisely, $\P( \a_{n,0}(\eta))$
is bounded above, up to a multiplicative constant depending on $a$, by the probability that, in the half-plane model, there is a path going from $(i,0)$,
for some $i$ such that $-an^2/16\leq i\leq an^2/16$,
to the leftmost geodesic from $(\lfloor an^2/4\rfloor,0)$
(or from $(-\lfloor an^2/4\rfloor,0)$)  with length at most $4\eta n/\kappa$. If such a path exists, this implies 
that $\mathrm{d}_{\rm gr}^\l((i,0),(\lfloor an^2/4\rfloor,0))\leq 8\eta n/\kappa$. If we choose $\eta$ small, 
the second assertion of Proposition \ref{dist-bdry3}
shows that the probability of the latter event can be made arbitrarily small, uniformly for all $n\geq n_0$,
for some $n_0$. This gives our claim. 

From now on, we fix $\eta$ so that both \eqref{keytech22} and \eqref{keytech23} hold for $n$ large.
We set
$$\b_n(\eta):= \Bigg(\bigcap_{k=0}^{\lfloor 9a^{-2}\rfloor} \h_{n,k\lfloor an^2/8\rfloor}(\eta)\Bigg)
\cap \Bigg(\bigcap_{k=0}^{\lfloor 9a^{-2}\rfloor} (\a_{n,k\lfloor an^2/8\rfloor}(\eta))^c\Bigg).$$
We note that, for $n$ large,
\begin{align*}
\P(\b_n(\eta)^c)
&\leq \P((\e_n(\eta))^c)+ \sum_{k=0}^{\lfloor 9a^{-2}\rfloor} \P(\e_n(\eta)\cap (\h_{n,k\lfloor an^2/8\rfloor}(\eta))^c)
+ \sum_{k=0}^{\lfloor 9a^{-2}\rfloor} \P(\a_{n,k\lfloor an^2/8\rfloor}(\eta))\\
&\leq \frac{\delta}{4} +( 9 a^{-2}+1) \times \frac{a^2\delta}{80} + (9 a^{-2}+1) \times \frac{a^2\delta}{80}\\
&\leq \frac{\delta}{2},
\end{align*}
using \eqref{keytech22} and \eqref{keytech23}. 

Next,  consider, for every $i\in \Z$, the event
$$\mathcal{D}^{(n)}_i=\{\mathrm{d}^{(n)}_{\rm fpp}(u^{(n)}_i,\partial^* B^\bullet_{n-\lfloor \eta n\rfloor})
\in [(1-\varepsilon)\mathbf{c}_0 \eta n ,(1+\varepsilon)\mathbf{c}_0\eta n]\},$$
where the FPP distance $\mathrm{d}^{(n)}_{\rm fpp}$ is defined as $\mathrm{d}_{\rm fpp}$, 
but considering only paths that stay in $B^\bullet_n$. 

Let $k\in\{0,\ldots,\lfloor 9a^{-2}\rfloor\}$ and assume that $\h_{n,k\lfloor an^2/8\rfloor}(\eta)$
and $(\a_{n,k\lfloor an^2/8\rfloor}(\eta))^c$ both hold.  Then, for every $i$ such that
$k\lfloor an^2/8\rfloor - an^2/16\leq i\leq k\lfloor an^2/8\rfloor + an^2/16$, 
the minimal FPP length of a path from $u^{(n)}_i$ to $\partial^* B^\bullet_{n-\lfloor \eta n\rfloor}$
that stays in $B^\bullet_{n}\backslash
B^\bullet_{n-\lfloor \eta n\rfloor}$ can be evaluated by considering only paths that
stay in $\g^{(n)}_{k\lfloor an^2/8\rfloor }(\eta)$. Indeed, the FPP length of a path that would 
hit $\partial_\ell\g^{(n)}_{k\lfloor an^2/8\rfloor }(\eta)$ is at least $\kappa \times (4\eta n/\kappa)= 4\eta n$, 
whereas the (graph distance) geodesic from $u^{(n)}_i$ to $\partial^* B^\bullet_{n-\lfloor \eta n\rfloor}$
has FFP length bounded above by $\eta n$.
Hence, using Proposition \ref{prop:compa}, we
have (provided that $n$ is large enough), for every $i$ such that
$k\lfloor an^2/8\rfloor - an^2/16\leq i\leq k\lfloor an^2/8\rfloor + an^2/16$, 
\begin{equation}
\label{key30}
\P(\b_{n}(\eta)\cap (\mathcal{D}^{(n)}_i)^c)\leq \P( \h_{n,k\lfloor an^2/8\rfloor}(\eta)\cap(\a_{n,k\lfloor an^2/8\rfloor}(\eta))^c\cap (\mathcal{D}^{(n)}_i)^c)\leq   C_1\,\P(F_n)
\end{equation}
where $F_n$ denotes the event $\{\mathrm{d}^\l_{\rm fpp}((0,0),\l_{\lfloor \eta n\rfloor})
\notin [(1-\varepsilon)\mathbf{c}_0\eta n,(1+\varepsilon)\mathbf{c}_0 \eta n]\}$ and $C_1$
is a constant depending only on $a$. 
To derive the last bound, we note that, on the event $\h_{n,k\lfloor an^2/8\rfloor}(\eta)$, the properties 
considered in the event $(\a_{n,k\lfloor an^2/8\rfloor}(\eta))^c\cap (\mathcal{D}^{(n)}_i)^c$ can be expressed in terms of
the trees $\tc^{(n,\lfloor\eta n\rfloor)}_j$
for $k\lfloor an^2/8\rfloor-\lfloor an^2/4\rfloor< j\leq k\lfloor an^2/8\rfloor+ \lfloor an^2/4\rfloor$
(together with the triangulations filling in the slots), and we again use
the bound of Proposition \ref{prop:compa}.

Let $A$ be the integer found in Proposition \ref{densit2},
where we replace $\gamma$
by $\varepsilon \mathbf{c}_0\eta/2$ and $\delta$ by $\delta/4$. 
Using the conclusion of Proposition \ref{densit2}, we see that we have 
for all sufficiently large $n$,
\begin{align}
\label{keyinclu}
&\b_{n}(\eta)\cap \{\mathrm{d}^{(n)}_{\rm fpp}(v,\partial^* B^\bullet_{n-\lfloor\eta n\rfloor})
\notin [(1-2\varepsilon)\mathbf{c}_0\eta n,(1+2\varepsilon)\mathbf{c}_0\eta n],
\hbox{ for some } v\in\partial^* B^\bullet_n\}\nonumber\\
&\subset\b_{n}(\eta)\cap
\{\mathrm{d}^{(n)}_{\rm fpp}(u^{(n)}_{\lfloor jn^2/A\rfloor},\partial^* B^\bullet_{n-\lfloor \eta n\rfloor})
\notin [(1-\varepsilon)\mathbf{c}_0 \eta n,(1+\varepsilon)\mathbf{c}_0 \eta n],
\hbox{ for some } 0\leq j\leq \lfloor a^{-1}A\rfloor\},
\end{align}
except possibly on an event of probability at most $\delta/4$. The point is that
if we assume that $\e_n(\eta)$ holds (in particular if $\b_n(\eta)$ holds) but discard the set of probability at most $\delta/4$
considered in Proposition \ref{densit2}, any vertex of $\partial^* B^\bullet_{n}$ will be at 
$\mathrm{d}^{(n)}_{\rm fpp}$-distance smaller than $2 \times \varepsilon \mathbf{c}_0\eta\,n/2$
from one of the vertices $u^{(n)}_{\lfloor jn^2/A\rfloor}$, $0\leq j\leq \lfloor a^{-1}A\rfloor$.

Next, using  Proposition \ref{lem:subadditive}, we have, for all 
sufficiently large $n$,
$$C_1\,\P(F_n)\leq \frac{a\delta}{2(A+1)},$$
and it follows from \eqref{key30} that $\P(\b_{n}(\eta)\cap(\mathcal{D}^{(n)}_i)^c)\leq \frac{a\delta}{2(A+1)}$ for every $i\in\{0,1,\ldots,\lfloor a^{-1}n^2\rfloor\}$
(observe that $\lfloor 9 a^{-2}\rfloor \times \lfloor an^2/8\rfloor \geq \lfloor a^{-1}n^2\rfloor$ if $n$ is large).

Finally, the probability of the event in the right-hand side of \eqref{keyinclu} is
bounded above for $n$ large by
\begin{align*}
\sum_{j=0}^{\lfloor a^{-1}A\rfloor}
\P\Big(\b_{n}(\eta)
\cap (\mathcal{D}^{(n)}_{\lfloor jn^2/A\rfloor})^c\Big)
\leq (\lfloor a^{-1}A\rfloor+1) \times \frac{a\delta}{2(A+1)} \leq \frac{\delta}{2}.
\end{align*}
Recalling that $ \mathbb{P}(\b_{n}(\eta)^c)\leq \delta/2$, and using the last bound together with \eqref{keyinclu},
we arrive at the bound 
$$\P\Big(\mathrm{d}^{(n)}_{\rm fpp}(v,\partial^* B^\bullet_{n-\lfloor \eta n\rfloor})
\in [(1-2\varepsilon)\mathbf{c}_0\lfloor \eta n\rfloor,(1+2\varepsilon)\mathbf{c}_0\lfloor \eta n\rfloor],
\hbox{ for every } v\in\partial^* B^\bullet_n\Big)\geq  1-\delta.$$
Now note that we can replace $\mathrm{d}^{(n)}_{\rm fpp}$ by $\mathrm{d}_{\rm fpp}$ in the last bound, 
since clearly $\mathrm{d}_{\rm fpp}\leq \mathrm{d}^{(n)}_{\rm fpp}$, and, on the other hand, it is 
also true that, for every $v\in \partial^* B^\bullet_n$,
$$\mathrm{d}_{\rm fpp}(v,\partial^*B^\bullet_{n-\lfloor \eta n\rfloor})
\geq \min_{v'\in  \partial^* B^\bullet_n} \mathrm{d}^{(n)}_{\rm fpp}(v',\partial^* B^\bullet_{n-\lfloor \eta n\rfloor}).$$
This completes the proof.
\endproof

\begin{proposition}
\label{dist-hull}
For every  $\varepsilon\in(0,1)$, 
$$\P\Big((\mathbf{c}_0-\varepsilon)n\leq \mathrm{d}_{\rm fpp}(\rho,v)\leq (\mathbf{c}_0+\varepsilon)n,\ 
\hbox{ for every vertex } v\hbox{ in } \partial^* B^\bullet_n \Big) \build{\la}_{n\to\infty}^{} 1.$$
\end{proposition}

\begin{proof}
Fix $ \varepsilon\in(0,1)$ and let $\delta \in (0, \varepsilon/(4 |\log( \varepsilon/16)|))$. By  Proposition \ref{keytech}, we can fix $\eta \in (0, \frac{1}{4})$
such that, for every sufficiently large $n$, the event
$$\g_{n}:= \Big\{ (\mathbf{c}_0-\frac{\varepsilon}{2}) \lfloor \eta n\rfloor \leq
\mathrm{d}_{\rm fpp}(v,\partial^* B^\bullet_{n-\lfloor \eta n\rfloor}) \leq 
(\mathbf{c}_0+\frac{\varepsilon}{2})\lfloor \eta n\rfloor\, ,\ \forall v\in \partial^* B^\bullet_{n}\Big\},$$
holds with probability at least $1-\delta^2$.

Let $n\geq 1$. Set $n_{0} = n, n_{1} = n- \lfloor \eta n \rfloor$ and by induction $ n_{i} = n_{i-1} - \lfloor \eta\, n_{i-1} \rfloor$ for $i \geq 1$.  Set
$$q= \left\lfloor \frac { \log( \varepsilon/16)}{\log(1- \eta)} \right\rfloor$$ 
so that we have $n_{q} \leq \varepsilon n  /4$ for $n$ large enough. By our choice of $\eta$, we have $\E\Big[ \sum_{j=0}^{q-1} \mathbf{1}_{\g_{n_{j}}^c} \Big]\leq \delta^2\,q,$
as soon as $n$ is sufficiently large,
and the Markov inequality gives
$$\P\Bigg(\sum_{j=0}^{q-1} \mathbf{1}_{\g_{n_{j}}^c} > \delta q\Bigg) \leq \delta .$$
In what follows, we argue on the event
$$\h_n:= \Bigg\{ \sum_{j=0}^{q-1} \mathbf{1}_{\g_{n_{j}}^c} \leq \delta\,q\Bigg\}.$$

Let  $v\in  \partial^* B^\bullet_n$. We construct inductively a finite sequence $(v_{(j)})_{0\leq j\leq q}$, such that
$v_{(j)}\in \partial^* B^\bullet_{n_{j}}$, for every $0\leq j\leq q$. We start with 
$v_{(0)}=v$. Then, if we have constructed $v_{(0)},\ldots,v_{(j)}$ for some $0\leq j < q$, we define $v_{(j+1)}$
as follows. If the event $\g_{n_{j}}$ holds, we let $v_{(j+1)}$ be any point in 
$\partial^* B^\bullet_{n_{j+1}}$ such that 
$\mathrm{d}_{\rm fpp}(v_{(j)},v_{(j+1)})=\mathrm{d}_{\rm fpp}(v_{(j)},\partial^* B^\bullet_{n_{j+1}})$. 
Otherwise, we choose $v_{(j+1)}\in \partial^* B^\bullet_{n_{j+1}}$
such that $\mathrm{d}_{\rm gr}(v_{(j)},v_{(j+1)})= n_{j}-n_{j+1}$. We note that
$\mathrm{d}_{\rm fpp}(\rho,v_{(q)})\leq \mathrm{d}_{\rm gr}(\rho,v_{(q)})=n_q\leq \varepsilon n/4$. Hence, for $n$ large enough, we have
 on the event $\h_n$,
\begin{align*}
\mathrm{d}_{\rm fpp}(\rho,v) &\leq
\sum_{j= 0}^{q-1} \mathrm{d}_{\rm fpp}(v_{(j)},v_{(j+1)}) +\frac{\varepsilon n}{4}\\
&\leq
(\mathbf{c}_0+\frac{\varepsilon}{2})\sum_{j= 0}^{q-1} (n_{j}-n_{j+1})+ \delta q \max_{0\leq i<q} \{n_{i}-n_{i+1}\} + \frac{\varepsilon n}{4} \\ 
&\leq (\mathbf{c}_0+\frac{\varepsilon}{2})n + \delta q \eta \, n  + \frac{\varepsilon n}{4}\\
&\leq (\mathbf{c}_0+\varepsilon)\,n\,,
\end{align*}
where we used in the last line the fact that $ \eta\leq |\log(1-\eta)|$ for $\eta \in (0,1)$ to get that $\delta q \eta \leq \varepsilon/4$.

On the other hand, take any path $\omega$ from $v$ to $\rho$, and
for every integer $j\in\{0,1,\ldots,q\}$,
write  $w_{(j)}$ for the
last point of $\omega$ that belongs to $\partial^* B^\bullet_{n_{j}}$. Then, for $n$ large enough, on
the event $\h_n$, the weight of the path $\omega$ is bounded below by
\begin{align*}
\sum_{j= 0}^{q-1}  \mathrm{d}_{\rm fpp}(w_{(j)},\partial^* B^\bullet_{n_{j+1}})
&\geq (\mathbf{c}_0-\frac{\varepsilon}{2})(n_{0}-n_{q}) - \delta q \max_{0\leq i<q}\{ n_{i} - n_{i+1}\} \\
& \geq n(\mathbf{c}_0-\frac{\varepsilon}{2})(1 - \frac{ \varepsilon}{4}) -  \delta q \eta \, n \\
&\geq (\mathbf{c}_0-\varepsilon)\,n\,,
\end{align*}
since we have $ n_{q} \leq \varepsilon n /4$, $c_{0}\leq 1$ and $\delta q \eta \leq \varepsilon /4$. This implies that, on the event $\h_n$, we have 
$$\mathrm{d}_{\rm fpp}(\rho,v) \geq (\mathbf{c}_0-\varepsilon)\,n.$$
Since we have $\P((\h_n)^c)\leq \delta$, for all sufficiently large $n$, where $\delta$ can be taken arbitrarily small, this completes the proof.
\end{proof}

\section{First-passage percolation on finite triangulations}

For every $n\geq 1$, let $\t_n^{(1)}$ be uniformly distributed over $\T_{n,1}$. Recall from Fig.~\ref{fig:transform-root} that we can transform $ \mathcal{T}_{n}^{(1)}$ into a uniform rooted plane triangulation with $n+1$ vertices, which is denoted by $ \mathcal{T}_{n}$. We write $\rho_{n}$ for the root vertex
of $\t_n^{(1)}$. We assign i.i.d. weights to the edges of $\t_n^{(1)}$, with the same distribution as in the last section, and we write $\mathrm{d}_\mathrm{fpp}$ for the associated first-passage
percolation distance on the vertex set $\mathsf{V}(\t_n^{(1)})$. We also keep the notation $\mathbf{c}_0$ for the
constant in Proposition \ref{lem:subadditive}. 

\begin{proposition}
\label{key-finitetri}
Let $o_{n}$ be a uniformly distributed inner vertex of $ \t_{n}^{(1)}$. Then,
for every $\ve>0$,
$$\P\Big( \big|\mathrm{d}_\mathrm{fpp}(\rho_{n},o_{n}) 
-\mathbf{c}_0\, \mathrm{d}_\mathrm{gr}(\rho_{n},o_{n})\big| > \ve\,n^{1/4}\Big)
\build{\longrightarrow}_{n\to\infty}^{} 0.$$
\end{proposition}

The proof relies on certain absolute continuity relations between finite triangulations and the UIPT, which are similar to \cite[Section 4.3]{CLGplane}. We start with a preliminary lemma.
Recall our notation $\C_{1,r}$ for the set of all triangulations of the cylinder of height $r$ with bottom cycle of 
length $1$. If $\mathrm{t}\in\C_{1,r}$, we denote the total number of vertices of $\mathrm{t}$
by $N(\mathrm{t})+1$. 

We write $\ov{\t}_n^{(1)}$ for the triangulation $\t_n^{(1)}$ given with the distinguished vertex $o_{n}$. 
The hull $B^\bullet_r(\ov\t_n^{(1)})$ makes sense provided that $\mathrm{d}_\mathrm{gr}(\rho_{n},
o_{n})>r$, and otherwise we let $B^\bullet_r(\ov\t_n^{(1)})$ be the edge-triangulation. 

\begin{lemma}
\label{auxlem}
There exists a constant $\ov c$ such that, for every $n\geq 1$, for every $r\geq 1$ and
every $\mathrm{t}\in\C_{1,r}$ such that $n>N(\mathrm{t})$,
\begin{equation}
\label{law-hull-fini}
\P(B^\bullet_r(\ov\t_n^{(1)})=\mathrm{t})\leq \ov c \,\Big(\frac{n}{n-N(\mathrm{t})}\Big)^{3/2}\,
\P(B^\bullet_r(\t_\infty^{(1)})= \mathrm{t}).
\end{equation}
\end{lemma}

\proof Fix $r\geq 1$ and $\mathrm{t}\in\C_{1,r}$ and write $N=N(\mathrm{t})$ to
simplify notation. As a consequence of (\ref{eq:loihullUIPHT}) and of the fact that
$\t^{(1)}_\infty$ is the local limit of $\t^{(1)}_n$ as $n\to\infty$
(see the end of subsection \ref{subsec:skeletondecom}), we know that
\begin{equation}
\label{law-hull-infi}
\P(B^\bullet_r(\t_\infty^{(1)})= \mathrm{t}) = \frac{C(p)}{C(1)} \,(12\sqrt{3})^{-N},
\end{equation}
where $p=|\partial^*\mathrm{t}|$ is the length of the top cycle of $\mathrm{t}$. 
On the other hand,  \eqref{eq:hulldiscrete} gives the explicit formula 
$$\P(B^\bullet_r(\ov\t_n^{(1)})=\mathrm{t})= \frac{n-N}{n}\, \frac{\#\T_{n-N,p}}{\#\T_{n,1}}$$ 
from which, using 
Lemma \ref{lem:bound} and the asymptotics \eqref{eq:asymp} (with $p=1$), we derive the bound
$$\P(B^\bullet_r(\ov\t_n^{(1)})=\mathrm{t})\leq c^*\,C(p)\,\Big(\frac{n}{n-N}\Big)^{3/2}\,(12\sqrt{3})^{-N},$$
with some constant $c^*$. By comparing the last bound with \eqref{law-hull-infi}, we get the desired result.  \endproof

\begin{proof}[Proof of Proposition \ref{key-finitetri}] We fix $\ve>0$ and $\nu\in(0,1)$. We will prove that for all sufficiently large $n$
we have
$$\P\Bigg(\Big|\frac{\mathrm{d}_\mathrm{fpp}(\rho_{n},o_{n})}{\mathrm{d}_\mathrm{gr}(\rho_{n},o_{n})}
-\mathbf{c}_0\Big| >2\ve\Bigg) < \nu.$$
Since the Gromov--Hausdorff convergence of rescaled triangulations to the Brownian map \cite{LG11} 
(see Theorem \ref{convGHP} in Appendix A1 below)
implies that the sequence $n^{-{1/4}}\mathrm{d}_\mathrm{gr}(\rho_{n},o_{n})$ is bounded 
in probability, the statement of the proposition follows. 

Let $r\geq 1$. For every $\mathrm{t}\in\C_{1,r}$, we can equip the edges of $ \mathrm{t}$ with i.i.d.~weights
distributed as previously, and then consider the associated first-passage 
percolation distance $\mathrm{d}_\mathrm{fpp}$.
We let $a_\ve(\mathrm{t})$ be the random variable defined by $a_\ve(\mathrm{t})=1$ if
$$\sup_{x\in \partial^*\mathrm{t}} \Big| \frac{\mathrm{d}_\mathrm{fpp}(\rho,x)}
{\mathrm{d}_\mathrm{gr}(\rho,x)}- \mathbf{c}_0\Big| \geq \ve$$
and $a_\ve(\mathrm{t})=0$ otherwise (here $\rho$ stands for the root
vertex of $\mathrm{t}$ and we recall that $\partial^*\mathrm{t}$ is the top cycle of $\mathrm{t}$).
By convention we also define $a_\ve(\mathrm{t}_0)=0$, when $\mathrm{t}_0$ is the
edge-triangulation.

Let $b\in(0,1)$. We observe that 
\begin{align*}
\P\Big( \{a_\ve(B^\bullet_r(\ov\t_n^{(1)}))=1\}\cap \{\# B^\bullet_r(\ov\t_n^{(1)}) \leq (1-b)n\}\Big)
&=\sum_{\mathrm{t}\in \C_{1,r},\; N(\mathrm{t})+1\leq (1-b)n} \P(B^\bullet_r(\ov\t_n^{(1)})=\mathrm{t})\,
\P(a_\ve(\mathrm{t})=1)\\
&\leq \ov c\, b^{-3/2} \sum_{\mathrm{t}\in \C_{1,r}} \P(B^\bullet_r(\t_\infty^{(1)})=\mathrm{t})\,
\P(a_\ve(\mathrm{t})=1),
\end{align*}
using the bound \eqref{law-hull-fini}.
It follows that
$$\P\Big( \{a_\ve(B^\bullet_r(\ov\t_n^{(1)}))=1\}\cap \{\# B^\bullet_r(\ov\t_n^{(1)}) \leq (1-b)n\}\Big)
\leq \ov c\, b^{-3/2} \,\P(a_\ve(B^\bullet_r(\t_\infty^{(1)}))=1).
$$
Using Proposition \ref{dist-hull}, we now get
\begin{equation}
\label{bd-finitetri}
\lim_{r\to\infty} \Bigg( \sup_{n\geq 1} \;\P\Big( \{a_\ve(B^\bullet_r(\ov\t_n^{(1)}))=1\}\cap \{\# B^\bullet_r(\ov\t_n^{(1)}) <(1-b)n\}\Big)\Bigg) =0.
\end{equation}

Let us fix $0<\alpha<\beta<\gamma$. 
Write $ \mathfrak{B}_r(\t_n^{(1)},o_{n})$ for the ball of radius $r$ centered at $o_{n}$
in $\t_n^{(1)}$. For every $n\geq 1$, consider the event
$$D_{\beta,\gamma,n} :=\{\beta\,n^{1/4} < \mathrm{d}_\mathrm{gr}(\rho_{n},o_{n})\leq \gamma \,n^{1/4}\}.$$
For future use, we note that $B^\bullet_{\lfloor\alpha n^{1/4}\rfloor}(\ov\t_n^{(1)})$ is nontrivial on the event
$D_{\beta,\gamma,n}$.
We then observe that
$$\Big(D_{\beta,\gamma,n} \cap \{\#  \mathfrak{B}_{\lfloor(\beta-\alpha)n^{1/4}\rfloor}(\t_n^{(1)},o_{n}) > b\,n\}\Big)
\subset \{\#B^\bullet_{\lfloor\alpha n^{1/4}\rfloor}(\ov\t_n^{(1)})\leq (1-b)n\},$$
simply because if $D_{\beta,\gamma,n}$ holds, the whole ball
$ \mathfrak{B}_{(\beta-\alpha)n^{1/4}}(\t_n^{(1)},o_{n}) $ is contained in the complement of 
the hull $B^\bullet_{\lfloor \alpha n^{1/4} \rfloor}(\ov\t_n^{(1)})$. It then follows from \eqref{bd-finitetri} that
\begin{equation}
\label{conv-finitri}
\lim_{n\to\infty} \P\Big(\{a_\ve(B^\bullet_{\lfloor\alpha n^{1/4}\rfloor}(\ov\t_n^{(1)}))=1\}\cap D_{\beta,\gamma,n} \cap \{\#  \mathfrak{B}_{\lfloor(\beta-\alpha)n^{1/4}\rfloor}(\t_n^{(1)},o_{n}) > b\,n\}\Big) =0.
\end{equation}
On the other hand, given any $y<1$, we can choose $b\in(0,1)$ such that
\begin{equation}
\label{occup-snake}
\liminf_{n\to\infty} \P(\#  \mathfrak{B}_{\lfloor(\beta-\alpha)n^{1/4}\rfloor}(\t_n^{(1)},o_{n}) > b\,n) \geq y.
\end{equation}
This follows from the relation between $\t_n^{(1)}$ and $\t_n$, and  the well-known convergence
in distribution of the rescaled profile  of distances from 
a vertex chosen uniformly at random in $\t_n$ toward a random measure that gives positive mass to
every neighborhood of $0$: See Theorem 1(iii) in \cite{Mie08b}.
Since \eqref{occup-snake} holds for $y$ arbitrarily close to $1$, we deduce from 
\eqref{conv-finitri} that we have also
\begin{equation}
\label{conv-finitri2}
\lim_{n\to\infty} \P\Big(\{a_\ve(B^\bullet_{\lfloor\alpha n^{1/4}\rfloor}(\ov\t_n^{(1)}))=1\}\cap D_{\beta,\gamma,n}\Big) =0.
\end{equation}

To complete the argument, choose $\delta\in(0,1/2)$ such that $2\delta<\ve$. We will apply
\eqref{conv-finitri2} to 
$$ \alpha_j= j\,\delta^2,\;\beta_j=(j+1)\delta^2,\;\gamma_j=(j+2)\delta^2$$
for integers $j$ such that $\lfloor\delta^{-1}\rfloor< j\leq\lfloor \delta^{-3}\rfloor$. We observe that
$$\P\Bigg(\bigcup_{j=\lfloor \delta^{-1}\rfloor+1}^{ \lfloor \delta^{-3}\rfloor }
D_{\beta_j,\gamma_j,n}\Bigg)
= \P\Big( (\lfloor \delta^{-1}\rfloor+2)\delta^2\, n^{1/4} < \mathrm{d}_\mathrm{gr}(\rho_{n},o_{n})\leq
(\lfloor \delta^{-3}\rfloor+2)\delta^2\,n^{1/4}\Big).$$
Now recall that $n^{-1/4}\mathrm{d}_\mathrm{gr}(\rho_{n},o_{n})$ converges in distribution to a 
positive random variable (see e.g. \cite[Theorem 1.2(iii)]{MW08}). Hence,
by choosing $\delta$ smaller if needed, it follows from the previous display that we have for all 
sufficiently large $n$,
$$\P\Bigg(\bigcup_{j=\lfloor \delta^{-1}\rfloor+1}^{ \lfloor \delta^{-3}\rfloor }
D_{\beta_j,\gamma_j,n}\Bigg)\geq 1-\frac{\nu}{2}.$$
If we combine this with \eqref{conv-finitri2} (applied with $\alpha=\alpha_j,\beta=\beta_j,\gamma=\gamma_j$
for the relevant values of $j$), we get that, for $n$ large enough,
$$\P\Bigg(\bigcup_{j=\lfloor \delta^{-1}\rfloor+1}^{ \lfloor \delta^{-3}\rfloor }
\Big(\{a_\ve(B^\bullet_{\lfloor\alpha_j n^{1/4}\rfloor}(\ov\t_n^{(1)}))=0\}\cap D_{\beta_j,\gamma_j,n}\Big)\Bigg)\geq 1-\nu.$$
To complete the proof, we just need to verify that we have 
$$\Big|\frac{\mathrm{d}_\mathrm{fpp}(\rho_{n},o_{n})}{\mathrm{d}_\mathrm{gr}(\rho_{n},o_{n})}
-\mathbf{c}_0\Big| \leq 2\ve$$
on the event whose probability is considered in the previous display. Indeed, suppose that,
for some $j\in\{\lfloor \delta^{-1}\rfloor+1,\ldots, \lfloor \delta^{-3}\rfloor\}$, the event 
$\{a_\ve(B^\bullet_{\lfloor\alpha_j n^{1/4}\rfloor}(\ov\t_n^{(1)}))=0\}\cap D_{\beta_j,\gamma_j,n}$ holds. Then, clearly,
$$\mathrm{d}_\mathrm{fpp}(\rho_{n},o_{n}) \geq \min\big\{
\mathrm{d}_\mathrm{fpp}(\rho_{n},x) : x\in\partial^*B^\bullet_{\lfloor\alpha_j n^{1/4}\rfloor}(\ov\t_n^{(1)})\big\}  \geq (\mathbf{c}_0-\ve)\,\lfloor \alpha_j n^{1/4}\rfloor,$$
and it follows that
$$\frac{\mathrm{d}_\mathrm{fpp}(\rho_{n},o_{n})}{\mathrm{d}_\mathrm{gr}(\rho_{n},o_{n})}
\geq \frac{(\mathbf{c}_0-\ve)\,\lfloor \alpha_j n^{1/4}\rfloor}{\gamma_j n^{1/4}}
\geq \mathbf{c}_0-2\ve.$$
using the fact that $\frac{\alpha_j}{\gamma_j} = \frac{j}{j+2}\geq 1-\frac{2}{j}>1-2\delta>1-\ve$.
On the other hand,
still on the event $\{a_\ve(B^\bullet_{\lfloor\alpha_j n^{1/4}\rfloor}(\ov\t_n^{(1)}))=0\}\cap D_{\beta_j,\gamma_j,n}$,
we have 
\begin{align*}
\mathrm{d}_\mathrm{fpp}(\rho_{n},o_{n}) &\leq \Big( \max\big\{
\mathrm{d}_\mathrm{fpp}(\rho_{n},x) : x\in\partial^*B^\bullet_{\lfloor\alpha_j n^{1/4}\rfloor}(\ov\t_n^{(1)})\big\}  \Big)+ (\lfloor\gamma_j n^{1/4}\rfloor-\lfloor\alpha_j n^{1/4}\rfloor)\\
&\leq (\mathbf{c}_0+\ve)\,\lfloor \alpha_j n^{1/4}\rfloor+ (\lfloor\gamma_j n^{1/4}\rfloor-\lfloor\alpha_j n^{1/4}\rfloor)
\end{align*}
which implies
$$\frac{\mathrm{d}_\mathrm{fpp}(\rho_{n},o_{n})}{\mathrm{d}_\mathrm{gr}(\rho_{n},o_{n})}
\leq \frac{(\mathbf{c}_0+\ve)\,\lfloor \alpha_j n^{1/4}\rfloor+ (\lfloor\gamma_j n^{1/4}\rfloor-\lfloor\alpha_j n^{1/4}\rfloor)}{\beta_j n^{1/4}}\leq \mathbf{c}_0+2\ve.$$
This completes the proof. \end{proof}

In the next theorem, we deal with the uniform rooted  plane triangulation $\t_n$ with $n+1$ vertices. We 
equip the vertex set of $\t_n$ with the first-passage percolation distance $\mathrm{d}_\mathrm{fpp}$
defined as previously.

\begin{theorem}
\label{main-FPP-tri}
For every $\ve >0$, we have
$$\P\Bigg( \sup_{x,y\in \mathsf{V}(\t_n)}\big|\mathrm{d}_\mathrm{fpp}(x,y) 
-\mathbf{c}_0\, \mathrm{d}_\mathrm{gr}(x,y)\big| > \ve\,n^{1/4}\Bigg)
\build{\longrightarrow}_{n\to\infty}^{} 0.$$
\end{theorem}

\proof 
As mentioned above, we may assume that $\t_n$ and $\t^{(1)}_n$ are linked via the transformation of 
Fig.~\ref{fig:transform-root}. Then $\mathsf{V}(\t_n^{(1)})=\mathsf{V}(\t_n)$ and the graph 
distances are the same in $\mathsf{V}(\t_n^{(1)})$ and in $\mathsf{V}(\t_n)$. The root vertex $\rho_n$
is also the same in $\t_n$ and in $\t^{(1)}_n$. Furthermore, 
if $o_n$ stands for a uniformly distributed inner vertex of $\t_n^{(1)}$,
as in Proposition \ref{key-finitetri}, we can couple $o_n$  with a uniformly distributed vertex ${o}'_{n}$ of 
$\mathsf{V}(\t_{n})$, so that $\P(o_n=o'_n)=n/(n+1)$. Finally, we can assume that the FPP weights are the same
for all edges shared by $\t_n$ and $\t^{(1)}_n$ (that is, for all edges except for those involved in the 
transformation of 
Fig.~\ref{fig:transform-root}). It then follows from Proposition \ref{key-finitetri} that we have
\begin{equation}
\label{main-tri1}
\P\Big( \big|\mathrm{d}_\mathrm{fpp}(\rho_{n},o'_{n}) 
-\mathbf{c}_0\, \mathrm{d}_\mathrm{gr}(\rho_{n},o'_{n})\big| > \ve\,n^{1/4}\Big)
\build{\longrightarrow}_{n\to\infty}^{} 0,
\end{equation}
where the graph and FPP distances refer to $\t_n$. Indeed, on the event $\{o_n=o'_n\}$,
the graph distance $\mathrm{d}_\mathrm{gr}(\rho_{n},o'_{n})$ (in $\t_n$) 
is the same as the graph distance $\mathrm{d}_\mathrm{gr}(\rho_{n},o_{n})$ (in $\t_n^{(1)}$), and the
FPP distance $\mathrm{d}_\mathrm{fpp}(\rho_{n},o'_{n})$ (in $\t_n$) may only differ
from the FPP distance $\mathrm{d}_\mathrm{fpp}(\rho_{n},o_{n})$ (in $\t_n^{(1)}$) by a quantity bounded
in probability. 

Write $\ov\t_n$ for $\t_n$ pointed at $o'_n$. Conditionally on $\ov\t_n$, choose
an oriented edge $e_{n}$ of $\t_n$ uniformly at random. Then $\ov\t_n$ re-rooted at $e_{n}$
(and with the same distinguished vertex ${o}'_{n}$) has the same distribution as $\ov\t_n$.
Il follows that \eqref{main-tri1} still holds if $\rho_n$ is replaced by the root $\rho'_n$ of $e_n$.

Let $\vec{\mathsf{E}}(\t_n)$ be the set of all oriented edges of $\t_n$, and for 
$e\in \vec{\mathsf{E}}(\t_n)$ let $e_*$ denote the initial vertex of $e$. Since $\#\vec{\mathsf{E}}(\t_n)=6(n-1)$ and
$\#\mathsf{V}(\t_n)=n+1$, the probability in
\eqref{main-tri1} (with $\rho_n$ replaced by $\rho'_n$) can be rewritten as
$$\E\Bigg[ \frac{1}{n+1}\,\frac{1}{6(n-1)} \sum_{v\in \mathsf{V}(\t_n)}\sum_{e\in \vec{\mathsf{E}}(\t_n)}
\mathbf{1}_{\{ |\mathrm{d}_\mathrm{fpp}(e_*,v) 
-\mathbf{c}_0\, \mathrm{d}_\mathrm{gr}(e_*,v)| > \ve\,n^{1/4}}\Bigg]$$
and, since any vertex of $\t_n$ can be written as $e_*$ for 
(at least) one choice of $e$, this is bounded below by
$$\frac{1}{6(n+1)^2}\E\Bigg[  \sum_{v\in \mathsf{V}(\t_n)}\sum_{\tilde v\in \mathsf{V}(\t_n)}
\mathbf{1}_{\{ |\mathrm{d}_\mathrm{fpp}(v,\tilde v) 
-\mathbf{c}_0\, \mathrm{d}_\mathrm{gr}(v,\tilde v)| > \ve\,n^{1/4}}\Bigg].$$
Therefore the quantity in the last display also tends to $0$ as $n\to\infty$. This is just saying that, if
$o''_{n}$ is another vertex of $\mathsf{V}(\t_n)$, which conditionally on $\ov\t_n$ is uniformly distributed over $\mathsf{V}(\t_n)$,
we have also, for every $\ve>0$,
\begin{equation}
\label{main-tri3}
\P\Big( \big|\mathrm{d}_\mathrm{fpp}(o'_{n}, o''_{n}) 
-\mathbf{c}_0\, \mathrm{d}_\mathrm{gr}(o'_{n}, o''_{n})\big| > \ve\,n^{1/4}\Big)
\build{\longrightarrow}_{n\to\infty}^{} 0.
\end{equation}

Consider then, conditionally on $\t_n$, a sequence $(o_{n}^i)_{i\geq 1}$ of
vertices chosen independently at random uniformly over $\mathsf{V}(\t_n)$. Given any fixed $\delta>0$, we
can choose an integer $N\geq 1$ large enough such that, for every $n$,
\begin{equation}
\label{densite-uniform}
\P\Big( \sup_{x\in \mathsf{V}(\t_n)}\Big(\inf_{1\leq j\leq N} \mathrm{d}_\mathrm{gr}(x,o^j_{n})
\Big) < \ve \,n^{1/4}\Big) > 1- \delta.
\end{equation}
This essentially follows from  the convergence of rescaled triangulations to the Brownian map
obtained in \cite{LG11}. We provide a detailed proof
of \eqref{densite-uniform} in Appendix A1. 

Once $N$ is fixed, we deduce from 
\eqref{main-tri3} that we have also, for all sufficiently large $n$,
$$\P\Bigg(\bigcap_{1\leq i\leq j\leq N} \{ \big|\mathrm{d}_\mathrm{fpp}(o^i_{n}, o^j_{n}) 
-\mathbf{c}_0\, \mathrm{d}_\mathrm{gr}(o^i_{n},o^j_{n})\big| \leq \ve\,n^{1/4}\}\Bigg)
>1-\delta.$$

To complete the proof, just observe that
\begin{align*}
\sup_{x,y\in \mathsf{V}(\t_n)}\big|\mathrm{d}_\mathrm{fpp}(x,y) 
-\mathbf{c}_0\, \mathrm{d}_\mathrm{gr}(x,y)\big|
&\leq \sup_{i,j\in\{1,\ldots,N\}} \big|\mathrm{d}_\mathrm{fpp}(o^i_{n}, o^j_{n}) 
-\mathbf{c}_0\, \mathrm{d}_\mathrm{gr}(o^i_{n},o^j_{n})\big|\\ 
&\  +\,
4\, \sup_{x\in \mathsf{V}(\t_n)}\Big(\inf_{1\leq j\leq N} \mathrm{d}_\mathrm{gr}(x,o^j_{n})
\Big)
\end{align*}
and the preceding two displays show that the right-hand side is bounded above by 
$5\ve\,n^{1/4}$ outside a set of probability at most $2\delta$, for all sufficiently large $n$. 
\endproof

We now return to the UIPT $\t_\infty$, which we equip with the first-passage percolation
distance $\mathrm{d}_\mathrm{fpp}$. We write $B^{\rm fpp}_r(\t_\infty)$
for the ball of radius $r$ in $\t_\infty$ for the first-passage percolation
distance: This ball may be defined as the union of all faces of the UIPT that are incident
to a vertex at $\mathrm{d}_\mathrm{fpp}$-distance strictly less than $r$ from the root.

\begin{theorem}
\label{balls-uipt}
Let $\ve\in(0,1)$.
We have
$$\lim_{r\to\infty} \P\Bigg(\sup_{x,y\in \mathsf{V}(B_r(\t_\infty))} \big|\mathrm{d}_\mathrm{fpp}(x,y) 
-\mathbf{c}_0\, \mathrm{d}_\mathrm{gr}(x,y)\big| > \ve r\Bigg) = 0.$$
Consequently,
$$ \mathbb{P}\Big(B_{(1-\ve)r/\mathbf{c}_0}(\t_\infty) \subset B^\mathrm{fpp}_{r}(\t_\infty)
\subset B_{(1+\ve)r/\mathbf{c}_0}(\t_\infty)\Big) 
\build{\la}_{r\to\infty}^{} 1.$$
\end{theorem}

\proof The second part of the theorem is an easy consequence of the first one, and so we
concentrate on the first assertion. By the same arguments as in the beginning of the proof of
Theorem \ref{main-FPP-tri}, it is enough to prove the desired result 
with $ \mathcal{T}_{\infty}$ replaced by $ \mathcal{T}_{\infty}^{(1)}$. We fix $\delta>0$, which can be taken arbitrarily small. Consider first an arbitrary 
(deterministic) rooted planar map
$\mathrm{m}$ with root vertex $\rho$. We equip the vertex set $\mathsf{V}(\mathrm{m})$ with
the graph distance $\mathrm{d}_\mathrm{gr}$ and with the (random)
first-passage percolation
distance $\mathrm{d}_\mathrm{fpp}$. For every integer $r>0$, we say that
$\mathrm{m}\in A^{(\delta)}_r$ if the property
$$\big|\mathrm{d}_\mathrm{fpp}(x,y) 
-\mathbf{c}_0\, \mathrm{d}_\mathrm{gr}(x,y)\big|\leq \ve r\,\hbox{ for all }x,y\in \mathsf{V}(\mathrm{m})
\hbox{
such that }\mathrm{d}_\mathrm{gr}(\rho,x)\vee \mathrm{d}_\mathrm{gr}(\rho,y)\leq r,$$
holds with probability at least $1-\delta$. 

In order to prove the first assertion of the theorem, it is enough to verify that, for
every fixed integer $K\geq 1$, we have for all sufficiently large integers $r$,
$$\P(B^\bullet_{Kr}(\t^{(1)}_\infty) \in A^{(\delta)}_r) \geq 1-\delta.$$
The point is that if $K$ is chosen sufficiently large, if $x$ and $y$ are
two vertices of $B_r(\t^{(1)}_\infty)$, the first-passage percolation distance (resp.~the
graph distance) between $x$ and $y$
in the graph $\t^{(1)}_\infty$ coincides with the first-passage percolation distance
(resp.~the graph distance) between $x$ and $y$ in the graph $B^\bullet_{Kr}(\t^{(1)}_\infty)$.

So let us fix $K\geq 1$. Recall our notation $1+N(\mathrm{t})$ for the total number
of vertices of $\mathrm{t}\in \C_{1,r}$.
We first observe that, by  \cite[Theorem 2]{CLGpeeling}, we can choose two positive integers $\alpha>1$ and $\beta>1$
such that, if $D_r:=\{\mathrm{t}\in \C_{1,Kr}: N(\mathrm{t})\geq \alpha r^4\hbox{ or } N(\mathrm{t})\leq \alpha^{-1} r^4 \hbox{ or } |\partial^*\mathrm{t}| \geq \beta r^2\}$,
we have 
$$\P(B^\bullet_{Kr}(\t^{(1)}_\infty)\in D_r) \leq \frac{\delta}{2}.$$
Note that Theorem 2 in \cite{CLGpeeling} deals with the type II UIPT, but the last section of \cite{CLGpeeling} explains that
a similar result holds for the type I triangulations that we consider here.

Then, if $\mathrm{t}\in \C_{1,Kr}\backslash D_r$, it follows from formula \eqref{eq:hulldiscrete}, using both assertions of Lemma  \ref{lem:bound}, that the quantity $\P(B^\bullet_{Kr}(\ov\t_{2\alpha r^4}^{(1)})=\mathrm{t})$ 
is bounded below by 
 $$\frac{c'\,C(|\partial^*\mathrm{t}|)\,(2\alpha r^4-N(\mathrm{t}))^{-3/2}\,(12\sqrt{3})^{2\alpha r^4-N(\mathrm{t})}}
 {c\,C(1)\,(2\alpha r^4)^{-3/2}\,(12\sqrt{3})^{2\alpha r^4}}
 \geq \frac{c'}{c}\, \frac{C(|\partial^*\mathrm{t}|)}{C(1)}\,(12\sqrt{3})^{-N(\mathrm{t})}
 =\frac{c'}{c}\,\P(B^\bullet_{Kr}(\t^{(1)}_\infty)=\mathrm{t})
 ,$$
 where the last equality is \eqref{law-hull-infi}. Summarizing, we have obtained the existence of 
 a constant $c''>0$ such that, for every $\mathrm{t}\in \C_{1,Kr}\backslash D_r$,
 $$\P(B^\bullet_{Kr}(\ov\t_{2\alpha r^4}^{(1)})=\mathrm{t})\geq c''\,\P(B^\bullet_{Kr}(\t^{(1)}_\infty)=\mathrm{t}).$$
 It follows that
\begin{align*}
\P(B^\bullet_{Kr}(\t^{(1)}_\infty) \notin A^{(\delta)}_r) &\leq 
 \P(B^\bullet_{Kr}(\t^{(1)}_\infty)\in D_r) + \sum_{\mathrm{t}\in (A^{(\delta)}_r)^c\cap (\C_{1,Kr}\backslash D_r)}
 \P(B^\bullet_{Kr}(\t^{(1)}_\infty)=\mathrm{t})\\
 &\leq \frac{\delta}{2} + (c'')^{-1}
  \sum_{\mathrm{t}\in (A^{(\delta)}_r)^c\cap (\C_{1,Kr}\backslash D_r)}
 \P(B^\bullet_{Kr}(\ov\t_{2\alpha r^4}^{(1)})=\mathrm{t})\\
 &\leq \frac{\delta}{2} + (c'')^{-1}\,\P(B^\bullet_{Kr}(\ov\t_{2\alpha r^4}^{(1)})\notin A^{(\delta)}_r).
 \end{align*}
 However, Theorem \ref{main-FPP-tri}, or more precisely the equivalent statement for triangulations of the $1$-gon, tells us that $\P(B^\bullet_{Kr}(\ov\t_{2\alpha r^4}^{(1)})\notin A^{(\delta)}_r)$
 tends to $0$ as $r\to\infty$. We thus obtain that $\P(B^\bullet_{Kr}(\t^{(1)}_\infty) \notin A^{(\delta)}_r) <\delta$
 for all sufficiently large $r$, which was the desired result. \endproof
 
\section{Dual and Eden distances}

In this section, we discuss variants of Theorems \ref{main-FPP-tri} and \ref{balls-uipt} that hold
for particular choices of distances on the dual graph associated with finite triangulations,
or with the UIPT.  These choices correspond to the distances
$ \mathrm{d}_{ \mathrm{gr}}^\dagger$ or $ \mathrm{d}_{ \mathrm{Eden}}^\dagger$ introduced in cases 1.~and 2.~discussed
in the Introduction. The main technical difference when dealing with these distances on the dual graph comes from the lack of an priori bound (both from above and 
from below) for the modified distance by a multiple of the graph distance. 

\subsection{Statement of the results and identification of the constants}

Recall our notation $\t_n$ for a uniformly distributed rooted plane triangulation
with $n+1$ vertices, and $\mathsf{F}(\t_n)$ for the set of all faces of $\t_n$. The dual graph distance 
on $\mathsf{F}(\t_n)$ is denoted by $\mathrm{d}^\dagger_{\mathrm{gr}}$. As in the Introduction, we also define the 
Eden distance $\mathrm{d}^\dagger_{\mathrm{Eden}}$ on $\mathsf{F}(\t_n)$ by assigning independently to every dual edge 
an exponential weight with parameter $1$, and considering the associated first-passage percolation distance. 

We now state our analog of Theorem \ref{main-FPP-tri} for these distances.

\begin{theorem}
\label{main-tri-dual} There exist two constants $ \mathbf{c}_{1}, \mathbf{c}_{2} \in (0, \infty)$ such that for every $\ve >0$, we have
\begin{align*}
\P\Bigg( \build{\sup_{ x, y \in \mathsf{V}(\t_{n}),f,g\in \mathsf{F}(\t_n)}}_{x\triangleleft f,y\triangleleft g}^{} 
\big|\mathrm{d}^\dagger_\mathrm{gr}(f,g) 
-  \mathbf{c}_{1}\, \mathrm{d}_\mathrm{gr}(x,y)\big| > \ve\,n^{1/4}\Bigg)
& \xrightarrow[n\to\infty]{} 0,\\
\P\Bigg( \build{\sup_{ x, y \in \mathsf{V}(\t_{n}),f,g\in \mathsf{F}(\t_n)}}_{x\triangleleft f,y\triangleleft g}^{} 
\big|\mathrm{d}^\dagger_\mathrm{Eden}(f,g) 
-  \mathbf{c}_{2}\, \mathrm{d}_\mathrm{gr}(x,y)\big| > \ve\,n^{1/4}\Bigg)
& \xrightarrow[n\to\infty]{} 0,
\end{align*}
where we recall that the notation $x\triangleleft f$ means that the vertex $x$ is incident to the face $f$.
\end{theorem}
Combining the above theorem with Theorem \ref{main-FPP-tri} and the known convergence of rescaled triangulations towards the Brownian map we obtain the following joint convergences. If $(E,d)$ is a metric space and $\alpha >0$, we use the notation $\alpha \cdot (E,d)=(E,\alpha d)$.

\begin{corollary} \label{cor:brownianmap} Let $(\mathbf{m}_\infty, D^*)$ be the Brownian map. The following convergences in distribution 
$$\begin{array}{rcl}
3^{1/4}n^{-1/4}\cdot (\mathsf{V}(\t_n),\,\mathrm{d}_\mathrm{gr}) &\xrightarrow[n\to\infty]{ \rm (d)} &(\mathbf{m}_\infty, D^*)\\
3^{1/4}n^{-1/4}\cdot (\mathsf{V}(\t_n),\, \mathrm{d}_\mathrm{fpp}) &\xrightarrow[n\to\infty]{ \rm (d)} &\mathbf{c_0}\cdot (\mathbf{m}_\infty, D^*)\\
3^{1/4}n^{-1/4}\cdot (\mathsf{F}(\t_n),\, \mathrm{d}^\dagger_\mathrm{gr}) &\xrightarrow[n\to\infty]{ \rm (d)} &\mathbf{c_1}\cdot (\mathbf{m}_\infty, D^*)\\
3^{1/4}n^{-1/4}\cdot (\mathsf{F}(\t_n),\, \mathrm{d}^\dagger_\mathrm{Eden})&\xrightarrow[n\to\infty]{ \rm (d)} &\mathbf{c_2}\cdot (\mathbf{m}_\infty, D^*)
\end{array}$$
hold jointly (with the same limit), in the the space of all isometry classes of compact metric spaces equipped with the Gromov--Hausdorff  distance.
\end{corollary}

The first convergence in distribution of the corollary is proved in \cite{LG11}. The other convergences, and the fact that 
they hold jointly with the first one, then follow from Theorems \ref{main-FPP-tri} and \ref{main-tri-dual}.

\begin{remark} As explained below in Appendix A1, the first convergence of the corollary holds in the stronger
sense of the Gromov--Hausdorff--Prokhorov measure, if the sets $\mathsf{V}(\t_n)$ are equipped with the 
uniform probability measure and the Brownian map with its volume measure. It follows that the second
convergence also holds in this stronger sense. The same should be true for the last two ones, but
verifying this would require some additional work.
\end{remark}

We can also state an analog of Theorem \ref{balls-uipt}. We let $\mathsf{F}(\t_\infty)$ be the set
of all faces of the UIPT $\t_\infty$, and we equip this set with the dual graph distance 
$\mathrm{d}^\dagger_{\mathrm{gr}}$
 and with the Eden distance $\mathrm{d}^\dagger_{\mathrm{Eden}}$ defined as above from independent 
 exponential weights on the dual edges. 
  
For every $r\geq 0$, we let $B^\mathrm{dual}_r(\t_\infty)$ (resp. $B^\mathrm{Eden}_r(\t_\infty)$) be the union of all faces of $\t_\infty$
 at dual graph distance (resp. at Eden distance) less than or equal to $r$ from the root face (by definition, the root face is the face lying 
 on the right  of the root edge).
 
 \begin{theorem}
\label{balls-uipt-dual}
Let $\ve\in(0,1)$. With the same constants $ \mathbf{c}_{1}$ and $ \mathbf{c}_{2}$ as in Theorem \ref{main-tri-dual} we have 
\begin{align*}
\lim_{r\to\infty} \P\Bigg(\build{\sup_{ x, y \in \mathsf{V}(B_r(\t_\infty)),f,g\in \mathsf{F}(B_r(\t_\infty))}}_{x\triangleleft f,y\triangleleft g}^{}  \big|\mathrm{d}^\dagger_\mathrm{gr}(f,g) 
- \mathbf{c}_{1}\, \mathrm{d}_\mathrm{gr}(x,y)\big| > \ve r\Bigg) &= 0,\\
\lim_{r\to\infty} \P\Bigg(\build{\sup_{ x, y \in \mathsf{V}(B_r(\t_\infty)),f,g\in \mathsf{F}(B_r(\t_\infty))}}_{x\triangleleft f,y\triangleleft g}^{}  
\big|\mathrm{d}^\dagger_\mathrm{Eden}(f,g) 
-  \mathbf{c}_{2}\, \mathrm{d}_\mathrm{gr}(x,y)\big| > \ve r\Bigg) &= 0.
\end{align*}
Consequently,
\begin{align*}
 \mathbb{P}\Big(B_{(1-\ve)r/ \mathbf{c}_{1}}(\t_\infty) \subset B^\mathrm{dual}_{r}(\t_\infty)
\subset B_{(1+\ve)r/ \mathbf{c}_1}(\t_\infty)\Big) 
&\build{\la}_{r\to\infty}^{} 1,\\
 \mathbb{P}\Big(B_{(1-\ve)r/ \mathbf{c}_{2}}(\t_\infty) \subset B^\mathrm{Eden}_{r}(\t_\infty)
\subset B_{(1+\ve)r/ \mathbf{c}_{2}}(\t_\infty)\Big) 
&\build{\la}_{r\to\infty}^{} 1.
\end{align*}
\end{theorem}

\paragraph{Identification of the constants.} The proofs of Theorems \ref{main-tri-dual} and \ref{balls-uipt-dual} will be given in the next subsections, but
we immediately explain how the values of the constants $ \mathbf{c}_{1}$ and $  \mathbf{c}_{2}$ can be derived by combining Theorem 
\ref{balls-uipt-dual} with the results of \cite{CLGpeeling}. 

\begin{theorem} \label{identification} The constants $ \mathbf{c}_{1}$ and $  \mathbf{c}_{2}$ of Theorems \ref{main-tri-dual} and \ref{balls-uipt-dual} are given by
$$ \mathbf{c}_{1} = 1 + 2 \sqrt{3} \quad \mbox{ and } \quad \mathbf{c}_{2} = 2 \sqrt{3}.$$
\end{theorem}
\proof We rely on results of \cite{CLGpeeling} on asymptotics of the volume of hulls. For every $r>0$, write 
$B^{\bullet,\mathrm{dual}}_r(\t_\infty)$ for the hull associated with $B^{\mathrm{dual}}_r(\t_\infty)$ (that is, the complement of the
unbounded connected component of the complement of $B^{\mathrm{dual}}_r(\t_\infty)$), and similarly $B^{\bullet,\mathrm{Eden}}_r(\t_\infty)$ for the hull associated with $B^{\mathrm{Eden}}_r(\t_\infty)$.
Let $|B^\bullet_r(\t_\infty)|$, resp. $|B^{\mathrm{dual}}_r(\t_\infty)|$, resp. $|B^{\bullet,\mathrm{Eden}}_r(\t_\infty)|$, stand for the volume (number of faces) of $B^\bullet_r(\t_\infty)$,
resp. $B^{\bullet,\mathrm{dual}}_r(\t_\infty)$, resp. $B^{\bullet,\mathrm{Eden}}_r(\t_\infty)$.

By \cite[Theorems 2,3 and 4]{CLGpeeling}, we have 
\begin{eqnarray}
\label{conv-hulls}
\Big(n^{-4}|B^\bullet_{\lfloor nt\rfloor}(\t_\infty)|\Big)_{t\geq 0}  &\xrightarrow[n\to\infty]{ \mathrm{(d)}} & \Big(\frac{64}{3}\, \mathcal{M}_t\Big)_{t\geq 0}\\
\label{conv-hulls-dual}
\Big(n^{-4}|B^{\bullet,\mathrm{dual}}_{\lfloor nt\rfloor}(\t_\infty)|\Big)_{t\geq 0} &\xrightarrow[n\to\infty]{\rm (d)}& \Big((1+2\sqrt{3})^{-4}\frac{64}{3}\, \mathcal{M}_t\Big)_{t\geq 0},\\
\label{conv-hulls-Eden}
\Big(n^{-4}|B^{\bullet,\mathrm{Eden}}_{\lfloor nt\rfloor}(\t_\infty)|\Big)_{t\geq 0} & \xrightarrow[n\to\infty]{\rm (d)}& \Big(\frac{4}{27}\, \mathcal{M}_t\Big)_{t\geq 0}.
\end{eqnarray}
where the process $\mathcal{M}_t$, which is defined in the introduction of \cite{CLGHull,CLGpeeling}, satisfies the scaling property
$$(\mathcal{M}_{\lambda t})_{t\geq 0}\build{=}_{}^{\rm(d)} \lambda^4 (\mathcal{M}_t)_{t\geq 0},$$ for every $\lambda>0$. 
Note that Theorems 2,3 and 4 in \cite{CLGpeeling} give the preceding convergences in the case of the UIPT of type II, but Section 6.1
in \cite{CLGpeeling} explains how theses statements can be extended to the UIPT of type I, and gives the values of the constants
arising in this case.

On the other hand, \eqref{conv-hulls} and Theorem \ref{balls-uipt-dual} imply that the limit in distribution 
(at least in the sense of finite-dimensional marginals) of $(n^{-4}|B^{\bullet,\mathrm{Eden}}_{\lfloor nt\rfloor}(\t_\infty)|)_{t\geq 0}$
is the process $(\frac{64}{3}\, \mathcal{M}_{t/\mathbf{c}_2})_{t\geq 0}$, which has the same distribution as
$(\frac{64}{3}\,(\mathbf{c}_2)^{-4}\, \mathcal{M}_{t})_{t\geq 0}$. Comparing with \eqref{conv-hulls-Eden}, we get that
$\frac{64}{3}\,(\mathbf{c}_2)^{-4}= \frac{4}{27}$, hence $\mathbf{c}_2= 2\sqrt{3}$. Similarly, the value 
$\mathbf{c}_1= 1 + 2\sqrt{3}$ is derived by comparing \eqref{conv-hulls} and \eqref{conv-hulls-dual}.  \endproof

In the remaining part of this section, we explain the proof of Theorems \ref{main-tri-dual}
and \ref{balls-uipt-dual}. The general outline is the same as for Theorems \ref{main-FPP-tri} and \ref{balls-uipt},
but some additional ingredients are needed.

\subsection{Preliminary estimates}

As in the previous sections, we discuss the UIPT before considering finite triangulations. In order to get upper bounds
on the distances $\mathrm{d}^\dagger_\mathrm{gr}$ and $\mathrm{d}^\dagger_\mathrm{Eden}$, it will be convenient
to consider certain special paths in $\mathsf{F}(\t^{(1)}_\infty)$. For every $r\geq 1$, we let $\mathsf{F}_r(\t^{(1)}_\infty)$
be the set of all downward triangles at height $r$ in $\t^{(1)}_\infty$. A {\em downward path} is a dual path $\omega$
that starts from some $f_0\in \mathsf{F}_r(\t^{(1)}_\infty)$ and ends at the bottom face, which is constructed in the
following way, see Fig.~\ref{fig:downward-path}. Let $v_0$ be the unique vertex of $\partial^* B^\bullet_{r-1}(\t^{(1)}_\infty)$ that is incident to $f_0$, let $e_0$ be the edge of $\partial^* B^\bullet_r(\t^{(1)}_\infty)$ incident to $f_0$ and let 
$e_0,e_1,\ldots$ be the sequence of edges of $\partial^* B^\bullet_r(\t^{(1)}_\infty)$ listed in counterclockwise order from $e_0$  (recall our orientation of 
the cycles $\partial^* B^\bullet_{j}(\t^{(1)}_\infty)$ in clockwise order). Let $e_N$ ($N\geq 0$) be the first one in this list that has at least one child in the skeleton of $B^\bullet_r(\t^{(1)}_\infty)$,
and let $e'$ be the unique edge of $\partial^* B^\bullet_{r-1}(\t^{(1)}_\infty)$ whose terminal vertex is $v_0$. Let $f_0,f_1,\ldots,f_N$ and $f'$ be the downward triangles associated respectively with $e_0,\ldots,e_N$ and $e'$. Notice that $v_0$ is incident to all the faces $f_0,f_1,\ldots,f_N$ and $f'$.
Our dual path $\omega$ will visit successively the faces $f_0,f_1,\ldots,f_N$ and $f'$. Between the visits of $f_i$ and $f_{i+1}$,
for $0\leq i\leq N-1$, or the visits of $f_N$ and $f'$, the path ``crosses'' the slot of boundary size $2$ (or size $c_{e_N}+2$
for the last one) between these two faces: It does so by turning in counterclockwise order around $v_0$, visiting successively 
all faces of the triangulation filling in the slot that are incident to $v_0$
(see Fig.~\ref{fig:downward-path}, where $N=2$). We have just described how the
downward path $\omega$ goes from $f_0$ to a certain triangle $f'\in \mathsf{F}_{r-1}(\t^{(1)}_\infty)$, but we can now continue
the construction by induction until we reach the bottom face. Notice that this downward path is in general not a geodesic for the dual metric.
\begin{figure}[!h]
 \begin{center}
 \includegraphics[width=0.8\linewidth]{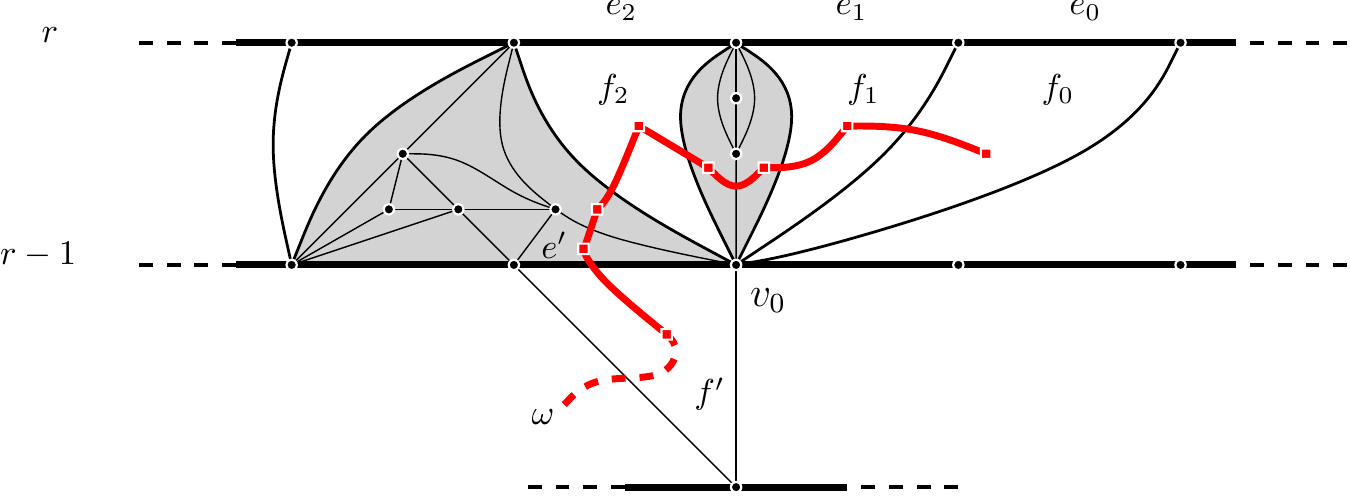}
 \caption{\label{fig:downward-path}Illustration of the construction of the downward path (in red).}
 \end{center}
 \end{figure}

Similarly, we can define downward paths in the lower half-plane model $\l$. For every
downward triangle $f$ incident to $\l_0$, there exists such a dual path connecting $f$
to a certain downward triangle $f'$ incident to $\l_r$, for
some $r\geq 1$. This path is constructed in exactly the same way as explained above for the UIPT. 

Notice that the time needed by a downward path to cross 
a slot  is exactly equal to the degree of the  root vertex of the triangulation filling in the slot (see
subsection \ref{sec:skeleton} for the definition of this root vertex). Tail estimates for the latter quantity are given in  Proposition \ref{prop:subexpo}
in Appendix A2 below,  and are used in the next lemma to bound the length of downward paths, first in the easier case of the LHPT. For every $i\in\Z$
and every $r\geq 0$, we write $f_{(i,r)}$ for the unique downward triangle of $\l$
that is incident to the edge between $(i-1,-r)$ and $(i,-r)$.

\begin{lemma}
\label{bound-do-path}
Let $\omega_r$ be the downward path in $\l$ connecting $f_{(0,0)}$ 
to a downward triangle incident to $\l_r$, and write $|\omega_r|$
for the length of $\omega_r$. There exist two constants
$\mu>0$ and $K<\infty$ such that for every integer $r\geq 1$,
$$\E[\exp(\mu |\omega_r|)] \leq K^r.$$
\end{lemma}

\proof By the independence properties of the LHPT, it is enough to consider the case
$r=1$ and to prove that $\E[\exp(\mu |\omega_1|)]<\infty$ for some $\mu>0$. Note that
the path $\omega_1$ connects $f_{(0,0)}$ to $f_{(0,1)}$. By the definition of
downward paths, $\omega_1$ visits successively $f_{(0,0)},f_{(-1,0)},\ldots, f_{(-N,0)}$ and $f_{(0,1)}$,
where the triangles $f_{(0,0)},f_{(-1,0)},\ldots, f_{(-N,0)}$ are incident to $(0,-1)$, but 
$f_{(-N-1,0)}$ is not (in the construction of $\l$, $N+1$
is the first positive integer $i$ such that the tree $\mathscr{T}_{-i}$ is non trivial). Furthermore, the construction of $\l$ shows that, for every $k\geq 0$,
$$\P(N\geq k)= \theta(0)^k.$$
Conditionally on the forest $(\mathscr{T}_i)_{i\in\Z}$ (and in particular on $N$), the slots associated with the downward triangles $f_{(0,0)},f_{(-1,0)},\ldots, f_{(-N+1,0)}$
are filled in by independent Boltzmann triangulations of the $2$-gon, and the slot associated 
with $f_{(-N,0)}$ is filled in by an independent Boltzmann triangulation of the $d+2$-gon, where 
$d$ is the number of children of the root of $\mathscr{T}_{-N-1}$. It then follows that we have
$$|\omega_1|=  \mathcal{D}_{2}^{(0)}+\cdots+ \mathcal{D}_2^{(N-1)} +  \mathcal{D}_{d+2},$$
where, conditionally on $(\mathscr{T}_i)_{i\in\Z}$, the variables $\mathcal{D}_{2}^{(0)}, \ldots ,\mathcal{D}_2^{(N-1)}$ and $ \mathcal{D}_{d+2}$ are independent, 
$\mathcal{D}_{2}^{(0)}, \ldots ,\mathcal{D}_2^{(N-1)}$ are distributed as the degree of the root vertex in a Boltzmann
triangulation of the $2$-gon, and $ \mathcal{D}_{d+2}$ is distributed as the degree of the root vertex in a Boltzmann
triangulation of the $d+2$-gon.  By Proposition \ref{prop:subexpo} we can choose $\beta>0$
small enough and a finite constant $C$ (not depending on $d$) such that $\E[\exp(\beta  \mathcal{D}_2^{(0)})]\leq C$ and $\E[\exp(\beta  \mathcal{D}_{d+2})]\leq C$. Then, 
if $0<\mu\leq \beta$, we have
$$\E[\exp(\mu |\omega_1|)] \leq C(1-\theta(0)) \sum_{k=0}^\infty \theta(0)^k \, \E[\exp(\mu  \mathcal{D}^{(0)}_2)]^k$$
and we get the desired result by choosing $\mu>0$ small enough so that $\theta(0)\,\E[\exp(\mu  \mathcal{D}^{(0)}_2)]<1$. \endproof

We can now state an analog of Proposition \ref{lem:subadditive}.
For every $r\geq 0$, we write $\l^\dagger_r$ for the collection of all downward triangles incident to an edge of $\l_r$
in the lower half-plane model. We assume that the dual graph of $\l$ is equipped with 
the graph distance $\mathrm{d}^\dagger_{\mathrm{gr}}$ and the Eden distance 
$\mathrm{d}^\dagger_{\mathrm{Eden}}$ defined as previously.

\begin{proposition}
\label{subadditive-dual}
There exist two constants $\mathbf{c}_1\geq 2$ and $\mathbf{c}_2>0$ such that
\begin{align*}
r^{-1}\,\mathrm{d}^\dagger_{\mathrm{gr}}(f_{(0,0)},\l^\dagger_r) & \xrightarrow[r\to\infty]{\rm a.s.} \mathbf{c}_1\,,\\
r^{-1}\,\mathrm{d}^\dagger_{\mathrm{Eden}}(f_{(0,0)},\l^\dagger_r) & \xrightarrow[r\to\infty]{\rm a.s.}\mathbf{c}_2\,.
\end{align*}
\end{proposition}

The proof is very similar to that of Proposition \ref{lem:subadditive} and the details are left to
the reader. In order to apply the subadditive ergodic theorem, we need the fact that
$\E[\mathrm{d}^\dagger_{\mathrm{gr}}(f_{(0,0)},\l^\dagger_r) ]<\infty$, which follows from Lemma
\ref{bound-do-path}. The property $\mathbf{c}_1\geq 2$ is obvious since we have
$\mathrm{d}^\dagger_{\mathrm{gr}}(f_{(0,0)},\l^\dagger_r)\geq 2r$. The fact that $\mathbf{c}_2$
is (strictly) positive is not completely obvious but can be verified as follows. Since we are dealing
with triangulations, there are at most $3^{2r}$ distinct injective dual paths of length $2r$ starting from $f_{(0,0)}$ in $\l$.
However, a crude large deviation argument shows that, if $\delta>0$ is small enough, a.s. for
all sufficiently large $r$, none of these
injective dual paths can have a total Eden weight smaller than $\delta r$. It follows that $\mathbf{c}_2\geq \delta>0$.

\subsection{Technical lemmas}
\label{dual-technic}

It will be important to have a good control of the dual distances $\mathrm{d}^\dagger_\mathrm{gr}$ and $\mathrm{d}^\dagger_\mathrm{Eden}$
in terms of the graph distance on the original graph. This is the goal of the two technical lemmas of this section. The first one
deals with the case of the UIPT and the second one with finite triangulations.

For every integer $r\geq 1$, we let $f_r$ be a downward triangle at height $r$ chosen uniformly at random in 
$\mathsf{F}_r(\t^{(1)}_\infty)$.

\begin{lemma}
\label{crude-bd-dual}
In the UIPT of the $1$-gon $\t^{(1)}_\infty$,
there exist positive constants $K,\alpha,\beta$ such that, for every integers $0\leq r<s$,
\begin{align*}
\P\Big(\mathrm{d}^\dagger_\mathrm{gr}(f_s,B^\bullet_r(\t^{(1)}_\infty))>\alpha(s-r)\Big)&\leq K\,e^{-\beta(s-r)},\\
\P\Big(\mathrm{d}^\dagger_\mathrm{Eden}(f_s,B^\bullet_r(\t^{(1)}_\infty))>\alpha(s-r)\Big)&\leq K\,e^{-\beta(s-r)},
\end{align*}
where $B^\bullet_0(\t^{(1)}_\infty)$ should be interpreted as the bottom face.
\end{lemma}

\proof It is enough to treat the case of $\mathrm{d}^\dagger_\mathrm{gr}$. Indeed, by considering the weights along
a geodesic dual path from $f_s$ to $B^\bullet_r(\t^{(1)}_\infty)$, one immediately gets that, for every $\alpha'>\alpha$,
$$\P\Big(\mathrm{d}^\dagger_\mathrm{gr}(f_s,B^\bullet_r(\t^{(1)}_\infty))>\alpha(s-r),\mathrm{d}^\dagger_\mathrm{Eden}(f_s,B^\bullet_r(\t^{(1)}_\infty))>\alpha'(s-r)\Big)$$
is bounded above, for $s-r$ sufficiently large, by $\exp(-\beta'(s-r))$ for some constant $\beta'>0$. 

So we deal with $\mathrm{d}^\dagger_\mathrm{gr}$, and we assume that $r\geq 1$ (the case $r=0$
is exactly similar). Recall the notation introduced before Proposition \ref{prop:compa}: $\wt\f^{(1)}_{r,s}$
is the skeleton of $B^\bullet_s(\t^{(1)}_\infty)\backslash B^\bullet_r(\t^{(1)}_\infty)$ reordered via a random cyclic permutation, and where
the distinguished vertex at height $s-r$ has been ``forgotten''. We 
may assume that $f_s$ is the downward triangle corresponding to the root of the last tree in $\wt\f^{(1)}_{r,s}$. 
In the proof of Proposition \ref{prop:compa}, we observed that, for every $p,q\geq 1$, we have
$$\P\Big(\wt{\mathcal{F}}^{(1)}_{r,s}=\mathcal{F}\,\Big|\,
L_r=p\Big)
=\frac{h(q)}{h(p)}\,\prod_{v\in\mathcal{F}^*} \theta(c_v)
,$$
for every fixed forest $\f\in\F''_{p,q,s-r}$.  Furthermore, the law of 
$L_r$ was obtained in the proof of Lemma \ref{bdlawperi}:
\begin{equation}
\label{law-perime}
\P(L_r=p)
=\frac{h(p)}{h(1)}\,\P_p(Y_r=1)
=\frac{h(p)}{h(1)}\,\frac{p}{(r+1)^3}\,\Big(1-(r+1)^{-2}\Big)^{p-1}.
\end{equation}
It follows that, for every forest $\f\in\F''_{p,q,s-r}$
\begin{equation}
\label{bd-law-for}
\P\Big(\wt{\mathcal{F}}^{(1)}_{r,s}=\mathcal{F}\Big)= \frac{h(q)}{h(1)}\,\frac{p}{(r+1)^3}\,\Big(1-(r+1)^{-2}\Big)^p\,\prod_{v\in\mathcal{F}^*} \theta(c_v)
\leq \frac{C}{\sqrt{q}(r+1)}\,\prod_{v\in\mathcal{F}^*} \theta(c_v),
\end{equation}
for some finite constant $C$ (we use the fact that
$$\frac{p}{(r+1)^2}\,\Big(1-(r+1)^{-2}\Big)^p\leq \frac{p}{(r+1)^2}\,\exp\Big(-\frac{p}{(r+1)^2}\Big)$$
is bounded above by a constant). 

It follows from \eqref{bd-law-for} that the law of $\wt{\mathcal{F}}^{(1)}_{r,s}$ under $\P(\cdot \cap\{L_s=q\})$ is dominated
by $C/(\sqrt{q}(r+1))$ times the law of a forest of $q$ independent Galton--Watson trees with offspring distribution $\theta$
truncated at height $s-r$, and we may restrict the latter law to the event where the truncated forest has height
exactly $s-r$. Now note that the
length of the downward path from $f_s$ to $B^\bullet_r(\t^{(1)}_\infty)$ is determined by the forest $\wt{\mathcal{F}}^{(1)}_{r,s}$ and by the
triangulations with a boundary filling in the slots associated with the vertices of this forest at height strictly less than $s-r$. 
It follows that the law of this length under $\P(\cdot \cap\{L_s=q\})$ is dominated by $C/(\sqrt{q}(r+1))$ times the law 
of the length of the  downward path that one would get by considering a triangulation of the cylinder of height $s-r$ whose 
(cyclically permuted) skeleton is 
a forest of $q$ independent Galton--Watson trees with offspring distribution $\theta$ truncated at height $s-r$ (and we restrict our attention to the event
where the truncated forest has height $s-r$), and whose slots are filled in by independent Boltzmann triangulations. In the latter model, the numbers
of downward triangles that the downward path crosses in each layer are independent variables, and so are the sizes of the slots which are not $2$-gons
crossed in the different layers. These considerations show that the law under 
$\P(\cdot \cap\{L_s=q\})$ of the length of the downward path from $f_s$ to $B^\bullet_r(\t^{(1)}_\infty)$
is dominated by $C/(\sqrt{q}(r+1))$ times the law of the downward path from $f_{(0,0)}$ to $\l_{s-r}$ in the LHPT model. Using Lemma \ref{bound-do-path}, we now get
\begin{align*}
\P(\{\mathrm{d}^\dagger_\mathrm{gr}(f_s,B^\bullet_r(\t^{(1)}_\infty))>\alpha(s-r)\} \cap\{L_s=q\})
&\leq \frac{C}{\sqrt{q}(r+1)}\,\P(|\omega_{s-r}| >\alpha(s-r))\\
&\leq \frac{C}{\sqrt{q}(r+1)}\,e^{-\mu\alpha(s-r)}\,K^{s-r}.
\end{align*}
We can now fix $\alpha>0$ and $\beta'>0$ such that $e^{-\mu\alpha} K\leq e^{-\beta'}$, and we get
\begin{equation}
\label{bd-dual1}
\P(\{\mathrm{d}^\dagger_\mathrm{gr}(f_s,B^\bullet_r(\t^{(1)}_\infty))>\alpha(s-r)\} \cap\{L_s=q\})
\leq \frac{C}{\sqrt{q}(r+1)}\,\exp({-\beta'(s-r)}).
\end{equation}
To complete the argument we need to sum over the possible values of $q$. Using \eqref{bd-dual1} for $q \leq (s-r)s^2$ and \eqref{eq:boundLs} for $q > (s-r)s^2$ we get that
 \begin{eqnarray*}\P(\{\mathrm{d}^\dagger_\mathrm{gr}(f_s,B^\bullet_r(\t^{(1)}_\infty))>\alpha(s-r)\}) & \leq&  \mathbb{P}( L_{s} > (s-r)s^2) + \sum_{q=1}^{(s-r)s^2} \frac{C}{\sqrt{q}(r+1)}\,\exp({-\beta'(s-r)})\\
 & \leq & C_{0} \exp(-(s-r)/5) + 2C \frac{\sqrt{(s-r)s^2}}{r+1} \exp(-\beta'(s-r))\\
 & \leq & C_{0} \exp(-(s-r)/5) + 4C (s-r)^{3/2} \exp(-\beta'(s-r)).  \end{eqnarray*}
 Taking $\beta \in (0, \frac{1}{5} \wedge \beta')$, the last display is bounded above by $C' \exp(-\beta(s-r))$ for some constant $C'>0$.
  \endproof

\begin{remark}
As the proof shows, we can replace $\mathrm{d}^\dagger_\mathrm{gr}(f_s,B^\bullet_r(\t^{(1)}_\infty))$ (resp.
$\mathrm{d}^\dagger_\mathrm{Eden}(f_s,B^\bullet_r(\t^{(1)}_\infty))$)
in the statement of Lemma \ref{crude-bd-dual} by the length (resp. the weight) of the downward path 
from $f_s$ to $B^\bullet_r(\t^{(1)}_\infty)$. The same remark holds for the next corollary. 
\end{remark}

\begin{corollary}
\label{unif-bd-dual}
Let $\alpha$ be as in the preceding lemma, and let $\delta>0$. For every integer $R\geq 1$, let $A_R(\delta)$ be the event where the property
$$\mathrm{d}^\dagger_\mathrm{gr}(f,B^\bullet_r(\t^{(1)}_\infty))\leq \alpha (s-r)$$
holds for every $0\leq r<s\leq R$ such that $s-r\geq \delta R$, for every downward triangle $f$ at height $s$. There exists a constant
$\wt\beta >0$ such that, for every sufficiently large $R$,
$$\P(A_R(\delta))\geq 1-e^{-\wt \beta R}.$$
The same result holds if $\mathrm{d}^\dagger_\mathrm{dual}$ is replaced by $\mathrm{d}^\dagger_\mathrm{Eden}$. 
\end{corollary}

\proof
Consider first fixed values of $r$ and $s$ such that $0\leq r<s\leq R$ and $s-r\geq \delta R$. Let $f_{(1)}$ be uniformly distributed
over $\mathsf{F}_s(\t^{(1)}_\infty)$, and define $f_{(1)},f_{(2)},f_{(3)},\ldots$ as the successive downward triangles at height $s$
visited when moving around $\partial B^\bullet_s(\t^{(1)}_\infty)$ in clockwise order, starting from $f_{(1)}$. For every integer $j\geq 1$, 
$f_{(j)}$ is also uniformly distributed
over $\mathsf{F}_s(\t^{(1)}_\infty)$. By Lemma \ref{crude-bd-dual},
$$\P\Big(\mathrm{d}^\dagger_\mathrm{gr}(f_{(j)},B^\bullet_r(\t^{(1)}_\infty))>\alpha(s-r)\Big)\leq K\,e^{-\beta\delta R}.$$
Using the bound \eqref{eq:boundLs}, it follows that
\begin{align*}
&\P\Big( \mathrm{d}^\dagger_\mathrm{gr}(f,B^\bullet_r(\t^{(1)}_\infty))> \alpha (s-r),\hbox{ for some downward triangle }f
\hbox{ at height }s\Big)\\
&\qquad\leq \sum_{j=1}^{R\,s^2} \P\Big(\mathrm{d}^\dagger_\mathrm{gr}(f_{(j)},B^\bullet_r(\t^{(1)}_\infty))> \alpha (s-r)\Big)
+ \P(L_s > R\,s^2)\\
&\qquad \leq K\,R^3\,\exp(-\beta \delta R) + C_{0}\,\exp(-R/5).
\end{align*}
To get the estimate of the corollary, we only need to sum this bound over the possible values of $r$ and $s$. The proof
for $\mathrm{d}^\dagger_\mathrm{Eden}$ is exactly the same. \endproof

We now turn to finite triangulations. As previously, $ \t_{n}$ denotes a uniform rooted plane triangulation with $n+1$ vertices. The next lemma gives a uniform estimate for the dual (or Eden)
distance on $\mathsf{F}(\t_n)$ in terms of the graph distance on $\mathsf{V}(\t_n)$.
Recall our notation $x\triangleleft f$ meaning that the vertex $x$ is incident
to the face $f$. 

\begin{lemma}
\label{bd-dual-fini}
Let $\alpha$ be as in Lemma \ref{crude-bd-dual}. 
Let $\ve\in(0,1/4)$, and for every integer $n\geq 1$, let $E_n$ be the event where the bound
$$ \mathrm{d}^\dagger_\mathrm{gr}(f,g) \leq \alpha \,\mathrm{d}_\mathrm{gr}(x,y) + n^\ve$$
holds for every $x,y\in \mathsf{V}(\t_n)$ and $f,g\in \mathsf{F}(\t_n)$ such that $x\triangleleft f$
and $y\triangleleft g$. Then
$$\P(E_n) \build{\la}_{n\to\infty}^{} 1.$$
The same result holds if $\mathrm{d}^\dagger_\mathrm{gr}$ is replaced by $\mathrm{d}^\dagger_\mathrm{Eden}$. 
\end{lemma}

\proof Recall the notation of Proposition \ref{key-finitetri}: $ \overline{\mathcal{T}}_{n}^{(1)}$ is a uniform pointed triangulation of the $1$-gon with $n$ inner vertices, whose root vertex is $\rho_{n}$ and the distinguished vertex is denoted by $o_{n}$. To simplify notation, set $d_n=\mathrm{d}_\mathrm{gr}(\rho_{n},o_{n})$. We can make sense
of the hull $B^\bullet_r(\ov\t_n^{(1)})$ provided that $0<r<d_n$. 
Then write $\Theta$ for the set of all triangulations $\mathrm{t}$ that belong to $\C_{1,r}$
for some $r>1$ and are such that there exists a 
face $f$ incident to $\partial^*\mathrm{t}$ whose dual graph distance from the bottom face is strictly greater than 
$\alpha r$. 

Using Lemma \ref{auxlem},
we have
\begin{align*}
\P\Big(d_n>r\,; B^\bullet_r(\ov\t_n^{(1)})\in \Theta\Big)
&\leq \sum_{\mathrm{t}\in \C_{1,r}} \mathbf{1}_{\Theta}(\mathrm{t})\, \P(d_n>r\,;B^\bullet_r(\ov\t_n^{(1)})=\mathrm{t})\\
& \leq \ov c n^{3/2} 
\sum_{\mathrm{t}\in \C_{1,r}} \mathbf{1}_{\Theta}(\mathrm{t})\, \P(B^\bullet_r(\t^{(1)}_\infty)=\mathrm{t})\\
&\leq \ov c n^{3/2} \exp(-\wt \beta r),
\end{align*}
where the last inequality follows from Corollary \ref{unif-bd-dual}. We can sum this bound from $r=\lfloor n^\ve\rfloor$
to $\infty$, to get
$$\E\Bigg[\sum_{r=\lfloor n^\ve\rfloor}^\infty \mathbf{1}_{\{r<d_n\,; B^\bullet_r(\ov\t_n^{(1)})\in \Theta\}}\Bigg] \leq \wt c
\,\exp(-a \,n^\ve),$$
with some other constants $\wt c>0, a>0$. It follows that
\begin{equation}
\label{crude-tec1}
\P\Big( d_n>n^\ve\,; B^\bullet_{d_n-1}(\ov\t_n^{(1)})\in \Theta\Big)\leq  \wt c
\,\exp(-a \, n^\ve).
\end{equation}

Then notice that $o_{n}$ is adjacent to a vertex $v_0$ that belongs to the boundary of $B^\bullet_{d_n-1}(\ov\t_n^{(1)})$
(take $v_0$ on a geodesic from $o_{n}$ to $\rho_{n}$). If $g$ is any face incident to $o_{n}$ and 
if $g'$ is a face incident to an edge between $o_{n}$ and $v_0$, the dual graph distance
distance between $g$ and $g'$ is bounded above by the degree of $o_{n}$ and thus by the maximal vertex degree of $ \mathcal{T}_{n}^{(1)}$, which we denote
by $\mathsf{MD}( \mathcal{T}_{n}^{(1)})$.
For the same reason, the dual graph distance between $g'$ and a downward triangle at height $d_n-1$ incident to $v_0$
is bounded by $\mathsf{MD}( \mathcal{T}_{n}^{(1)})$, and so is the dual graph distance between any face $f$ incident to $\rho_{n}$ 
and the bottom face. The preceding considerations 
and the definition of $\Theta$ show that, on the event $\{d_n>n^\ve\,; B^\bullet_{d_n-1}(\ov\t_n^{(1)})\notin \Theta\}$,
we have $\mathrm{d}^\dagger_\mathrm{gr}(f,g)\leq \alpha d_n + 3 \mathsf{MD}( \mathcal{T}_{n}^{(1)})$, whenever $\rho_{n}\triangleleft f$ 
 and $o_{n}\triangleleft g$. 

Obviously, the same result holds if one considers instead the rooted and pointed plane triangulation  $\overline{ \mathcal{T}}_{n}$ constructed from $ \overline{ \mathcal{T}}_{n}^{(1)}$ via the transformation of Fig.~\ref{fig:transform-root}. Now re-root $\ov\t_n$ at an oriented edge $e_n$ chosen uniformly and independently of $o_{n}$, and write $\ov\t'_n$
for the resulting rooted and pointed planar map. Then, $\ov\t'_n$ has the same distribution as $\ov\t_n$
and writing $\rho'_n$ for the initial vertex of $e_n$, and $d'_n=\mathrm{d}_\mathrm{gr}(\rho'_{n},o_{n})$, we have
from \eqref{crude-tec1}
$$
\P\Big( d'_n>n^\ve\,; \mathrm{d}^\dagger_\mathrm{gr}(f,g)> \alpha d'_n + 3 \mathsf{MD}( \mathcal{T}_{n})
\hbox{ whenever  }\rho'_{n}\triangleleft f \hbox{ and  }o_{n}\triangleleft g\Big)
 \leq \wt c
\,\exp(-a \, n^\ve).
$$
By the same considerations as in the beginning of the proof of Theorem \ref{main-FPP-tri}, this implies
$$\E\Big[\sum_{v,v'\in \mathsf{V}(\t_n)}
\mathbf{1}_{\{\mathrm{d}_\mathrm{gr}(v,v')>n^\ve\}}
\mathbf{1}_{\{\mathrm{d}^\dagger_\mathrm{gr}(f,g)> \alpha \mathrm{d}_\mathrm{gr}(v,v') + 3 \mathsf{MD}( \mathcal{T}_{n})\;{\rm whenever\;}v
\triangleleft f \;{\rm and\;}v'\triangleleft g\}}\Big]
\leq 6(n+1)^2 \, \wt c \,\exp(-a \, n^\ve).
$$
Hence, with probability tending to $1$ as $n\to\infty$, we have the bound
$$\mathrm{d}^\dagger_\mathrm{gr}(f,g)> \alpha \mathrm{d}_\mathrm{gr}(v,v') + 3 \mathsf{MD}( \mathcal{T}_{n})$$
whenever $v,v'\in \mathsf{V}(\t_n)$,  
$\mathrm{d}_\mathrm{gr}(v,v')>n^\ve$ and $v
\triangleleft f$, $v'\triangleleft g$. However,
Lemma \ref{max-degree} in Appendix A2 below shows that 
we can find a constant $A>0$ such that the bound $\mathsf{MD}( \mathcal{T}_{n})\leq A\,\log n$
holds with probability tending to $1$ as $n\to\infty$. Combining this bound with the previous display, we get the desired result, except for the restriction $\mathrm{d}_\mathrm{gr}(v,v')>n^\ve$.
However, if $\mathrm{d}_\mathrm{gr}(v,v')\leq n^\ve$, we can just use the simple bound $\mathrm{d}^\dagger_{\mathrm{gr}}(f,g) \leq  \mathsf{MD}( \mathcal{T}_{n}) (\mathrm{d_{gr}}(v,v')+1)$. 
This completes 
the proof of the result for $\mathrm{d}^\dagger_\mathrm{gr}$. The case of $\mathrm{d}^\dagger_\mathrm{Eden}$
is treated in a similar way, using also the fact that the maximal weight of a dual edge in $\t_n$
can be bounded by $n^\ve$ outside a set of probability tending to $0$ as $n\to\infty$.  \endproof

\subsection{Proof of the theorems about dual distances}

Lemma \ref{crude-bd-dual} and Lemma \ref{bd-dual-fini} provide the technical ingredients that are needed to
extend the arguments of the proofs of Theorems \ref{main-FPP-tri} and \ref{balls-uipt} to the setting
of Theorems \ref{main-tri-dual}
and \ref{balls-uipt-dual}. In the present subsection, we briefly explain the necessary adaptations of the proofs.

Recall the constants $\mathbf{c}_1$ and
$\mathbf{c}_2$ from Proposition \ref{subadditive-dual}. 
Let us state an analog of Propositions \ref{keytech} and \ref{dist-hull}.
Recall that we interpret $B^\bullet_0(\t^{(1)}_\infty)$ as the bottom face.

\begin{proposition}
\label{keytech-dual}
Let $\ve >0$ and $\delta>0$. We can find $\eta\in(0, \frac{1}{2})$ 
such that, for every sufficiently large $n$, the bounds
$$(1- \varepsilon)\mathbf{c}_1\lfloor \eta n\rfloor \leq
\mathrm{d}_{\rm gr}^\dagger(f,B^\bullet_{n-\lfloor \eta n\rfloor}(\t^{(1)}_\infty)) \leq 
(1+ \varepsilon)\mathbf{c}_1\lfloor \eta n\rfloor$$
hold for every $f\in \mathsf{F}_n(\t^{(1)}_\infty)$, with probability at least $1-\delta$. Furthermore,
$$\P\Big((\mathbf{c}_1- \varepsilon)n\leq \mathrm{d}_{\rm gr}^\dagger(B^\bullet_0(\t^{(1)}_\infty),f)\leq (\mathbf{c}_1+ \varepsilon)n,\ 
\hbox{ for every }f\in \mathsf{F}_n(\t^{(1)}_\infty) \Big) \build{\la}_{n\to\infty}^{} 1.$$
The same properties hold when $\mathrm{d}_{\rm gr}^\dagger$ is replaced by
$\mathrm{d}_{\rm Eden}^\dagger$ provided that $\mathbf{c}_1$
is replaced by $\mathbf{c}_2$. 
\end{proposition}

\begin{proof}[Proof (sketch)]
We start with the first assertion.
Fix $n$ and first choose $f$ uniformly at random 
in $\mathsf{F}_n(\t^{(1)}_\infty)$. We argue in a very similar way as in the proof of Proposition \ref{keytech},
using Proposition \ref{subadditive-dual} instead of Proposition \ref{lem:subadditive}, and
noting that Corollary \ref{unif-bd-dual} already gives us the bound $\mathrm{d}_{\rm gr}^\dagger(f,B^\bullet_{n-\lfloor \eta n\rfloor}
(\t^{(1)}_\infty))\leq \alpha\lfloor \eta n\rfloor$
outside a set of small probability. Recall the notation of the proof of
Proposition \ref{keytech}, and, for every $i\in\Z$, write $f^{(n)}_i\in \mathsf{F}_n(\t^{(1)}_\infty)$ for the unique downward triangle 
incident to the edge of $\partial B^\bullet_n(\t^{(1)}_\infty)$ from $u^{(n)}_i$ to $u^{(n)}_{i+1}$. Let $j$ be such that $f=f^{(n)}_j$.
We need to bound the probability that, for
some $i$ with $j-an^2/16\leq i\leq j+an^2/16$, there is a dual path from $f^{(n)}_i$ to $B^\bullet_{n-\lfloor \eta n\rfloor}(\t^{(1)}_\infty)$ with
length smaller than  $4 \alpha \eta n$, which stays in $B^\bullet_{n}(\t^{(1)}_\infty)$ and
exits the region $\g^{(n)}_j(\eta)$ before hitting $B^\bullet_{n-\lfloor \eta n\rfloor}(\t^{(1)}_\infty)$.
However, a simple argument shows that,
if there exists such a dual path, there will also exist a path (in the primal graph) from $u^{(n)}_i$ to $\partial_\ell\g^{(n)}_j(\eta)$ in $B^\bullet_n(\t^{(1)}_\infty)
\backslash B^\bullet_{n-\lfloor \eta n\rfloor}(\t^{(1)}_\infty)$,
with length smaller than $4\alpha \eta n +1$, and we know from the proof of Proposition \ref{keytech} that this cannot occur except on a set of
small probability. To get a similar estimate in the case of $\mathrm{d}_{\rm Eden}^\dagger$, we need an additional ingredient. Precisely, the
same large deviation argument as in the proof of Proposition \ref{subadditive-dual} allows us to verify the existence of
a constant $\gamma>0$ such that, except on a set of probability tending to $0$ as $k\to\infty$, any injective dual path of length $k$ starting from 
$ \mathsf{F}_{n}(\t^{(1)}_\infty)$ will have total Eden weight at least $\gamma k$ (the point is that there are less than $3^k$ such paths with a given
starting face). So, except on a set of probability tending to $0$ as $n\to\infty$, the existence of a dual path (which can be assumed to be injective)
from $f^{(n)}_i$ to $B^\bullet_{n-\lfloor \eta n\rfloor}(\t^{(1)}_\infty)$ with
Eden weight smaller than  $4 (\alpha/\gamma) \eta n$, which stays in $B^\bullet_{n}(\t^{(1)}_\infty)$ and
exits the region $\g^{(n)}_j(\eta)$ before hitting $B^\bullet_{n-\lfloor \eta n\rfloor}(\t^{(1)}_\infty)$, implies that the same dual path has length smaller than $4\alpha\eta n$, and we can use
the first part of the argument. 

When adapting the final part of the proof of Proposition \ref{keytech}, we also need to verify that, for every $\beta>0$, we can find an integer $A$ sufficiently large so that,
except on a set of probability tending to $0$ as $n\to\infty$, any downward triangle at height $n$ is connected 
to one of the downward triangles $f^{(n)}_j$, $0\leq j\leq \lfloor a^{-1}A\rfloor$, by a dual path in $B^\bullet_n(\t^{(1)}_\infty)$ with length (or Eden weight) at most $\beta \mathbf{c_1}\eta n$
($\beta \mathbf{c_2}\eta n$ in the case of the Eden weight). 
To this end, we use again Proposition \ref{densit2}. We observe that if $f=f^{(n)}_j$ and $f'=f^{(n)}_{j'}$ are two downward triangles at height $n$, the fact that the left-most geodesics from $u^{(n)}_j$
and $u^{(n)}_{j'}$ coalesce above height $n'<n$ implies that the same property holds for the downward paths from $f$ and from $f'$. We can then use
the bounds on the lengths of downward paths obtained in subsection \ref{dual-technic}
(see the remark preceding Corollary \ref{unif-bd-dual}) to get the desired control on the length (or Eden weight) of the dual path from $f$ to $f'$
obtained by the concatenation of the respective downward paths from $f$ and $f'$ up to their coalescence time.

The proof of the second assertion of the proposition is similar to that of Proposition \ref{dist-hull}. We now use Corollary \ref{unif-bd-dual} to handle the
``bad'' values of $i$ for which the bound 
$$(1- \varepsilon)\mathbf{c}_1(n_{i}-n_{i+1})  \leq
\mathrm{d}_{\rm gr}^\dagger(f,B^\bullet_{n_{i+1}}(\t^{(1)}_\infty)) \leq 
(1+ \varepsilon)\mathbf{c}_1 (n_{i}-n_{i+1})  , \hbox{ for every }f\in\mathsf{F}_{n_{i}}(\t^{(1)}_\infty)$$
fails (with the notation of the proof of Proposition \ref{dist-hull}). The remaining part of the argument is the same.
\end{proof}

 \begin{proof}[Proof of Theorem \ref{main-tri-dual}] 
 We start by deriving an analog of Proposition \ref{key-finitetri}. Let $f_*$
 stands for the bottom face of $\t_n^{(1)}$. Let $o_n$ be distributed uniformly on 
 $\mathsf{V}(\t_n^{(1)})$ and let $f_{n}$ be a face incident to $o_n$ (which may be fixed in some deterministic manner
 given $o_n$). Then
\begin{equation}
\label{dual-estim}
\P\Big( \big|\mathrm{d}^\dagger_\mathrm{gr}(f_*,f_{n}) 
-\mathbf{c}_1\, \mathrm{d}_\mathrm{gr}(\rho_{n},o_{n})\big| > \ve\,n^{1/4}\Big)
\build{\longrightarrow}_{n\to\infty}^{} 0,
\end{equation}
and similarly if $\mathrm{d}^\dagger_\mathrm{gr}$ is replaced by $\mathrm{d}^\dagger_\mathrm{Eden}$
provided $\mathbf{c}_1$ is replaced by $\mathbf{c}_2$. The proof is essentially the same as that
of Proposition \ref{key-finitetri}, using the same absolute continuity argument (justified by Lemma \ref{auxlem}) but relying now on the second assertion of Proposition \ref{keytech-dual} 
instead of Proposition \ref{dist-hull}.
The only notable modification is at the end of the proof where, on the event $\{\beta_jn^{1/4}< \mathrm{d}_\mathrm{gr}(\rho_{n},o_{n})\leq \gamma_j n^{1/4}\}$, 
we now use Lemma \ref{bd-dual-fini} to get an upper bound on $\mathrm{d}^\dagger_\mathrm{gr}(f_{n},B^\bullet_{\lfloor \alpha_jn^{1/4}\rfloor}(\t_n^{(1)}))$, or
on $\mathrm{d}^\dagger_\mathrm{Eden}(f_{n},B^\bullet_{\lfloor \alpha_jn^{1/4}\rfloor}(\t_n^{(1)}))$

Once \eqref{dual-estim} has been established, we obtain the statement of Theorem \ref{main-tri-dual} via a straightforward adaptation of the proof
of Theorem \ref{main-FPP-tri}. Lemma \ref{bd-dual-fini} is used once again to verify that if we pick independently a
sufficiently large number $N$ of vertices uniformly distributed over $\mathsf{V}(\t_n)$, then, with high probability uniformly in $n$, any face will be within dual (or Eden)
distance at most $\ve n^{1/4}$ of one of the faces incident to these vertices.
\end{proof}

\begin{proof}[Proof of Theorem \ref{balls-uipt-dual}] This proof goes through by exactly the same absolute continuity argument as in the proof of Theorem 
\ref{balls-uipt}, modulo of course the replacement of $ \mathbf{c}_{0}$ by $\mathbf{c}_1$ or $\mathbf{c}_2$.
\end{proof}

\section*{Appendix A1}

In this appendix, we give a precise justification of \eqref{densite-uniform}. To this end, we need to
obtain a refined version of the convergence of rescaled triangulations to the Brownian map \cite{LG11}.
We will verify that this convergence holds in the sense of
 the Gromov--Hausdorff--Prokhorov metric, if the vertex set of the triangulations is equipped with the
 uniform probability measure, and the Brownian map with its canonical volume measure. 
We refer the reader to \cite[Section 6.2]{Mie09} for the definition of the Gromov--Hausdorff--Prokhorov metric.

\begin{theorem} 
\label{convGHP}
Let $ \mathcal{T}_{n}$ be a uniformly distributed rooted plane triangulation with $n+1$ vertices, and
let  $\mathrm{d^n_{gr}}$
denote the graph distance on $\mathsf{V}(\t_n)$.
If $\mu_{n}$ denotes the uniform probability measure on $\mathsf{V}(\t_n)$, we have 
$$ (\mathsf{V}( \mathcal{T}_{n}), 3^{1/4}n^{-1/4} \mathrm{d^n_{gr}}, \mu_{n}) \xrightarrow[n\to\infty]{\rm(d)} ( \mathbf{m}_{\infty}, D^*, \mu),$$ in distribution for the Gromov--Hausdorff--Prokhorov topology, where $( \mathbf{m}_{\infty}, D^*)$ is the Brownian map and $\mu$ is the volume 
measure on $\bm_\infty$.
\end{theorem}

\begin{proof} 
We first recall that the Brownian map  is defined in terms of a random continuous function $D^*$ on $[0,1]\times [0,1]$. The mapping
$(s,t)\mapsto D^*(s,t)$ is a pseudo-metric on $[0,1]$, and if one considers the associated equivalence relation (namely,
$s\sim t$ if and only if $D^*(s,t)=0$), the Brownian map $\bm_\infty$ is the quotient space $[0,1]/\sim$, which is equipped
with the distance induced by $D^*$. We write $\bp$ for the projection from $[0,1]$ onto $\bm_\infty$
and note that $\bp(1)=\bp(0)$. The volume measure $\mu$ on $\bm_\infty$
is just the image of Lebesgue measure on $[0,1]$ under $\bp$. 

Let us recall some ingredients from the proof in \cite[Section 8]{LG11}, to which we refer for more details. It is convenient to consider a 
vertex $o_n$ uniformly distributed over $\mathsf{V}(\t_n)$. Note that it is enough to
prove that the convergence of the theorem holds when $\mathsf{V}( \mathcal{T}_{n})$ is replaced by 
$\mathsf{V}( \mathcal{T}_{n})\backslash\{o_n\}$ and $\mu_n$ is replaced by the uniform probability
measure $\nu_n$ on $\mathsf{V}( \mathcal{T}_{n})\backslash\{o_n\}$.

 Say that an edge of $\t_n$ is special if both ends of this edge are at the same
graph distance from $o_n$ (in particular, loops are special). Define another planar map $\wt\t_n$ by adding a new vertex 
at the ``middle'' of every special edge, and write $\mathsf{V}(\wt\t_n)\supset \mathsf{V}(\t_n)$ for the vertex set of $\wt \t_n$. 
The function $(u,v)\mapsto \mathrm{d}^n_{\mathrm{gr}}(u,v)$ defined on $\mathsf{V}(\t_n)\times\mathsf{V}(\t_n)$ is then extended
to $\mathsf{V}(\wt\t_n)\times\mathsf{V}(\wt\t_n)$ by declaring that the distance $\mathrm{d}^n_{\mathrm{gr}}(u,v)$ between $u\in \mathsf{V}(\wt\t_n)$
and $v\in \mathsf{V}(\wt\t_n)$ is the minimal length of a path from $u$ to $v$, assuming that edges of $\wt\t_n$ that correspond to
non-special edges of $\t_n$
have length $1$ (as usual) whereas edges of $\wt\t_n$ obtained by the splitting of a special edge of $\t_n$
have length $1/2$ (see \cite[Section 8.3]{LG11} for more details). 

According to \cite[Section 8]{LG11}, we can find an integer $k_n\geq n$ (which depends on $\t_n$)
and a mapping $j\mapsto v^n_j$ from $\{0,1,2,\ldots,k_n\}$ onto $\mathsf{V}(\wt\t_n)\backslash \{o_n\}$
(called the white contour sequence in \cite{LG11}), such that 
we have the convergence in distribution 
\begin{equation}
\label{convtoBM}
\Big(3^{1/4}n^{-1/4}\mathrm{d}^n_{\mathrm{gr}}(v^n_{\lfloor k_ns\rfloor},v^n_{\lfloor k_nt\rfloor})\Big)_{s,t\in[0,1]} \build{\la}_{n\to\infty}^{\rm(d)} 
\Big( D^*(s,t)\Big)_{s,t\in[0,1]}
\end{equation}
in the sense of the uniform convergence of continuous functions on $[0,1]^2$.  We refer to (58) and (59) in \cite{LG11} for the convergence \eqref{convtoBM}, which is
indeed a key ingredient of the proof of the convergence of rescaled triangulations to the Brownian map. 
By using the Skorokhod representation theorem, we may and will assume that the triangulations $\t_n$
have been constructed so that the convergence \eqref{convtoBM} holds a.s.

Write $w^n_0,w^n_1,\ldots,w^n_{n-1}$
for the vertices of $\mathsf{V}(\t_n)\backslash\{o_n\}$ listed in their order of appearance in the sequence $v^n_0,v^n_1,\ldots,v^n_{k_n}$. 
For every $j\in\{0,1,\ldots,k_n\}$, let $L^n_j$ be the number of distinct vertices of $\mathsf{V}(\t_n)\backslash\{o_n\}$
in the sequence $v^n_0,\ldots,v^n_j$. By Proposition 8.2 in \cite{LG11}, we have 
\begin{equation}
\label{unifo-repart}
\sup_{0\leq t\leq 1} | n^{-1} L^n_{\lfloor k_nt\rfloor} -t| \build{\la}_{n\to\infty}^{} 0,
\end{equation}
in probability. Also set $\Lambda^n_i=\min\{j\in\{0,1,\ldots,k_n\}: v^n_j= w^n_i\}$, for every $i\in\{0,1,\ldots,n-1\}$.
As a straightforward consequence of \eqref{unifo-repart}, we have 
\begin{equation}
\label{unifo-repart2}
\sup_{0\leq t\leq 1} | k_n^{-1} \Lambda^n_{\lfloor nt\rfloor} -t| \build{\la}_{n\to\infty}^{} 0,
\end{equation}
in probability. Writing $w^n_{\lfloor ns\rfloor}=v^n_{\Lambda^n_{\lfloor ns\rfloor}}$, we can now combine the convergence \eqref{convtoBM} (assumed to hold a.s.) with \eqref{unifo-repart2}
to get
\begin{equation}
\label{convtoBM2}
\sup_{s,t\in[0,1]}\Big| D^*(s,t)-3^{1/4}n^{-1/4}\mathrm{d}^n_{\mathrm{gr}}(w^n_{\lfloor ns\rfloor},w^n_{\lfloor nt\rfloor})\Big| \build{\la}_{n\to\infty}^{}
0,
\end{equation}
in probability. 

Let $\mathcal{R}_n=\{(\bp(s),w^n_{\lfloor ns\rfloor}): s\in [0,1)\}$, which is a compact correspondence between $(\bm_\infty, D^*)$
and $(\mathsf{V}(\t_n)\backslash\{o_n\},\mathrm{d}^n_{\mathrm{gr}})$. Also let $\pi$ be the probability measure on $\bm_\infty\times (\mathsf{V}(\t_n)\backslash\{o_n\})$
defined by
$$\langle \pi,\phi\rangle=\int_0^1 \mathrm{d}s\,\phi(\bp(s),w^n_{\lfloor ns\rfloor}).$$
 Plainly the first and second marginals of $\pi$ are $\mu$ and $\nu_n$ respectively, and moreover $\pi$ is supported
 on $\mathcal{R}_n$ by construction. According to \cite[Proposition 6]{Mie09}, the desired Gromov--Hausdorff--Prokhorov convergence
 will follow if we can check that the distortion of $\mathcal{R}_n$ converges to $0$ in probability as $n\to\infty$. However, this is an 
 immediate consequence of \eqref{convtoBM2}. 
\end{proof}

Let us now explain why \eqref{densite-uniform} follows from Theorem \ref{convGHP}. As in \eqref{densite-uniform}, we consider, for every $n\geq 1$, a 
sequence $(o^n_j)_{j\geq 1}$ of vertices chosen independently uniformly over $\mathsf{V}(\t_n)$. From \cite[Proposition 10]{Mie09}
and the preceding theorem, we get that, for every $k\geq 1$, the random $k$-pointed metric spaces 
$\big( ( \mathsf{V}( \mathcal{T}_{n}), 3^{1/4} n^{-1/4}\mathrm{d_{gr}}), ( o^n_{i})_{1 \leq i \leq k}\big) $ converge in distribution, in the sense
of the $k$-pointed Gromov--Haus\-dorff metric, to $\big( ( \mathbf{m}_{\infty}, D^*), ( \bp(\xi_{i}))_{1 \leq i \leq k}\big)$, where 
$\xi_1,\xi_2,\ldots$ are i.i.d. uniform random variables on $[0,1]$. On the other hand, for any fixed $\delta>0$ and $\ve>0$, we can choose an integer $N$
such that, with probability at least $1-\delta/2$, any point of $\bm_\infty$ lies within distance at most $\ve/2$ from one of the points
$\bp(\xi_{1}),\ldots,\bp(\xi_N)$. Using the preceding convergence of random $k$-pointed metric spaces (with $k=N$), we obtain that 
\eqref{densite-uniform} holds for all sufficiently large $n$. The small values of $n$ can then be handled by taking $N$
even larger if necessary.

\section*{Appendix A2}
This appendix gathers a few estimates about vertex degrees in random triangulations. Although these results will not be surprising to experts of the field, we were not able to locate precise references dealing explicitly with our case of type I triangulations.

\begin{proposition} \label{prop:subexpo} Let $p \geq 1$ and let $ \mathcal{T}^{(p)}$ be a Boltzmann triangulation of the $p$-gon. We denote by $   \mathcal{D}_{p}$ the degree (i.e.~the number of incident half-edges) of the root vertex in $ \mathcal{T}^{(p)}$. There exist two constants $K_{0}$ and $\lambda < 1$ which do not depend on $p$, such that, for every $k\geq 1$, 
$$ \mathbb{P}( \mathcal{D}_{p} \geq k) \leq K_{0} \lambda^k.$$
\end{proposition}
\proof We denote the origin vertex of $ \mathcal{T}^{(p)}$ by $\rho^{(p)}$. The idea is to explore the neighborhood of $\rho^{(p)}$ from left to right using the peeling process and discarding the parts 
that are useless to determine the degree of the root vertex. This is similar to the proof \cite[Lemma 4.2]{AS03} in the case of type II triangulations, but the case of type I triangulations is  trickier. The peeling process of Boltzmann triangulations is studied in detail in \cite{BCK15}, and we will briefly recall the properties that we need. 

We assume that $\t^{(p)}$ is drawn in the plane so that the unbounded face is the bottom face, and the bottom cycle is then
oriented counterclockwise (in agreement with our convention that the bottom face lies on the right of the root edge). We start by revealing the (finite) face incident to the edge of the boundary whose 
terminal vertex is $\rho^{(p)}$. There are several possibilities, which are illustrated in Fig.~\ref{fig:degree} and whose respective probabilities are expressed in terms of the quantities $Z(k)$ given in \eqref{eq:zp}.
\begin{itemize}
\item The revealed triangle has a new vertex in $ \mathcal{T}^{(p)}$. This event happens with a probability equal to $ \frac{Z(p+1)}{12 \sqrt{3} Z(p)}$. In this case, the remaining triangulation, obtained after removing the discovered triangle, is distributed as $ \mathcal{T}^{(p+1)}$.
\item The third vertex of the revealed triangle belongs the boundary of $ \mathcal{T}^{(p)}$ 
and there are $k$ edges of the boundary, for some $k \in\{1,2, \ldots, p-1 \}$, on the path going from the root vertex to this third vertex along the
boundary, in counterclockwise order. This event happens with probability $ \frac{Z(k+1)Z(p-k)}{Z(p)}$. On this event, the removal of the discovered triangle splits the triangulation into two subtriangulations which are distributed respectively as $ \mathcal{T}^{(k+1)}$ and $ \mathcal{T}^{(p-k)}$. For the remaining part of the argument, we need only consider the subtriangulation (distributed as $ \mathcal{T}^{(k+1)}$) whose boundary contains the root vertex (we discard the hatched part in Fig.~\ref{fig:degree}).

\item The third vertex of the revealed triangle is the root vertex $\rho^{(p)}$. On this event, which happens with probability $Z(1)$, the root vertex is incident to two subtriangulations distributed respectively as $ \mathcal{T}^{(1)}$ and $ \mathcal{T}^{(p)}$. We then need to continue the exploration in each of these subtriangulations (we may say that the exploration branches).
\item Finally, when $p=2$, there is a special case: with probability $ Z(2)^{-1}$ the triangulation of the $2$-gon that we obtain is  the edge-triangulation and the exploration process stops.
\end{itemize}

\medskip
\begin{figure}[!h]
 \begin{center}
 \includegraphics[width=0.9\linewidth]{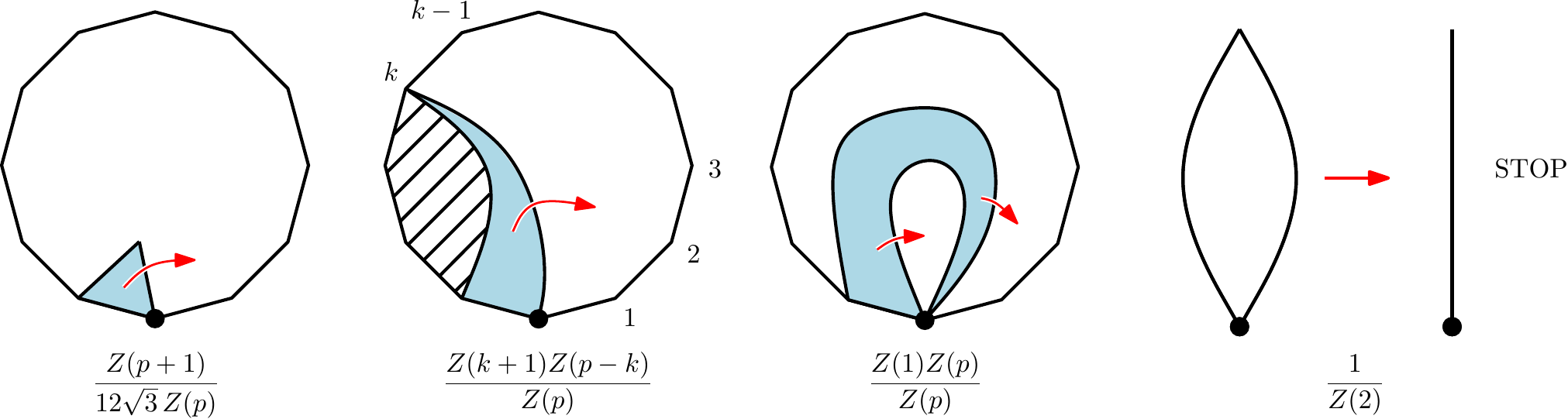}
 \caption{ \label{fig:degree}Discovering the triangle incident to the edge of the boundary whose terminal vertex is the root vertex. In the first two cases on the left, we continue discovering the triangles incident to the root vertex in the ``unknown'' part of the triangulation that is incident to the root vertex. In the third case, we
 need to  continue the exploration in the two unknown parts (they are both incident to the root vertex). In the last case the exploration stops.}
 \end{center}
 \end{figure} 
 
 \vspace{-5mm}

The above exploration allows us to discover the degree of the root vertex. Note that the exploration branches when the peeling of a face splits the triangulation into two subtriangulations that
are both incident to the root vertex. Recording the perimeters of the successive ``subtriangulations'' incident  to the root vertex
that pop up during the exploration leads to a (discrete time) multi-type branching process $ \mathcal{B}$ where the types of the particles are in $\{1,2,3, ... \}$. Furthermore, the branching transitions
are described as follows: for $p \geq 1$ and $ 1 \leq k \leq p-1$,
 \begin{eqnarray} \label{eq:branchingrules} \begin{array}{ll}
\hbox{an individual of type }p \hbox{ has 1 child of type } p+1 & \mbox{ with probab. }  \frac{Z(p+1)}{12 \sqrt{3} Z(p)}\\
\noalign{\smallskip}
\hbox{an individual of type }p \hbox{ has 1 child of type } k+1 & \mbox{ with probab. }  \frac{Z(k+1)Z(p-k)}{Z(p)}\\
\hbox{an individual of type }p \hbox{ has 2 children of respective types }1\hbox{ and }p &\mbox{ with probab. }  Z(1)\\
\hbox{an individual of type }2 \hbox{ has no child} & \mbox{ with probab. }  Z(2)^{-1}.
\end{array}  \end{eqnarray}
The total degree of the root vertex in $ \mathcal{T}^{(p)}$ is bounded above by twice the total number of individuals in the branching process $ \mathcal{B}$ starting from a single particle of type $p \geq 1$. The proof of the proposition is then completed by the following lemma. 

\begin{lemma} Let $ \mathcal{N}_{p}$ be the total number of particles in a multi-type branching process with branching transitions described in \eqref{eq:branchingrules} and started from a single particle of type $p \geq 1$. Then there exist two positive constants $K_{1}>0$ and $0<\wt\lambda <1$, which do not depend on $p$, such that, for all $k \geq 1$,
$$ \mathbb{P}( \mathcal{N}_{p} \geq k ) \leq K_{1} \wt\lambda^k.$$
\end{lemma}

\proof We define another multi-type branching process $ \mathcal{B}'$ with only $3$ types of particles called $ \mathbf{1}, \mathbf{2}$ and $ \mathbf{\overline{3}}$, whose branching 
transitions are described as follows:
$$ \begin{array}{ll}
\hbox{an individual of type }\mathbf{1} \hbox{ has 1 child of type } \mathbf{2} & \mbox{ with probab. } \frac{Z(2)}{12 \sqrt{3}\, Z(1)}= 1-Z(1)\\
\hbox{an individual of type }\mathbf{2} \hbox{ has 1 child of type } \mathbf{\overline{3}} & \mbox{ with probab. } \frac{Z(3)}{12 \sqrt{3}\, Z(2)}\\
\hbox{an individual of type }\mathbf{\overline{3}} \hbox{ has 1 child of type } \mathbf{\overline{3}} & \mbox{ with probab. } 1- Z(1)- \frac{Z(2)}{12}\\
\hbox{an individual of type }\mathbf{2} \hbox{ has 1 child of type } \mathbf{2} & \mbox{ with probab. }  Z(1)\\
\hbox{an individual of type }\mathbf{\overline{3}} \hbox{ has 1 child of type } \mathbf{2} & \mbox{ with probab. }  \frac{Z(2)}{12}\\
\hbox{an individual of type }\mathbf{p} \in \{ \mathbf{1}, \mathbf{2}, \mathbf{\overline{3}}\} \hbox{ has 2 children of types }1\hbox{ and }p& \mbox{ with probab. }  Z(1)\\
\hbox{an individual of type }\mathbf{2}\hbox{ has no child } & \mbox{ with probab. }  Z(2)^{-1}.\\
\end{array}$$
We can interpret these transition probabilities as follows. Starting from either type $\mathbf{1}$ or type $\mathbf{2}$, the transition probabilities are the same as
in \eqref{eq:branchingrules} except that all types $p\geq 3$ are merged into a single type $\overline{\mathbf{3}}$. The probability starting from type $\overline{\mathbf{3}}$ 
of having two children (of types $1$ and $\mathbf{\overline 3}$) is the same as the corresponding probability in \eqref{eq:branchingrules}. The other transitions from type $\overline{\mathbf{3}}$ are designed so that
the following property holds. The probability of the transition $  \mathbf{\overline{3}} \to \mathbf{2}$ in $\mathcal{B}'$ is equal to $Z(2)/12$  and is thus smaller, for any $p\geq 3$, than 
the probability of the transition $p \to 2$ in $\mathcal{B}$, which is equal to 
$$Z(2)\frac{Z(p-1)}{Z(p)} = Z(2)\frac{p}{6(2p-5)}. $$
Note that the probability of the last transition $  \mathbf{\overline{3}} \to \mathbf{\overline{3}}$ is 
set to $1-Z(1)-Z(2)/12$ so that the sum of the transitions starting from $ \mathbf{ \overline{3}}$ is equal to $1$. Using the preceding remarks, it is then easy
to verify that we can couple a branching process $ \mathcal{B}$ starting from a single particle $p$ and a branching process $ \mathcal{B}'$ starting from a single particle of type $ \mathbf{1}$ if $p=1$, of type $ \mathbf{2}$ if $p=2$, and of type $ \mathbf{\overline{3}}$ if $p \geq 3$, so that the total number of particles in $ \mathcal{B}'$ is larger than that in $ \mathcal{B}$.

Now observe that the matrix of the mean offspring numbers of each type in $\mathcal{B}'$ is given by 
$$ \left( \begin{array}{ccc} 
2Z(1) &Z(1) & Z(1)\\
\frac{Z(2)}{12 \sqrt{3}\, Z(1)} & 2 Z(1)& \frac{Z(2)}{12}\\
0 & \frac{Z(3)}{12 \sqrt{3}\, Z(2)} & 1 - \frac{Z(2)}{12}\end{array}
\right)$$
and from the explicit formulas \eqref{eq:zp} one checks that the spectral radius of this matrix is $0.917457... < 1$. It follows by classical results (see \cite[Chapter V]{AN72}) that the total number of particles in $\mathcal{B}$ (starting from any of the three possible types) has an exponential tail. This completes the proof of the proposition. \endproof

Recall our notation $ \t_{n}$ for a uniformly distributed plane triangulation with $n+1$ vertices.
\begin{lemma}
\label{max-degree}
Let $\mathsf{MD}( \mathcal{T}_{n})$ be the maximal degree of a vertex in $\t_n$. There exists $A>0$ such that 
$$\P(\mathsf{MD}( \mathcal{T}_{n})> A \log n)  \build{\la}_{n\to\infty}^{} 0.$$
\end{lemma}
\proof We write $ \mathrm{deg}_{G}(x)$ for the degree of a vertex $x$ in a graph $G$ to avoid confusion.
Let $ \mathcal{T}_{n}^{(1)}$ be a uniform triangulation of the $1$-gon with $n$ inner vertices. One can assume that $ \mathcal{T}_{n}$ is obtained from $ \mathcal{T}_{n}^{(1)}$ 
via the transformation of Fig.~\ref{fig:transform-root} . In particular the root vertex $ \rho_{n}$ of $ \mathcal{T}_{n}$ is also the root vertex of $ \mathcal{T}_{n}^{(1)}$, and we have $$ \mathrm{deg}_{ \mathcal{T}_{n}}(\rho_{n}) \leq \mathrm{deg}_{ \mathcal{T}_{n}^{(1)}}( \rho_{n}).$$
On the other hand, for any $k \geq 1$, if $ \mathcal{T}^{(1)}$ denotes a Boltzmann triangulation of the $1$-gon and 
$\rho^{(1)}$ stands for its root vertex,
$$ \mathbb{P}( \mathrm{deg}_{ \mathcal{T}_{n}^{(1)}}( \rho_{n}) \geq k) = \mathbb{P}(\mathrm{deg}_{ \mathcal{T}^{(1)}}( \rho^{(1)}) \geq k \mid \# {N}( \mathcal{T}^{(1)}) =n) \leq \frac{Z(1)}{(12 \sqrt{3})^{-n}\# \mathbb{T}_{n,1}}\  \mathbb{P}(\mathrm{deg}_{ \mathcal{T}^{(1)}}( \rho^{(1)}) \geq k).$$ 
Using \eqref{eq:asymp} and the case $p=1$ of Proposition \ref{prop:subexpo}, we get for some constants $C>0$ and $\lambda \in (0,1)$,
 $$ \mathbb{P}( \mathrm{deg}_{ \mathcal{T}_{n}}( \rho_{n}) \geq k)\leq  \mathbb{P}(\mathrm{deg}_{ \mathcal{T}_{n}^{(1)}}( \rho_{n})\geq k) \leq C n^{5/2} \lambda ^k.$$ 
We finally use the same argument as in the proof of Theorem \ref{main-FPP-tri} and we get by re-rooting invariance that, for every $k\geq1$,
  \begin{eqnarray*} \mathbb{P}( \exists x \in \mathsf{V}( \mathcal{T}_{n}) : \mathrm{deg}_{ \mathcal{T}_{n}}(x) \geq k) &\leq&  \mathbb{E}\left[ \sum_{x \in \mathsf{V}( \mathcal{T}_{n})}\mathbf{1}_{ \mathrm{deg}_{ \mathcal{T}_{n}}(x) \geq k}\right]\\ 
  &\leq&   \mathbb{E}\left[\sum_{x \in \mathsf{V}( \mathcal{T}_{n})} \mathrm{deg}_{ \mathcal{T}_{n}}(x) \mathbf{1}_{ \mathrm{deg}_{ \mathcal{T}_{n}}(x) \geq k}\right]\\
  & =& 6(n-1)\,\mathbb{P}( \mathrm{deg}_{ \mathcal{T}_{n}}(\rho_{n}) \geq k)\\
  &\leq& 6C n^{7/2} \lambda^k.  \end{eqnarray*}
Applying the last bound to $k=A \log n$ with $A > 4/ |\log \lambda|$ yields the desired result.\endproof

\smallskip
\noindent  \textsc{Nicolas Curien \& Jean-Fran\c cois Le Gall}\\
Laboratoire de Mathématiques d'Orsay, \\
Univ. Paris-Sud, CNRS, Université Paris-Saclay, \\
91405 Orsay, France.

\end{document}